\documentclass[12pt,reqno]{amsart}

\usepackage{amsmath,amssymb,ifthen}
\usepackage{fullpage}
\usepackage{graphicx,psfrag,subfigure}
\usepackage{color}
\usepackage{moreverb}
\usepackage{amsthm}
\usepackage{mathrsfs}
\usepackage{esint}  
\usepackage{graphicx}
\usepackage{psfrag}
\usepackage{hhline}
\usepackage{enumerate}
\usepackage{bbm}

\def\MM{\mathcal M}

\def\UU{\mathcal{U}}
\def\T{\mathbb{T}}
\def\R{{\mathbb R}}

\def\N{{\mathbb N}}

\def\KK{{\mathcal K}}
\def\NN{{\mathcal N}}
\def\HH{{\mathcal H}}
\def\OO{{\mathcal O}}

\def\SS{{\mathcal S}}
\def\TT{{\mathcal T}}
\def\XX{{\mathcal X}}

\def\YY{{\mathcal Y}}
\def\diam{{\rm diam}}
\def\C#1{C_{\rm #1}}
\def\edual#1#2{\langle#1\,,\,#2\rangle_{\LL}}

\def\norm#1#2{\|#1\|_{#2}}

\def\enorm#1{|\hspace*{-.5mm}|#1|\hspace*{-.5mm}|_{\WW}}

\def\set#1#2{\big\{#1\,:\,#2\big\}}
\def\eps{\varepsilon}
\def\dual#1#2{\langle#1\,,\,#2\rangle}

\def\level{{\rm level}}

\def\LL{\mathcal{L}}
\def\div{{\rm div}}

\def\AA{\textit{\textbf{A}}}

\def\BB{\mathcal{B}}

\def\bb{\textit{\textbf b}}

\def\refine{{\tt ref}}
\def\trunc{{\rm trunc}}
\def\Trunc{{\rm Trunc}}
\def\refine{{\tt refine}}
\def\MM{\mathcal M}

\def\RR{\mathcal{R}}

\def\supp{{\rm supp}}
\def\dist{{\rm dist}}
\def\Cdist{C_{\rm dist}}

\def\loc{\mathrm{loc}}
\def\proj{\mathrm{proj}}

\def\Ctrace{C_{\rm trace}}
\def\Cshape{C_{\rm shape}}
\def\Cpatch{C_{\rm patch}}

\def\Cell{C_{\rm ell}}
\def\Cclos{C_{\rm clos}}

\def\Cson{C_{\rm son}}

\def\Cinv{C_{\rm inv}}

\newcounter{constantsnumber}
\def\namec#1#2{%
  \ifthenelse{\equal{#1}{lipschitz}}{C_{\rm lip}}{%
  \ifthenelse{\equal{#1}{c:unifEquivLevel}}{C_{\rm level}}{%
  \ifthenelse{\equal{#1}{mark}}{C_{\rm mark}}{%
  \ifthenelse{\equal{#1}{basis}}{C_{\rm basis}}{%
  \ifthenelse{\equal{#1}{monotone}}{C_{\rm mon}}{%
  \ifthenelse{\equal{#1}{cea}}{C_{\mbox{\rm\scriptsize C\'ea}}}{%
  \ifthenelse{\equal{#1}{norm}}{C_{\rm norm}}{%
  \ifthenelse{\equal{#1}{mon}}{{C}_{\rm mon}}{
  \ifthenelse{\equal{#1}{lip}}{{C}_{\rm lip}}{
  \ifthenelse{\equal{#1}{monA}}{c_{\rm mon}}{
  \ifthenelse{\equal{#1}{lipA}}{c_{\rm lip}}{
  \ifthenelse{\equal{#1}{normequiv1}}{c_{\rm norm}}{ 
  \ifthenelse{\equal{#1}{inv}}{C_{\rm inv}}{ 
  \ifthenelse{\equal{#1}{inv2}}{\widetilde{C}_{\rm inv}}{ 
  \ifthenelse{\equal{#2}{newcounter}}{\refstepcounter{constantsnumber}\label{const#1}}{}C_{\ref{const#1}}}%
  }}}}}}}}}}}}}}

\numberwithin{equation}{section}
\numberwithin{figure}{section}
\newtheorem{theorem}{Theorem}[section]
\newtheorem{proposition}[theorem]{Proposition}
\newtheorem{lemma}[theorem]{Lemma}
\newtheorem{corollary}[theorem]{Corollary}
\newtheorem{algorithm}[theorem]{Algorithm}

\newtheorem{remark}[theorem]{Remark}

\def\subsection#1
{
 \bigskip

 \refstepcounter{subsection}
 {\noindent\bf\arabic{section}.\arabic{subsection}.~#1.~}
}
\renewcommand{\subsection}[1]{\refstepcounter{subsection}\medskip{\bf\thesubsection.~#1.}}

\newcommand*\patchAmsMathEnvironmentForLineno[1]{%
  \expandafter\let\csname old#1\expandafter\endcsname\csname #1\endcsname
  \expandafter\let\csname oldend#1\expandafter\endcsname\csname end#1\endcsname
  \renewenvironment{#1}%
     {\linenomath\csname old#1\endcsname}%
     {\csname oldend#1\endcsname\endlinenomath}}%
\newcommand*\patchBothAmsMathEnvironmentsForLineno[1]{%
  \patchAmsMathEnvironmentForLineno{#1}%
  \patchAmsMathEnvironmentForLineno{#1*}}%
\AtBeginDocument{%
\patchBothAmsMathEnvironmentsForLineno{equation}%
\patchBothAmsMathEnvironmentsForLineno{align}%
\patchBothAmsMathEnvironmentsForLineno{flalign}%
\patchBothAmsMathEnvironmentsForLineno{alignat}%
\patchBothAmsMathEnvironmentsForLineno{gather}%
\patchBothAmsMathEnvironmentsForLineno{multline}%
}
\usepackage[mathlines]{lineno}

\begin{document}

\title{Adaptive IGAFEM with optimal convergence rates:\\Hierarchical B-splines}

\author{Gregor Gantner}

\author{Daniel Haberlik}

\author{Dirk Praetorius}



\begin{abstract}
We consider an  adaptive algorithm for finite element methods for the isogeometric analysis (IGAFEM) of elliptic (possibly non-symmetric) second-order partial differential equations in arbitrary space dimension $d\ge2$. We employ hierarchical B-splines of arbitrary degree and different order of smoothness. 
We propose a refinement strategy to generate a sequence of locally refined meshes and corresponding discrete solutions. 
Adaptivity is driven by some weighted residual {\sl a~posteriori} error estimator.
We prove linear convergence of the error estimator 
(resp.\ the sum of energy error plus data oscillations)
  with optimal algebraic rates. Numerical experiments underpin the theoretical findings.
\end{abstract}

\keywords{Isogeometric analysis; hierarchical splines; adaptivity.}

\subjclass[2010]{41A15, 65D07, 65N12, 65N30}

\maketitle


\section{Introduction}

\subsection{Adaptivity in isogeometric analysis}
The central idea of isogeometric analysis (IGA) is to use the same ansatz functions for the discretization of the partial differential equation (PDE) as for the representation of the problem geometry in computer aided design (CAD); see~\cite{pioneer,bible,approximation}.
The CAD standard for spline representation in a multivariate setting relies on tensor-product B-splines.
However, to allow for adaptive refinement, several extensions of the B-spline model have recently emerged, e.g.,  analysis-suitable T-splines~\cite{scott,beirao}, hierarchical splines~\cite{juttler,juttler2,vanderzee},  or LR-splines~\cite{lr,dokken}; see also \cite{comparison1,comparison2} for a  comparison of these approaches.
All these concepts have been studied via numerical experiments.
However, so far there exists only little literature concerning the thorough mathematical analysis of adaptive isogeometric finite element methods (IGAFEM).
Indeed, we are only aware of the works~\cite{garau16} which investigates an estimator reduction of an IGAFEM with certain hierarchical splines introduced in~\cite{garau}, as well as ~\cite{bg} which investigates linear  convergence of an IGAFEM with truncated hierarchical B-splines introduced in~\cite{juttler2}.
In the continuation of the latter work \cite{bg},~\cite{morgenstern} studied the corresponding mesh-refinement strategy together with some refinement-related properties for the proof of optimal convergence.
However, the mathematical proof that the adaptive strategy of~\cite{bg} leads to optimal convergence rates, is still missing in the literature.
Moreover, in this case one cannot use hierarchical B-splines for the implementation, but has to use truncated hierarchical B-splines instead.
It is important to note that  the procedure of truncation requires a specific construction that entails complicated supports of the basis functions, which are in general not even connected, and their use may produce an  overhead with an adaptive strategy that cannot be neglected.
For standard FEM with globally continuous piecewise polynomials, adaptivity is  well understood; see, e.g.,~\cite{doerfler,mns00,bdd,stevenson,ckns,ffp14,axioms}  for milestones on convergence and optimal convergence rates.
In the frame of isogeometric boundary element methods (IGABEM), we
also mention our recent work~\cite{resigaconv} which shows linear convergence  with optimal rates for some adaptive isogeometric boundary element method in 2D from~\cite{fgp,resigabem},
where it is however sufficient to use the span of univariate non-uniform rational B-splines (NURBS).

\subsection{Model problem}
On the bounded Lipschitz-domain $\Omega\subset\R^d$ with $d\ge 2$, we consider a general second-order linear elliptic PDE in divergence form with homogenous Dirichlet boundary conditions
\begin{align}\label{eq:problem}
\begin{split}
\mathcal{L}u:=-\div(\AA\nabla u)+\bb\cdot\nabla u+cu&=f\quad \text{in }\Omega,\\
u&=0\quad\text{on }\partial\Omega.
\end{split}
\end{align}
We pose the following regularity assumptions on the coefficients:
$\AA(x)\in\R^{d\times d}_{\mathrm{sym}}$ is a symmetric matrix with $\AA\in W^{1,\infty}(\Omega)$.
The vector $\bb(x)\in\R^d$   and the scalar $c(x)\in\R$ satisfy  $\bb,c\in L^\infty(\Omega)$.
We interpret  $\mathcal{L}$ in its weak form and define the corresponding bilinear form 
\begin{align}
\edual{w}{v}:=\int_\Omega \AA(x)\nabla w(x)\cdot\nabla v(x)+\bb(x)\cdot\nabla w(x) v(x)+c(x)w(x)v(x)\,dx. 
\end{align}
The bilinear form is continuous, i.e., it holds that
\begin{align}
\edual{w}{v}\le \big(\norm{\AA}{L^\infty(\Omega)}+\norm{\bb}{L^\infty(\Omega)}+\norm{c}{L^\infty(\Omega)}\big)\norm{w}{H^1(\Omega)}\norm{v}{H^1(\Omega)}
\text{ for all $v,w\in H^1(\Omega)$.}
\end{align}
Additionally, we suppose ellipticity of $\edual{\cdot}{\cdot}$ on $H^1_0(\Omega)$, i.e.,
\begin{align}\label{eq:ellipticity}
\edual{v}{v}\ge \Cell\norm{v}{H^1(\Omega)}^2\quad\text{for all }v\in H_0^1(\Omega).
\end{align}
Note that~\eqref{eq:ellipticity} is for instance satisfied if $\AA(x)$ is uniformly positive definite and if $\bb\in H({\rm div},\Omega)$ with $-\frac12\,{\rm div}\,\bb(x)+c(x)\ge0$ almost everywhere in $\Omega$.

Overall, the boundary value  problem~\eqref{eq:problem} fits into the setting of the Lax-Milgram theorem and therefore admits a unique solution $u\in H_0^1(\Omega)$ to the weak formulation
\begin{align}
 \edual{u}{v} = \int_\Omega fv\,dx
 \quad\text{for all }v\in H^1_0(\Omega).
\end{align}
Finally, we note that the additional regularity $\AA\in W^{1,\infty}(\Omega)$ (instead of $\AA\in L^{\infty}(\Omega)$) is only required for the well-posedness of the residual {\sl a~posteriori} error estimator, see Section~\ref{sec:abstract setting}.

\subsection{Outline \& Contributions}
The remainder of this work is roughly organized as follows: Section~\ref{sec:abstract setting} provides an abstract framework for adaptive mesh-refinement for conforming FEM for the model problem~\eqref{eq:problem}. Its main result is Theorem~\ref{thm:abstract} which states optimal convergence behavior of some standard adaptive algorithm.
Section~\ref{sec:hierarchical} considers conforming FEM based on hierarchical splines. Its main result is Theorem~\ref{thm:main} which states that hierarchical splines fit into the framework of Section~\ref{sec:abstract setting}. The proofs of Theorem~\ref{thm:abstract} and Theorem~\ref{thm:main} are given in Section~\ref{sec:proof abstract} and Section~\ref{sec:proof}, respectively.
Three numerical experiments in Section~\ref{sec:numerics} underpin the optimal convergence behavior of adaptive IGAFEM with hierarchical splines.

\def\Ceff{C_{\rm eff}}%
\def\Crel{C_{\rm rel}}%
\def\osc{{\rm osc}}%
In more detail, the contribution of Section~\ref{sec:abstract setting} can be paraphrased as follows: We formulate the standard adaptive strategy (Algorithm~\ref{the algorithm}) driven by some residual {\sl a~posteriori} error estimator~\eqref{eq:eta} in the frame of conforming FEM. We formulate three assumptions~\eqref{M:shape}--\eqref{M:trace} for the underlying meshes (Section~\ref{subsec:general mesh}), five assumptions~\eqref{R:sons}--\eqref{R:overlay} on the mesh-refinement (Section~\ref{subsec:general refinement}), and six assumptions~\eqref{S:inverse}--\eqref{S:grad} on the 
FEM spaces (Section~\ref{subsec:ansatz}). 
First, these assumptions are sufficient to guarantee that the error estimator $\eta_\bullet$ associated with the FEM solution $U_\bullet\in\XX_\bullet\subset H^1_0(\Omega)$ is 
efficient and
reliable, i.e., there exist 
$\Ceff$,
$\Crel>0$ such that 
\begin{align}\label{intro:releff}
 \Ceff^{-1}\,\eta_\bullet
 \le \inf_{V_\bullet\in\XX_\bullet}\big(\norm{u-V_\bullet}{H^1(\Omega)} + \osc_\bullet(V_\bullet)\big)
 \le 
 \norm{u-U_\bullet}{H^1(\Omega)} 
 + \osc_\bullet(U_\bullet)
 \le \Crel\,\eta_\bullet.
\end{align} 
where $\osc_\bullet(\cdot)$ denotes certain data oscillation terms.
 Second, Theorem~\ref{thm:abstract} states that Algorithm~\ref{the algorithm} leads to linear convergence with optimal rates in the spirit of~\cite{ckns,axioms}: Let $\eta_\ell$ denote the error estimator in the $\ell$-th step of the adaptive algorithm. Then, there exist $C>0$ and $0<q<1$ such that
\begin{align}\label{intro:linear}
 \eta_{\ell+n} \le C\,q^n\,\eta_\ell
 \quad\text{for all }\ell,n\in\N_0.
\end{align}
Moreover, for sufficiently small marking parameters in Algorithm~\ref{the algorithm}, the estimator 
(resp. the so-called total error $\norm{u-U_\ell}{H^1(\Omega)} + \osc_\ell(U_\ell)$; see~\eqref{intro:releff}) 
decays even with the optimal algebraic convergence rate in the sense of certain nonlinear approximation classes (Section~\ref{sec:axioms}).

In explicit terms, we identify sufficient conditions of the underlying meshes, the local FEM spaces, as well as  the employed (local) mesh-refinement rule which guarantee that the related residual {\sl a~posteriori} error estimator satisfies the \emph{axioms of adaptivity} from~\cite{axioms}, so that 
linear convergence with optimal rates for the standard adaptive algorithm follows.
While we exploit this framework only for IGAFEM with hierarchical splines, we believe that it might also serve as a promising starting point to analyze different technologies for adaptive IGAFEM like (analysis-suitable) $T$-splines or LR-splines, as well as for other conforming discretizations like the virtual element method (VEM) from~\cite{vem}.

Section~\ref{sec:hierarchical} recalls the definition of hierarchical splines from~\cite{juttler}, derives the canonical basis of the hierarchical spline space $\XX_\bullet\subset H^1_0(\Omega)$ with Dirichlet boundary conditions (Section~\ref{section:basis}), formulates an admissibility criterion~\eqref{def:admissible} for hierarchical meshes (Section~\ref{subsec:admissible meshes}), and introduces some local mesh-refinement rule for admissible hierarchical meshes (Section~\ref{subsec:concrete refinement}). 
One crucial observation is that the new mesh-refinement strategy for hierarchical meshes (Algorithm~\ref{alg:refinement}) guarantees that the number of (truncated) hierarchical B-splines on each element as well as the number of active elements contained in the support of each (truncated) hierarchical B-spline is uniformly bounded (Proposition~\ref{prop:bounded number}). If one uses the strategy of~\cite{bg,morgenstern} instead, this property is is not satisfied for hierarchical B-splines, but only for truncated hierarchical B-splines. In general, the latter have a smaller, but also more complicated and not necessarily connected support.

The main result of Section~\ref{sec:hierarchical} and the entire work is Theorem~\ref{thm:main} which states that hierarchical splines together with the proposed local mesh-refinement strategy satisfy all assumptions of Section~\ref{sec:abstract setting}, so that Theorem~\ref{thm:abstract} applies. In particular, our work goes beyond~\cite{bg} in two respects: While~\cite{bg} only proves linear convergence of the adaptive algorithm, we give the first proof of optimal convergence rates for IGAFEM. Moreover,~\cite{bg} adapts the analysis of~\cite{ckns} and is hence restricted to symmetric problems (i.e., $\bb=0$ and $c\ge0$ in~\eqref{eq:problem}). Our analysis exploits the framework of~\cite{axioms} together with some recent ideas from~\cite{ffp14,helmholtz} and also covers the non-symmetric problem~\eqref{eq:problem}.

Technical contributions of general interest include the following: We prove that a hierarchical mesh is admissible if and only if it can be obtained by the mesh-refinement strategy of Algorithm~\ref{alg:refinement} (Proposition~\ref{prop:refineT subset T}). Moreover, admissible meshes also allow a simpler computation of truncated hierarchical B-splines in the sense that truncation simplifies considerably (Proposition~\ref{prop:trunc}). Together with some ideas from~\cite{speleers}, we use this observation to define a Scott-Zhang-type projector $J_\bullet:L^2(\Omega)\to\XX_\bullet$ which is locally $L^2$- and $H^1$-stable and has a first-order approximation property (Section~\ref{subsec:E4.1 true}).

\subsection{General notation}
Throughout, $|\cdot|$ denotes the absolute value of scalars, the Euclidean norm of vectors in $\R^d$, as well as  the $d$-dimensional measure of a set in $\R^d$. 
Moreover, $\#$ denotes the cardinality of a  set as well as the multiplicity of a knot within a given knot vector.
We write $A\lesssim B$ to abbreviate $A \le cB$ with some generic constant $c > 0$ which is clear from the context. Moreover, $A \simeq  B$ abbreviates $A\lesssim B \lesssim A$. Throughout, mesh-related quantities have the same index, e.g., $\XX_\bullet$ is the ansatz space corresponding to the  mesh  $\TT_\bullet$. 
The analogous notation is used for meshes $\TT_\circ$, $\TT_\star$ or $\TT_{\ell}$ etc.
Moreover, we use $\widehat{\cdot}$ to transfer  quantities in the physical domain $\Omega$ to the parameter domain $\widehat \Omega$, e.g., we write $\widehat\T$ for the set of all admissible meshes in the parameter domain instead of $\T$ which denotes the set of all admissible meshes in the physical domain.


\def\coarse{\bullet}
\def\fine{\circ}
\def\qson{q_{\rm son}}
\def\kproj{k_{\proj}}
\def\kloc{k_{\rm loc}}
\def\kapp{k_{\rm app}}
\def\kgrad{k_{\rm grad}}
\def\Crel{C_{\rm rel}}
\def\Clin{C_{\rm lin}}
\def\qlin{q_{\rm lin}}
\def\copt{c_{\rm opt}}
\def\Copt{C_{\rm opt}}
\def\Cmin{C_{\rm min}}
\def\Cstab{C_{\rm stb}}
\def\Cred{C_{\rm red}}
\def\qred{q_{\rm red}}
\def\Cdrel{C_{\rm drel}}
\def\Cref{C_{\rm ref}}
\def\opt{{\rm opt}}
\def\qest{q_{\rm est}}
\def\Cest{C_{\rm est}}
\def\enorm#1{|\!|\!|#1|\!|\!|_\LL}

\section{Axioms of adaptivity (revisited)}
\label{sec:abstract setting}

The aim of this section is to formulate an adaptive algorithm (Algorithm~\ref{the algorithm}) for conforming FEM discretizations of our model problem~\eqref{eq:problem}, where adaptivity is driven by the residual {\sl a~posteriori} error estimator (see \eqref{eq:eta} below).
We identify the crucial properties of the underlying meshes, the mesh-refinement, as well as the finite element spaces which ensure that the residual error estimator fits into the general framework of \cite{axioms} and which hence guarantee optimal convergence behavior of the adaptive algorithm.
The main result of this section is Theorem~\ref{thm:abstract} which is proved in Section~\ref{sec:proof abstract}.

\subsection{Meshes}
\label{subsec:general mesh}
«Throughout, $\TT_\coarse$ is a mesh of $\Omega$ in the following sense:
\begin{itemize}
\item $\TT_\coarse$ is a finite set of compact Lipschitz domains;
\item for all $T,T'\in\TT_\coarse$ with $T\neq T'$, the intersection $T\cap T'$ has measure zero;
\item $\overline\Omega = \bigcup_{T\in\TT_\coarse}{T}$, i.e., $\TT_\coarse$ is a partition of $\Omega$.
\end{itemize}
We suppose that there is a countably infinite set $\T$ of admissible meshes. For  $\TT_\coarse\in\T$ and $\omega\subseteq\overline\Omega$, we define the patches of order $k\in\N_0$ inductively by
\begin{align}
 \pi_\coarse^0(\omega) := \omega,
 \quad 
 \pi_\coarse^k(\omega) := \bigcup\set{T\in\TT_\coarse}{ {T}\cap \pi_\coarse^{k-1}(\omega)\neq \emptyset}.
\end{align}
The corresponding set of elements is
\begin{align}
 \Pi_\coarse^k(\omega) := \set{T\in\TT_\coarse}{ {T} \subseteq \pi_\coarse^k(\omega)},
 \quad\text{i.e.,}\quad
 \pi_\coarse^k(\omega) = \bigcup\Pi_\coarse^k(\omega).
\end{align}
To abbreviate notation, we let $\pi_\coarse(\omega) := \pi_\coarse^1(\omega)$ and $\Pi_\coarse(\omega) := \Pi_\coarse^1(\omega)$.
For $\SS_\coarse\subseteq\TT_\coarse$, we define $\pi_\coarse^k(\SS_\coarse):=\pi_\coarse^k(\bigcup\SS_\coarse)$
and $\Pi_\coarse^k(\SS_\coarse):=\Pi_\coarse^k(\bigcup\SS_\coarse)$. 

We suppose that there exist $\Cshape,\Cpatch,\Ctrace>0$ such that all meshes $\TT_\coarse\in\T$ satisfy the following three properties~\eqref{M:shape}--\eqref{M:trace}:
\begin{enumerate}
\renewcommand{\theenumi}{M\arabic{enumi}}
\bf\item\rm\label{M:shape}
\textbf{Shape regularity.}
For all $T\in\TT_\coarse$ and all $T'\in\Pi_\coarse(T)$, it holds that $\Cshape^{-1}|T'| \le |T| \le \Cshape\,|T'|$, i.e., neighboring elements have comparable size.
\bf\item\rm\label{M:patch}
\textbf{Bounded element patch.}
For all $T\in\TT_\coarse$, it holds that $\#\Pi_\coarse(T)\le\Cpatch$, 
i.e., the number of elements in a patch is uniformly bounded.
\bf\item\rm\label{M:trace}
\textbf{Trace inequality.}
For all $T\in\TT_\coarse$ and all $v\in H^1(\Omega)$, it holds that
$\norm{v}{L^2(\partial T)}^2\le \Ctrace \big(|T|^{-1/d}\norm{v}{L^2(T)}^2+ |T|^{1/d}\norm{\nabla v}{L^2(T)}^2\big).$
\end{enumerate}

\subsection{Mesh-refinement}
\label{subsec:general refinement}
For $\TT_\coarse\in\T$ and an arbitrary set of marked elements $\MM_\coarse\subseteq\TT_\coarse$, we associate a corresponding refinement $\TT_\fine:=\refine(\TT_\coarse,\MM_\coarse) \in\T$ with $\MM_\coarse\subseteq\TT_\coarse\setminus\TT_\fine$, i.e., at least the marked elements have been refined.
We define $\refine(\TT_\coarse)$ as the set of all $\TT_\fine$ such that there exist  meshes $\TT_{(0)},\dots,\TT_{(J)}$ and marked elements $\MM_{(0)},\dots,\MM_{(J-1)}$ with $\TT_\fine=\TT_{(J)}=\refine(\TT_{(J-1)},\MM_{(J-1)}),\dots,\TT_{(1)}=\refine(\TT_{(0)},\MM_{(0)})$ and $\TT_{(0)}=\TT_\coarse$. 
Here, we formally allow $J=0$, i.e., $\TT_\coarse\in\refine(\TT_\coarse)$.
We assume that there exists a fixed initial mesh $\TT_0\in\T$ with $\T=\refine(\TT_0)$.

We suppose that there exist $\Cson\ge2$ and $0<\qson<1$ such that all meshes $\TT_\coarse\in\T$  satisfy for arbitrary marked elements $\MM_\coarse\subseteq\TT_\coarse$ with corresponding refinement $\TT_\fine:=\refine(\TT_\coarse,\MM_\coarse)$, the following elementary properties~\eqref{R:sons}--\eqref{R:reduction}:
\begin{enumerate}
\renewcommand{\theenumi}{R\arabic{enumi}}
\bf\item\rm\label{R:sons}
\textbf{Bounded number of sons.}
It holds that $\#\TT_\fine \le \Cson\,\#\TT_\coarse$, i.e., one step of refinement leads to a bounded increase of elements.
%
%
\bf\item\rm\label{R:union}
\textbf{Father is union of sons.}
It holds that $T=\bigcup\set{{T'}\in\TT_\fine}{T'\subseteq T}$ for all $T\in\TT_\coarse$, i.e., each element $T$ is the union of its successors.
\bf\item\rm\label{R:reduction}
\textbf{Reduction of sons.}
It holds that $|T'| \le \qson\,|T|$ for all $T\in\TT_\coarse$ and all $T'\in\TT_\fine$ with $T'\subsetneqq T$, i.e., successors are uniformly smaller than their father.
\end{enumerate}
By induction and the definition of $\refine(\TT_\coarse)$, one easily sees that \eqref{R:union}--\eqref{R:reduction} remain valid if $\TT_\fine$ is an arbitrary mesh in $\refine(\TT_\coarse)$.
In particular, \eqref{R:union}--\eqref{R:reduction} imply that each refined element $T\in\TT_\coarse\setminus\TT_\fine$ is split into at least two sons, wherefore 
\begin{align}\label{eq:R:refine}
\#(\TT_\coarse\setminus\TT_\fine)\le \#\TT_\fine-\#\TT_\coarse\quad\text{for all }\TT_\fine\in\refine(\TT_\coarse).
\end{align}
Besides~\eqref{R:sons}--\eqref{R:reduction}, we suppose the following less trivial requirements~\eqref{R:closure}--\eqref{R:overlay} with a generic constant $\Cclos>0$:
\begin{enumerate}
\renewcommand{\theenumi}{R\arabic{enumi}}
\setcounter{enumi}{3}
\bf\item\rm\label{R:closure}
\textbf{Closure estimate.}
If $\MM_\ell\subseteq\TT_\ell$ and $\TT_{\ell+1}=\refine(\TT_\ell,\MM_\ell)$ for all $\ell\in\N_0$, then 
\begin{align*}
\# \TT_L-\#\TT_0\le \Cclos\sum_{\ell=0}^{L-1}\#\MM_\ell\quad\text{for all }L\in\N.
\end{align*}
\bf\item\rm\label{R:overlay}
\textbf{Overlay estimate.}
For all $\TT_\coarse,\TT_\star\in\T$, there exists a common refinement $\TT_\fine\in\refine(\TT_\coarse)\cap\refine(\TT_\star)$ such that 
\begin{align*}
\#\TT_\fine \le \#\TT_\coarse + \#\TT_\star - \#\TT_0.
\end{align*}
\end{enumerate}

\subsection{Finite element space}\label{subsec:ansatz}
With each $\TT_\coarse\in\T$, we associate a finite dimensional space
\begin{align}
\XX_\coarse \subset \set{v\in H^1_0(\Omega)}{v|_T\in H^2(T)\text{ for all }T\in\TT_\coarse}.
\end{align}
Let $U_\coarse\in\XX_\coarse$ be the corresponding Galerkin approximation to the solution $u\in H^1_0(\Omega)$, i.e.,
\begin{align}
 \edual{U_\coarse}{V_\coarse} = \int_\Omega fV_\coarse\,dx
 \quad\text{for all }V_\coarse\in\XX_\coarse.
\end{align}
We note the Galerkin orthogonality
\begin{align}\label{eq:galerkin}
 \edual{u-U_\coarse}{V_\coarse} = 0
 \quad\text{for all }V_\coarse\in\XX_\coarse
\end{align}
as well as the resulting C\'ea-type quasi-optimality
\begin{align}\label{eq:cea}
 \norm{u-U_\coarse}{H^1(\Omega)}
 \le C_{\text{C\'ea}}\min_{V_\coarse\in\XX_\coarse}\norm{u-V_\coarse}{H^1(\Omega)},
 \text{ with }
 C_{\text{C\'ea}} := \textstyle\frac{\norm{\AA}{L^\infty(\Omega)}\!+\!\norm{\bb}{L^\infty(\Omega)}\!+\!\norm{c}{L^\infty(\Omega)}}{\Cell}.
\end{align}%

We suppose that there exist constants $\Cinv>0$ and $\kloc,\kproj \in\N_0$ such that the following properties~\eqref{S:inverse}--\eqref{S:local} hold for all $\TT_\coarse\in\T$:
\begin{enumerate}
\renewcommand{\theenumi}{S\arabic{enumi}}
\bf\item\rm\label{S:inverse}
\textbf{Inverse estimate.}
For  all $i,j\in\{0,1,2\}$ with $j\le i$, all $V_\coarse\in\XX_\coarse$ and all $T\in\TT_\bullet$, it holds that $|T|^{(i-j)/d} \norm{V_\coarse}{H^i(T)}\le \Cinv \, \norm{V_\coarse}{H^j(T)}$.
\bf\item\rm\label{S:nestedness}
\textbf{Refinement guarantees nestedness.}
For  all $\TT_\fine\in\refine(\TT_\coarse)$, it holds that $\XX_\coarse\subseteq\XX_\fine$.
\bf\item\rm\label{S:local}
\textbf{Local domain of definition.}
With $\Pi_\coarse^{\rm loc}:=\Pi_\coarse^{\kloc}$,  $\pi_\coarse^{\rm loc}:=\pi_\coarse^{\kloc}$ and $\pi_\coarse^{\proj} := \pi_\coarse^{\kproj}$, it holds for all $\TT_\fine\in\refine(\TT_\coarse)$ and all
 $T\in\TT_\coarse\setminus \Pi_\coarse^{\rm loc}( \TT_\coarse\setminus\TT_\fine)\subseteq\TT_\coarse\cap\TT_\fine$,  that
$V_\fine|_{\pi_\coarse^{\rm proj}(T)} \in \set{V_\coarse|_{\pi_\coarse^{\rm proj}(T)}}{V_\coarse\in\XX_\coarse}$.
\end{enumerate}

Besides \eqref{S:inverse}--\eqref{S:local}, we suppose that there exist $\C{sz}>0$ as well as $\kapp\in\N_0$ such that for all $\TT_\coarse\in\T$, there exists a Scott-Zhang-type projector $J_\coarse:H^1_0(\Omega)\to\XX_\coarse$ with the following properties~\eqref{S:proj}--\eqref{S:grad}:
\begin{enumerate}
\renewcommand{\theenumi}{S\arabic{enumi}}
\setcounter{enumi}{3}
\bf\item\rm\label{S:proj}
\textbf{Local projection property.}
With $\kproj\in\N_0$ from \eqref{S:local}, let $\pi_\coarse^{\proj} := \pi_\coarse^{\kproj}$.
For all $v\in H^1_0(\Omega)$ and $T\in\TT_\coarse$, it holds that $(J_\coarse v)|_T = v|_T$, if $v|_{\pi_\coarse^\proj(T)} \in \set{V_\coarse|_{\pi_\coarse^{\proj}(T)}}{V_\coarse\in\XX_\coarse}$.
\bf\item\rm\label{S:app}
\textbf{Local $\boldsymbol{L^2}$-approximation property.}
Let $\pi_\coarse^\mathrm{app}:=\pi_\coarse^{\kapp}$.
For all $T\in\TT_\coarse$ and all $v\in H_0^1(\Omega)$, it holds that $\norm{(1-J_\coarse)v}{L^2(T)}\le \C{sz} \,|T|^{1/d}\,\norm{v}{H^1(\pi_\coarse^\mathrm{app}(T))}$.
%
%
\bf\item\rm\label{S:grad}
\textbf{Local $\boldsymbol{H^1}$-stability.}
Let $\pi_\coarse^\mathrm{grad}:=\pi_\coarse^{\kgrad}$. For all $T\in\TT_\coarse$ and $v\in H_0^1(\Omega)$, it holds that $\norm{\nabla J_\coarse v}{L^2(T)}\le \C{sz} \norm{v}{H^1(\pi_\coarse^\mathrm{grad}(T))}$.
\end{enumerate}
%

\subsection{Error estimator}\label{subsec:estimator}
Let $\TT_\coarse\in\T$ and $T_1\in\TT_\coarse$.
For almost every $x\in\partial T_1\cap\Omega$, there exists a unique element $T_2\in\TT_\coarse$ with $x\in T_1\cap T_2$.
We denote the corresponding outer normal vectors by $\nu_1$ resp. $\nu_2$ and define 
the normal jump as 
\begin{align}
 [\AA \nabla U_\coarse\cdot\nu](x)
 = \AA  \nabla U_\coarse|_{T_1}(x)\cdot \nu_1(x)+\AA  \nabla U_\coarse|_{T'}(x)\cdot \nu_2(x).
\end{align}
With this definition, we employ the residual {\sl a~posteriori} error estimator
\begin{subequations}\label{eq:eta}
\begin{align}
 \eta_\coarse := \eta_\coarse(\TT_\coarse)
 \quad\text{with}\quad 
 \eta_\coarse(\SS_\coarse)^2:=\sum_{T\in\SS_\coarse} \eta_\coarse(T)^2
 \text{ for all }\SS_\coarse\subseteq\TT_\coarse,
\end{align}
where, for all $T\in\TT_\coarse$, the local refinement indicators read
\begin{align}
\eta_\coarse(T)^2:=|T|^{2/d} \norm{f+\div\AA\nabla U_\coarse-\bb\cdot\nabla U_\coarse-c U_\coarse}{L^2(T)}^2+|T|^{1/d}\norm{[\AA \nabla U_\coarse\cdot \nu]}{L^2(\partial T\cap \Omega)}^2.
\end{align}
\end{subequations}
We refer, e.g., to the monographs~\cite{ainsworth-oden,verfuerth} for the analysis of the residual {\sl a~posteriori} error estimator~\eqref{eq:eta} in the frame of standard FEM with piecewise polynomials of fixed order.

\begin{remark}\label{rem:C11}
If $\XX_\bullet\subset C^1(\overline\Omega)$, then the jump contributions in~\eqref{eq:eta} vanish and $\eta_\coarse(T)$ consists only of the volume residual; see~\cite{bg} in the frame of IGAFEM.\qed
\end{remark}%

\subsection{Data oscillations}
The definition of the data oscillations corresponding to the residual error estimator \eqref{eq:eta} requires some further notation. 
Let $\mathcal{P}(\Omega)\subset H^1(\Omega)$ be a fixed discrete subspace.
We suppose that there exists  $\C{inv}'$ such that the following property~\eqref{O:inverse} holds for all $\TT_\coarse\in\T$:
\begin{enumerate}
\renewcommand{\theenumi}{O\arabic{enumi}}
\bf\item\rm\label{O:inverse}
\textbf{Inverse estimate in dual norm.}
For all ${W}\in\mathcal{P}(\Omega)$, it holds that $|T|^{-1/d} \norm{{W}}{H^{-1}(T)}\le \C{inv}' \, \norm{{W}}{L^2(T)}$, where $\norm{{W}}{H^{-1}(T)}:= \sup\set{\int_T{W} v\,dx}{w\in H_0^1(T)\wedge \norm{\nabla v}{L^2(T)}=1}$.
\end{enumerate}

Besides \eqref{O:inverse}, we suppose that there exists $\C{lift}>0$  such that for all $\TT_\coarse\in\T$ and all $T,T'\in\TT_\coarse$ with $(d-1)$-dimensional intersection $E:=T\cap T'$,  there exists a lifting   operator $ L_{\coarse,E}:\set{W|_E}{W\in\mathcal{P}(\Omega)} \to H_0^1(T\cup T')$ with the following properties~\eqref{O:dual}--\eqref{O:grad}:
\begin{enumerate}
\renewcommand{\theenumi}{O\arabic{enumi}}
\setcounter{enumi}{1}
\bf\item\rm\label{O:dual}
\textbf{Dual inequality.}
For all ${W}\in\mathcal{P}(\Omega)$, it  holds that $\int_E{W}^2\,dx \le \C{lift} \int_E\,L_{\coarse,E} ({{W|_E}})W\,dx$.
\bf\item\rm\label{O:stab}
\textbf{$\boldsymbol{L^2}$-control.}
For all  ${W} \in \mathcal{P}(\Omega)$, it holds that $\norm{L_{\coarse,E}({W|_E})}{L^2(T\cup T')}^2\le \C{lift} |T\cup T'|^{1/d} \norm{{W}}{L^2(E)}^2$.
\bf\item\rm\label{O:grad}
\textbf{$\boldsymbol{H^1}$-control.}
For all  ${W} \in \mathcal{P}(\Omega)$, it holds that $\norm{\nabla L_{\coarse,E}({W|_E})}{L^2(T\cup T')}^2\le  \C{lift} |T\cup T'|^{-1/d} \norm{{W}}{L^2(E)}^2$.
\end{enumerate}

Let $\TT_\coarse\in\T$.
For $T\in\TT_\coarse$, we define the $L^2$-orthogonal projection $P_{\coarse,T}:L^2(T)\to \set{{W}|_{T}}{{W}\in\mathcal{P}(\Omega)}$.
For an interior edge 
$E\in\mathcal{E}_{\coarse,T}:=\set{T\cap T'}{T'\in\TT_\coarse\wedge {\rm dim}(T\cap T')=d-1}$, we define the $L^2$-orthogonal projection  $P_{\coarse,  E}:L^2(E)\to \set{{W}|_{E}}{{W}\in\mathcal{P}(\Omega)}$.
Note that $\bigcup\mathcal{E}_{\coarse,T}=\overline{\partial T\cap \Omega}$.
For $V_\coarse\in\XX_\coarse$, we define  the corresponding oscillations
\begin{subequations}\label{eq:osc}
\begin{align}
 \osc_\coarse(V_\coarse) := \osc_\coarse(V_\coarse,\TT_\coarse)
 \quad\text{with}\quad 
 \osc_\coarse({{V_\coarse}},\SS_\coarse)^2:=\sum_{T\in\SS_\coarse} \osc_\coarse({{V_\coarse}}, T)^2
 \text{ for all }\SS_\coarse\subseteq\TT_\coarse,
\end{align}
where, for all $T\in\TT_\coarse$, the local oscillations read
\begin{align}
\begin{split}
\osc_\coarse({{V_\coarse}},T)^2&:=|T|^{2/d} \norm{(1-P_{\coarse,T})(f+\div\AA\nabla V_\coarse-\bb\cdot\nabla V_\coarse-c V_\coarse)}{L^2(T)}^2\\&\quad+\sum_{E\in\mathcal{E}_{\coarse,T}}|T|^{1/d}\norm{(1-P_{\coarse,E})[\AA \nabla V_\coarse\cdot \nu]}{L^2(E)}^2.
\end{split}
\end{align}
\end{subequations}
We refer, e.g., to~\cite{nv} for the analysis of oscillations in the frame of standard FEM with piecewise polynomials of fixed order.
\begin{remark}\label{rem:C12}
If $\XX_\bullet\subset C^1(\overline\Omega)$, then the jump contributions in~\eqref{eq:osc} vanish and $\osc_\coarse(V_\coarse,T)$ consists only of the volume oscillations; see~\cite{bg} in the frame of IGAFEM.\qed
\end{remark}%

\subsection{Adaptive algorithm}
We consider the common formulation of an adaptive mesh-refining algorithm; see, e.g.,~\cite[Algorithm~2.2]{axioms}.

\begin{algorithm}
\label{the algorithm}
\textbf{Input:} 
Adaptivity parameter $0<\theta\le1$ and marking constant $\Cmin\ge 1$.\\
\textbf{Loop:} For each $\ell=0,1,2,\dots$, iterate the following steps~{\rm(i)}--{\rm(iv)}:
\begin{itemize}
\item[\rm(i)] Compute Galerkin approximation $U_\ell\in\XX_\ell$.
\item[\rm(ii)] Compute refinement indicators $\eta_\ell({T})$
for all elements ${T}\in\TT_\ell$.
\item[\rm(iii)] Determine a set of marked elements $\MM_\ell\subseteq\TT_\ell$ which has up to the multiplicative constant $\Cmin$  minimal cardinality such that
$ \theta\,\eta_\ell^2 \le \eta_\ell(\MM_\ell)^2$.
\item[\rm(iv)] Generate refined mesh $\TT_{\ell+1}:=\refine(\TT_\ell,\MM_\ell)$. 
\end{itemize}
\textbf{Output:} Refined meshes $\TT_\ell$ and corresponding Galerkin approximations $U_\ell$ with error estimators $\eta_\ell$ for all $\ell \in \N_0$.
\end{algorithm}

\subsection{Main theorem on rate optimal convergence}
\label{sec:axioms}
We define 
\begin{align}
 \T(N):=\set{\TT_\coarse\in\T}{\#\TT_\coarse-\#\TT_0\le N}
 \quad\text{for all }N\in\N_0
\end{align}
and for all $s>0$
\begin{align}
 \norm{u}{\mathbb{A}_s}
 := \sup_{N\in\N_0}\min_{\TT_\coarse\in\T(N)}(N+1)^s\,\eta_\coarse\in[0,\infty],
\end{align}
and 
\begin{align}
 \norm{u}{\mathbb{B}_s}
 := \sup_{N\in\N_0}\Big(\min_{\TT_\coarse\in\T(N)}(N+1)^s\,\inf_{V_\coarse\in\XX_\coarse} \big(\norm{u-V_\coarse}{H^1(\Omega)}+\osc_\coarse(V_\coarse)\big)\Big)\in[0,\infty].
\end{align}

By definition, $\norm{u}{\mathbb{A}_s}<\infty$ (resp. $\norm{u}{\mathbb{B}_s}<\infty$) implies that the error estimator $\eta_\coarse$ (resp. the total error) on the optimal meshes $\TT_\coarse$ decays at least with rate $\OO\big((\#\TT_\coarse)^{-s}\big)$. The following main theorem states that each possible rate $s>0$ is in fact realized by Algorithm~\ref{the algorithm}.
The (sketch of the) proof is given in Section~\ref{sec:proof abstract}.
It is split into eight steps and builds upon the analysis of~\cite{axioms}.

\begin{theorem}\label{thm:abstract}
\begin{itemize}
\item[\rm (i)]
Suppose \eqref{M:patch}--\eqref{M:trace} and \eqref{S:app}--\eqref{S:grad}.
Then, 
the residual error estimator~\eqref{eq:eta} satisfies reliability, i.e., there exists a constant $\Crel>0$ such that
\begin{align}\label{eq:reliable}
 \norm{u-U_\coarse}{H^1(\Omega)}+\osc_\bullet\le \Crel\eta_\coarse\quad\text{for all }\TT_\coarse\in\T.
\end{align}
\item[(ii)]

Suppose \eqref{M:shape}--\eqref{M:trace}, \eqref{S:inverse}, and   \eqref{O:inverse}--\eqref{O:grad}.
Then, 
the residual error estimator satisfies efficiency, i.e., there  exists  a  constant $\C{eff}>0$ such  that
\begin{align}\label{eq:efficient}
\C{eff}^{-1}\eta_\coarse\le \inf_{V_\coarse\in\XX_\coarse}\big( \norm{u-V_\coarse}{H^1(\Omega)}+\osc_\coarse(V_\coarse)\big).
\end{align}
\item[\rm(iii)]
Suppose  \eqref{M:shape}--\eqref{M:trace}, \eqref{R:union}--\eqref{R:reduction}, and \eqref{S:inverse}--\eqref{S:grad}.
Then, for arbitrary $0<\theta\le1$ and $\Cmin\ge1$, there exist constants $\Clin>0$ and $0<\qlin<1$ such that the estimator sequence of Algorithm~\ref{the algorithm} guarantees linear convergence in the sense of
\begin{align}\label{eq:linear}
\eta_{\ell+j}^2\le \Clin\qlin^j\eta_\ell^2\quad\text{for all }j,\ell\in\N_0.
\end{align}
\item[\rm (iv)]
Suppose \eqref{M:shape}--\eqref{M:trace},  \eqref{R:sons}--\eqref{R:overlay}, and \eqref{S:inverse}--\eqref{S:grad}.
Then, there exists a constant $0<\theta_\opt\le1$ such that for all $0<\theta<\theta_\opt$, all $\Cmin\ge1$, and all $s>0$, there exist constants $\copt,\Copt>0$ such that
\begin{align}\label{eq:optimal}
 \copt\norm{u}{\mathbb{A}_s}
 \le \sup_{\ell\in\N_0}{(\# \TT_\ell-\#\TT_0+1)^{s}}\,{\eta_\ell}
 \le \Copt\norm{u}{\mathbb{A}_s},
\end{align}
i.e., the estimator sequence will decay with each possible rate $s>0$. 
\end{itemize}
\noindent All involved constants $\C{rel},\C{eff},\C{lin},q_{\rm lin},\theta_{\rm opt},\copt,$ and $\Copt$ depend only on the assumptions made as well as the coefficients of the differential operator $\LL$ and $\diam(\Omega)$, where $\Clin,\qlin$ depend additionally on $\theta$ and the sequence $(U_\ell)_{\ell\in\N_0}$, and $\copt,\Copt$ depend furthermore on $\Cmin$, and $s>0$.

\end{theorem}

\begin{remark}
If the assumptions of Theorem~\ref{thm:abstract} {\rm (i)--(ii)} are  satisfied, there holds in particular
\begin{align}\label{eq:equivalence}
\C{eff}^{-1}\norm{u}{\mathbb{A}_s}\le\norm{u}{\mathbb{B}_s}\le \C{rel}\norm{u}{\mathbb{A}_s}\quad\text{for all }s>0.
\end{align}
\end{remark}

\begin{remark}
Note that almost minimal cardinality of $\MM_\ell$ in Step {\rm(iii)} of Algorithm \ref{the algorithm} is only required to prove optimal convergence behavior \eqref{eq:optimal}, while linear convergence \eqref{eq:linear} formally allows $\Cmin=\infty$, i.e., it suffices that $\MM_\ell$ satisfies the D\"orfler marking criterion in Step {\rm (iii)}.
We refer to \cite[Section 4.3--4.4]{axioms} for details.\qed
\end{remark}

\begin{remark}
{\rm(a)} If the bilinear form $\edual{\cdot}{\cdot}$ is symmetric, $\Clin$, $\qlin$ as well as $\copt$, $\Copt$ are then independent of $(U_\ell)_{\ell\in\N_0}$; see Remark~\ref{rem:E3} below.

{\rm(b)} If the bilinear form $\edual{\cdot}{\cdot}$ is non-symmetric, there exists an index $\ell_0\in\N_0$ such that the constants $\Clin$, $\qlin$ as well as $\copt$, $\Copt$ are independent of $(U_\ell)_{\ell\in\N_0}$, if \eqref{eq:linear}--\eqref{eq:optimal} are formulated only for $\ell\ge\ell_0$. We refer to the recent work~\cite[Theorem~19]{helmholtz}.\qed
\end{remark}%

\begin{remark}
If $\XX_\bullet\subset C^1(\overline\Omega)$, all jump contributions  vanish; see Remark~\ref{rem:C11} and Remark~\ref{rem:C12}.
In this case, the assumptions \eqref{O:dual}--\eqref{O:grad} are not necessary for the proof of \eqref{eq:efficient}.\qed
\end{remark}

\begin{remark}
{\rm(a)} Let $h_\ell:=\max_{T\in\TT_\ell}|T|^{1/d}$ be the maximal mesh-width. Then, $h_\ell\to0$ as $\ell\to\infty$, ensures that
$\XX_\infty:=\overline{\bigcup_{\ell\in\N_0}\XX_\ell}=H_0^1(\Omega)$.
To see this, recall that~\eqref{S:nestedness} ensures that $\bigcup_{\ell\in\N_0}\XX_\ell$ is a vector space and, in particular, convex.
By Mazur's lemma (see, e.g., \cite[Theorem 3.12]{rudin}),  it is thus sufficient to show that $\bigcup_{\ell\in\N_0}\XX_\ell$ is weakly dense in $H_0^1(\Omega)$.
Let $v\in H_0^1(\Omega)$.
The Banach-Alaoglu theorem (see, e.g., \cite[Theorem 3.15]{rudin}) together with \eqref{M:patch} and \eqref{S:app}--\eqref{S:grad} proves that each subsequence $(J_{\ell_m}v)_{m\in\N_0}$  admits a further subsequence $(J_{\ell_{m_n}}v)_{n\in\N_0}$ which is weakly convergent in $H^1_0(\Omega)$ towards some limit $w\in H^1_0(\Omega)$. The Rellich compactness theorem hence implies $\norm{w-J_{\ell_{m_n}}v}{L^2(\Omega)}\to0$ as $n\to\infty$. On the other hand,
\eqref{S:app} together with \eqref{M:patch}, \eqref{R:sons}--\eqref{R:reduction}, and $h_\ell\to0$ shows that $\norm{v-J_\ell v}{L^2(\Omega)}\lesssim h_\ell\,\norm{v}{H^1(\Omega)}\to0$ as $\ell\to\infty$.
Together with the uniqueness of limits, these two observations conclude $v=w$. Overall, each subsequence $(J_{\ell_{m}}v)_{m\in\N_0}$ of $(J_\ell v)_{\ell\in\N}$ admits a further subsequence $(J_{\ell_{m_n}}v)_{n\in\N_0}$ which converges weakly in $H^1_0(\Omega)$ to $v$. Basic calculus thus yields that $J_\ell v\rightharpoonup v$ weakly in $H^1_0(\Omega)$ as $\ell\to\infty$. This concludes the proof.

{\rm(b)} We note that the latter observation allows to follow the ideas of~\cite{helmholtz} and to show that the adaptive algorithm yields convergence even if the bilinear form $\dual\cdot\cdot_\LL$ is only elliptic up to some compact perturbation, provided that the continuous problem is well-posed. This includes, e.g., adaptive FEM for the Helmhotz equation. For details, the reader is referred  to~\cite{helmholtz}.\qed%
\end{remark}


\section{Hierarchical setting}\label{sec:hierarchical}

In this section, we recall the definition of hierarchical (B-)splines from \cite{juttler} and propose a local mesh-refinement strategy.
The main result of this section is Theorem~\ref{thm:main} which states that hierarchical splines  together with the proposed mesh-refinement strategy fit into the abstract setting of Section~\ref{sec:abstract setting} and are hence covered by Theorem~\ref{thm:abstract}.
The proof of Theorem~\ref{thm:main} is given in Section~\ref{sec:proof}.

\subsection{Nested tensor meshes and splines}\label{subsec:splines}
We define the parameter domain $\widehat{\Omega}:=(0,1)^d$.
Let $p_1,\dots,p_d\ge 1$ be  fixed polynomial degrees with $p:=\max_{i=1,\dots,d}p_i$.
Let $\widehat\KK^0$ be an arbitrary fixed $d$-dimensional vector of $p_i$-open knot vectors with multiplicity smaller or equal to $p_i$ for the interior knots, i.e.,
\begin{align}
\widehat\KK^0=(\widehat\KK^0_1\dots,\widehat\KK^0_d),
\end{align}
where $\widehat\KK^0_i=(t^0_{i,j})_{j=0}^{N^0_i+p_i}$ is a non-decreasing vector in $[0,1]$ such that $t^0_{i,0}=\dots=t^0_{i,p_i}=0$, $t^0_{i,N^0_i}=\dots=t^0_{i,N^0_i+p_i}$, and $\# t^0_{i,j}:=\#\set{k\in\{0,\dots,N_i^0+p_i\}}{t_{i,k}^0=t_{i,j}^0}\le p_i$ for $j=p_i+1,\dots,N^0_i-1$.
For $k\in\N_0$, we recursively define $\widehat\KK^{k+1}$ as the uniform $h$-refinement of $\widehat\KK^{k}$, i.e., it is obtained by inserting the knot $\frac{t_{i,j}^k+t_{i,j+1}^k}{2}$ of multiplicity one in each  knot span $[t^k_{i,j},t^k_{i,j+1}]$ with $t^k_{i,j}\neq t^k_{i,j+1}$.
Let $\widehat\BB^k$ be the corresponding tensor-product B-spline basis, i.e., 
\begin{align}
\widehat\BB^k=\set{\widehat B^k_{j_1,\dots,j_d}}{j_i\in\{1,\dots,N_i^k\}}, 
\end{align}
where for $(s_1,\dots,s_d)\in \R^d$
\begin{align}\label{eq:B-spline}
\widehat B^k_{j_1,\dots,j_d}(s_1,\dots,s_d):=\prod_{i=1}^d \widehat B(s_i|t_{i,j_i-1}^k,\dots,t_{i,j_i+p_i}^k), 
\end{align}
where $\widehat B(\cdot|t_{i,j_i-1}^k,\dots,t_{i,j_i+p_i}^k)$ denotes the one-dimensional B-spline corresponding to the local knot vector $(t_{i,j_i-1}^k,\dots,t_{i,j_i+p_i}^k)$.
It is well known that the function in $\widehat\BB^k$ have support  $\supp(\widehat B^k_{j_1,\dots,j_d})=[t^k_{1,j_1-1},t^k_{1,j_1+p_1}]\times\dots\times [t^k_{d,j_d-1},t^k_{d,j_d+p_d}]\subseteq[0,1]^d$,  form a partition of unity, and   are even locally linearly independent, i.e, for any open set $O\subseteq [0,1]^d$, the restricted B-splines  $\set{\widehat\beta|_{O}}{\widehat\beta\in\widehat\BB^k\wedge \supp(\widehat\beta)\cap O\neq\emptyset}$ are linearly independent. 
Let $\widehat\YY^k:={\rm span}(\widehat\BB^k)$.
This yields a nested sequence of tensor-product $d$-variate spline function spaces $(\widehat\YY^k)_{k\in\N_0}$ that are at least Lipschitz continuous 
\begin{align}\widehat\YY^k\subset\widehat\YY^{k+1}\subset W^{1,\infty}({\widehat\Omega}).\end{align}
In particular,  each $\widehat\beta^k\in\widehat\BB^k$ can be written as linear combination of functions in $\widehat\BB^{k+1}$, i.e., it has a  unique representation of the form 
\begin{align}\label{eq:twoscale1}
\widehat\beta^k=\sum_{\widehat\beta^{k+1}\in\widehat\BB^{k+1}} c_{\widehat\beta^{k+1}}\widehat\beta^{k+1}.
\end{align}

By the knot insertion procedure, one can show that these coefficients satisfy

\begin{align}\label{eq:twoscale}
\sum_{\widehat\beta^{k+1}\in\widehat\BB^{k+1}} c_{\widehat\beta^{k+1}}=1\quad\text{and}\quad c_{\widehat\beta^{k+1}}\ge 0.
\end{align}

 see, e.g., \cite[Section 2.1.3]{variational}.
With
\begin{align}
\widehat T^k_{j_1,\dots,j_d}:=[t^k_{1,j_1-1},t^k_{1,j_1}]\times\dots\times [t^k_{d,j_d-1},t^k_{d,j_d}] \text{ for all } j_i=1,\dots,N^k_i+p_i,\quad i=1,\dots,d,
\end{align}
we define 
 the corresponding  set of all non-trivial closed cells of level $k$ as 
\begin{align}
\widehat\TT^k:=\set{\widehat T^k_{j_1,\dots,j_d}}{|\widehat T^k_{j_1,\dots,j_d}|>0\wedge j_i=1,\dots,N^k_i+p_i\wedge i=1,\dots,d}.
\end{align}
Each function in $\widehat\YY^k$ is a $\widehat\TT^k$ piecewise polynomial, where the smoothness across the boundary of an element $\widehat T$ depends only on the corresponding knot multiplicities.
For a more detailed presentation of tensor-product splines, we refer to, e.g., \cite{boor,schumaker,variational}.

\subsection{Hierarchical meshes and splines in the parameter domain $\bold{\widehat{\Omega}}$}\label{subsec:parameter hsplines}
Meshes $\TT_\coarse$ and corresponding spaces $\XX_\coarse$ are defined through their counterparts on the parameter domain $\widehat\Omega:=(0,1)^d$.
Let $(\widehat\Omega^k_\bullet)_{k\in\N_0}$ be a nested sequence of closed subsets of $\overline{\widehat\Omega}= [0,1]^d$ such that 
\begin{align}\widehat\Omega^0_\bullet=\overline{\widehat{\Omega}} \quad\text{and}\quad \widehat\Omega^k_\bullet\supseteq\widehat\Omega^{k+1}_\bullet.\end{align} We suppose that  
for $k>0$ each $\widehat\Omega_\bullet^k$ is the union of a selection of cells of level $k-1$, i.e., 
\begin{align}
\widehat\Omega_\bullet^{k}=\bigcup\set{\widehat T\in\widehat\TT^{k-1}}{\widehat T\subseteq\widehat\Omega_\bullet^k}.
\end{align}
Moreover, 
we assume the existence of some minimal $M_\bullet>0$ such that $\widehat\Omega_\bullet^{M_\bullet}=\emptyset$.
Then, we define the mesh in the parameter domain
\begin{align}\label{eq:parameter mesh}
\widehat\TT_\bullet:=\bigcup_{k\in\N_0}\set{\widehat T\in\widehat\TT^k }{\widehat T\subseteq \widehat \Omega_\bullet^k\wedge \widehat T\not\subseteq\widehat\Omega_\bullet^{k+1}}.
\end{align}
\begin{figure}[t] 
\psfrag{z}[r][r]{z}
\psfrag{T}[r][r]{T}
\begin{center}
\includegraphics[width=0.4\textwidth]{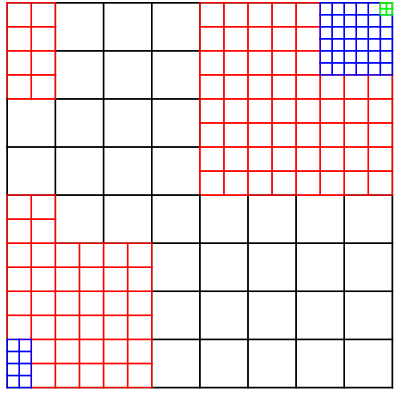}
\end{center}
\caption{
A two-dimensional hierarchical mesh $\widehat\TT_\bullet$ on the parameter domain is depicted, where $\widehat\Omega_\bullet^4=\emptyset$.
The corresponding  domains $\widehat\Omega^0_\bullet\supseteq\widehat\Omega^1_\bullet\supseteq\widehat\Omega^2_\bullet\supseteq\widehat\Omega^3_\bullet$ are highlighted in black, red, blue and green.
}
\label{fig:mesh}
\end{figure}
Note that $\widehat\TT^k\cap\widehat\TT^{k'}=\emptyset$ for $k\neq k'\in\N_0$. 
For $\widehat T\in\widehat\TT_\coarse$,  there exists  a unique  $\level(\widehat T):=k\in\N_0$  
with $\widehat T\subseteq \widehat \Omega_\bullet^k$ and $\widehat T\not\subseteq\widehat\Omega_\bullet^{k+1}$.
Note that $\widehat\TT_\bullet$ is a mesh of $\widehat \Omega$
in the sense of Section~\ref{subsec:general mesh}.

With these preparations, one inductively defines the set of all hierarchical B-splines in the parameter domain $\widehat\HH_\bullet:=\widehat \HH_\coarse^{M_\coarse-1}$ as follows:
\begin{itemize}
\item[(i)]  Define  $\widehat\HH^0_\bullet:=\widehat\BB^0$.
\item[(ii)] For $k=0,\dots,M_\bullet-2$,  define $\widehat\HH_\bullet^{k+1}:={\rm old}(\widehat\HH_\bullet^{k+1})\cup{\rm new}(\widehat\HH_\bullet^{k+1})$, where 
\begin{align}
\begin{split}
{\rm old}(\widehat\HH^{k+1}_\bullet)&:=\set{\widehat\beta\in\widehat\HH_\bullet^k}{\supp(\widehat\beta)\not\subseteq\widehat\Omega^{k+1}_\bullet},\\
{\rm new}(\widehat\HH^{k+1}_\bullet)&:=\set{\widehat\beta\in\widehat\BB^{k+1}}{\supp(\widehat\beta)\subseteq\widehat\Omega_\bullet^{k+1}}.
\end{split}
\end{align}
\end{itemize}
One can prove that the so-called hierarchical basis $\widehat\HH_\bullet$ is linearly independent; see \cite[Lemma~2]{juttler}.
By definition, it holds that
\begin{align}\label{eq:short cHH}
\widehat\HH_\bullet=\bigcup_{k\in\N_0}\set{\widehat\beta\in\widehat\BB^k}{ \supp(\widehat\beta)\subseteq\widehat\Omega_\bullet^k\wedge\supp(\widehat\beta)\not\subseteq\widehat\Omega_\bullet^{k+1}}.
\end{align}
Note that $\widehat\BB^k\cap\widehat\BB^{k'}=\emptyset$ for $k\neq k'\in\N_0$.
For $\widehat\beta\in\widehat\HH_\coarse$, there exists a unique $\level(\widehat\beta):=k\in\N_0$ with $\supp(\widehat\beta)\subseteq\widehat\Omega_\bullet^k$ and $\supp(\widehat\beta)\not\subseteq\widehat\Omega_\bullet^{k+1}$ .

The hierarchical basis  $\widehat\HH_\bullet$ and the mesh $\widehat\TT_\bullet$ are compatible in the following sense:
For all $\widehat\beta\in\widehat\HH_\bullet$, the corresponding support can be written as union of elements in $\widehat\TT^{\level(\widehat\beta)}$, i.e., 
\begin{align}
\supp(\widehat\beta)=\bigcup\set{\widehat T\in\widehat\TT^{\level(\widehat\beta)}}{\widehat T\subseteq\supp(\widehat\beta)}.
\end{align}
Each such element $\widehat T\in\widehat\TT^{\level(\widehat\beta)}$ with $\widehat T\subseteq\supp(\widehat\beta)\subseteq \widehat\Omega_\bullet^{\level(\widehat\beta)}$ satisfies $\widehat T\in\widehat\TT_\bullet$ or $\widehat T\subseteq\widehat\Omega_\bullet^{\level(\widehat\beta)+1}$.
In either case, we see that $\widehat T$ 
can be written as union of elements in $\widehat\TT_\bullet$ with level greater or equal to $\level(\widehat\beta)$.
Altogether, we have 
\begin{align}\label{eq:supp elements}
\supp(\widehat\beta)=\bigcup_{k\ge \level(\widehat\beta)}\set{\widehat T\in\widehat\TT_\bullet\cap\widehat\TT^k}{\widehat T\subseteq\supp(\widehat\beta)}.
\end{align}
Moreover, $\supp(\widehat\beta)$ must contain at least one element of level $\level(\widehat\beta)$, otherwise one would get the contradiction $\supp(\widehat\beta)\subseteq\widehat\Omega_\bullet^{\level(\widehat\beta)+1}$.
In particular, this shows that 
\begin{align}\label{eq:level beta is}
\level(\widehat\beta)=\min_{\widehat T\in\widehat\TT_\bullet\atop \widehat T\subseteq\supp(\widehat\beta)}\level(\widehat T)\quad\text{for all } \widehat\beta\in\widehat\HH_\bullet.
\end{align}

Define  the space of hierarchical splines in the parameter domain by $\widehat\YY_\bullet:={\rm span}(\widehat\HH_\coarse)$.
According to \eqref{eq:supp elements}, each $\widehat V_\bullet\in\widehat\YY_\bullet$ is a $\widehat{\TT}_\bullet$-piecewise tensor polynomial of degree $(p_1,\dots,p_d)$.
We define our ansatz space in the parameter domain as 
\begin{align}
\widehat\XX_\bullet:=\set{\widehat V_\bullet\in\widehat\YY_\bullet}{\widehat V_\bullet|_{\partial\widehat\Omega}=0}\subset\widehat\YY_\coarse\subset \set{\widehat v\in W^{1,\infty}_0({\widehat\Omega})}{\widehat v|_{\widehat T}\in C^2(\widehat T)\text{ for all }\widehat T\in\widehat\TT_\bullet}.
\end{align}
Note that this specifies the abstract setting of Section~\ref{subsec:ansatz}.
 For a more detailed introduction to hierarchical meshes and splines, we refer to, e.g., \cite{juttler,garau,speleers}.

\subsection{Basis of $\widehat\XX_\bullet$}\label{section:basis}
In this section, we characterize a basis of the hierarchical splines $\widehat\XX_\bullet$ that vanish on the boundary.
To this end, we first determine the restriction of the hierarchical basis $\widehat\HH_\bullet$ to a facet of the boundary.
It turns out that  this restriction coincides with the set of $(d-1)$-dimensional hierarchical B-splines.
\begin{proposition}\label{prop:restriction}
Let $\widehat\TT_\bullet$ be an arbitrary hierarchical mesh on the parameter domain $\widehat\Omega$. 
For
 $E=[0,1]^{I-1}\times\{e\}\times{[0,1]^{d-I}}$ with some $I\in\{1,\dots,d\}$ and some  $e\in\{0,1\}$, set 
$\widehat\KK^0|_E:=(\widehat\KK^0_1,\dots,\widehat\KK^0_{I-1}\widehat\KK^0_{I+1},\dots,\widehat\KK^0_{d}),$
 and 
$ \widehat\Omega_\bullet^k|_E:=\set{(s_1,\dots,s_{I-1},s_{I+1},\dots,s_d)}{(s_1,\dots,s_d)\in\widehat\Omega_\bullet^k\cap E}$ for $k\in\N_0$.
 Moreover, let $\widehat\TT_\coarse|_E$ be  the corresponding hierarchical mesh and  $\widehat\HH_\coarse|_E$   the corresponding hierarchical basis.
Then, there holds\footnote{Actually, the set on left-hand side consists of functions defined  on $[0,1]^{d-1}$, whereas the right-hand side functions are defined on $E$.
However, clearly these functions can be identified.
}
$\widehat\HH_\bullet|_E=\set{\widehat\beta|_E}{\widehat\beta\in\widehat\HH_\bullet\wedge \widehat\beta|_E\neq0}$.
Moreover, the restriction $\widehat\HH_\coarse\to\widehat\HH_\coarse|_E$ is essentially injective, i.e.,    for $\widehat\beta_1,\widehat\beta_2\in\widehat\HH_\bullet$ with $\widehat\beta_1\neq\widehat \beta_2$ and $\widehat\beta_1|_E\neq 0$, it follows that $ \widehat\beta_1|_E\neq\widehat\beta_2|_E$.
\end{proposition}
\begin{proof}
We prove the assertion in two steps.

\textbf{Step 1:}
Let $k\in\N_0$.
We recall that the  knot vectors $\widehat\KK^k_i$ are $p_i$-open.
In particular, this implies that the corresponding one-dimensional B-splines  $\widehat\BB^k_i$ are interpolatoric at the end points $e\in\{0,1\}$.
This means that the first resp. last B-spline in $\widehat\BB^k_i$ (i.e., $B(\cdot|t_{i,0}^k,\dots,t_{i,1+p_i}^k)$ resp. $B(\cdot|t_{i,N_i^k-1}^k,\dots,t_{i,N_i^k+p_i}^k)$) is equal to one at  $0$ resp. $1$ and that all other B-splines of $\widehat\BB^k_i$ vanish at these points; see, e.g., \cite[Lemma 2.1]{schimanko}.

%

\textbf{Step 2:}
We consider arbitrary  $d>1$.
For $k\in\N_0$, let $\widehat\BB^k|_E$ be the set of tensor product B-splines  induced by the reduced knots $\widehat\KK^k|_E$ which are defined analogously to $\widehat\KK^0|_E$.
Since $\widehat\KK^k_j$ is $p_j$-open, it holds that $\widehat\BB^k|_E=\set{\widehat \beta|_{E}}{\widehat\beta \in \widehat\BB^k\wedge \widehat \beta|_{E}\neq0}$; see also Step~1.
Then,  the identity \eqref{eq:short cHH} shows
\begin{align}\label{eq:restricted}
\widehat\HH_\bullet|_E=\bigcup_{k\in\N_0}&\Big\{\widehat\beta|_{E}: 
\widehat\beta\in\widehat\BB^k\wedge \widehat\beta|_{E}\neq 0\wedge\supp(\widehat\beta|_{E})\subseteq\widehat\Omega_\bullet^k|_E
\wedge\supp(\widehat\beta|_{E})\not\subseteq\widehat\Omega_\bullet^{k+1}|_E\Big\}.
\end{align}
Let $\widehat\beta\in\widehat\BB^k$ for some $k\in\N_0$ with $\widehat\beta|_{E}\neq 0$.
We set $J:=0$ for $e=0$ resp. $J:=N_I-1$ for $e=1$.
Since $\widehat B(e|t_{I,j_I}^k,\dots,t_{I,j_I+p_I+1}^k)$ does not vanish only if $j_I=J$ (see Step 1), $\widehat\beta$ must be of the form 
\begin{align}\label{eq:restricted spline}
\hspace{-2mm}\widehat\beta(s_1,\dots,s_d)=\prod_{\substack{i=1\\ i\neq I}}^{d} \widehat B(s_i|t_{i,j_i}^k,\dots,t_{i,j_i+p_i+1}^k) \widehat B(s_I|t_{I,J}^k,\dots,t_{I,J+p_I+1}^k)\quad\text{for all }s\in\R^d,
\end{align}
where the second factor is one if $s_I=e$ and   satisfies $\supp(\widehat B(\cdot|t_{I,J}^k,\dots,t_{I,J+p_I+1}^k))=[t_{I,J}^k,t_{I,J+p_I+1}^k]$.
This shows that  $\supp(\widehat\beta)$ is the union of elements  $\widehat T\in\widehat\TT^k$ with non-empty intersection with $E$.
Hence $\supp(\widehat\beta|_{E})\subseteq\widehat\Omega_\bullet^k|_E$ is equivalent to $\supp(\widehat\beta)\subseteq\widehat\Omega_\bullet^k$, and $\supp(\widehat\beta|_{E})\not\subseteq\widehat\Omega_\bullet^{k+1}|_E$ is equivalent to $\supp(\widehat\beta)\not\subseteq\widehat\Omega_\bullet^{k+1}$.
Therefore, \eqref{eq:restricted} becomes 
\begin{align*}
\widehat\HH_\bullet|_E=\bigcup_{k\in\N_0}&\Big\{\widehat\beta|_{E}: 
\widehat\beta\in\widehat\BB^k\wedge \widehat\beta|_{E}\neq 0\wedge\supp(\widehat\beta)\subseteq\widehat\Omega_\bullet^k
\wedge\supp(\widehat\beta)\not\subseteq\widehat\Omega_\bullet^{k+1}\Big\}.
\end{align*}
Together with \eqref{eq:short cHH}, this shows $\widehat\HH_\bullet|_E=\set{\widehat\beta|_E}{\widehat\beta\in\widehat\HH_\bullet\wedge \widehat\beta|_E\neq0}$.
Finally, let $\widehat\beta_1,\widehat\beta_2\in\widehat\HH_\coarse$ with $\widehat\beta_1|_E\neq 0$.
If $\widehat\beta_1|_E=\widehat\beta_2|_E$, then \eqref{eq:restricted spline} already implies $\widehat\beta_1=\widehat\beta_2$.
This concludes the proof.

\end{proof}

\begin{corollary}\label{cor:basis of X}
Let $\widehat\TT_\bullet$ be an arbitrary hierarchical mesh on the parameter domain $\widehat\Omega$. 
Then, $\set{\widehat\beta\in\widehat\HH_\bullet}{\beta|_{\partial\widehat\Omega}=0}$ is a basis of 
$\widehat\XX_\bullet$.
\end{corollary}
\begin{proof}

Linear independence as well as $\set{\widehat\beta\in\widehat\HH_\bullet}{\widehat\beta|_{\partial\widehat\Omega}=0}\subseteq\XX_\bullet$ are obvious. 
To see $\XX_\bullet\subseteq{\rm span}\set{\widehat\beta\in\widehat\HH_\bullet}{\widehat\beta|_{\partial\widehat\Omega}=0}$,
 let $\widehat V_\bullet\in\widehat\XX_\bullet$.
Consider  the unique representation 
$\widehat V_\bullet=\sum_{\widehat\beta\in\widehat\HH_\bullet} c_{\widehat\beta}\widehat\beta$ with  $c_{\widehat\beta}\in\R$.
For arbitrary $\widehat\beta\in\widehat\HH_\bullet$ with $\widehat\beta|_{\partial\widehat\Omega}\neq 0$, 
we have to prove $c_{\widehat\beta}=0$, i.e., we have to show the implication
\begin{align*}
\sum_{\substack{\widehat\beta\in\widehat\HH_\bullet\\\widehat\beta|_{\partial\widehat\Omega}\neq 0}} c_{\widehat\beta}\,\widehat\beta|_{\partial\widehat\Omega}=0\quad\Longrightarrow \quad \Big(\forall \widehat\beta\in\widehat\HH_\bullet\text{ with }\beta|_{\partial\widehat\Omega}\neq 0\quad  c_{\widehat\beta}=0\Big).
\end{align*}
Let $E=[0,1]^{I-1}\times\{e\}\times{[0,1]^{d-I}}$ with $I\in\{1,\dots,d\}$ and  $\sum_{\widehat\beta\in\widehat\HH_\bullet\wedge\widehat\beta|_{E}\neq 0} c_{\widehat\beta}\,\widehat\beta|_E=0$.
According to Proposition \ref{prop:restriction}, the family $\big(\widehat\beta|_E:\widehat\beta\in\widehat\HH_\bullet\wedge\widehat\beta|_{E}\neq 0\big)$ is linearly independent.
Hence,  $\widehat c_{\widehat \beta}=0$ for $\widehat\beta\in\widehat\HH_\coarse$ with $\widehat\beta|_E\neq 0$.
Since $\partial\widehat\Omega$ is the union of such facets $E$, this concludes the proof.
\end{proof}

\subsection{Admissible  meshes in the parameter domain $\bold{\widehat\Omega}$}\label{subsec:admissible meshes}
Let $\widehat\TT_\bullet$ be an arbitrary hierarchical mesh.
We define the set of all neighbors  of an element $\widehat T\in\widehat\TT_\bullet$ as
\begin{align}\label{eq:neigbors}\begin{split}
\NN_\bullet(\widehat T)&:=\set{\widehat T'\in\widehat\TT_\bullet}{\exists\widehat\beta\in\widehat\HH_\bullet\quad\widehat T,\widehat T'\subseteq\supp(\widehat\beta)},
\end{split}
\end{align}
According to  \eqref{eq:supp elements},  the condition   $\widehat T,\widehat T'\subseteq\supp(\widehat\beta)$ is equivalent to $|\widehat T\cap\supp(\widehat\beta)|\neq0\neq |\widehat T'\cap\supp(\widehat \beta)|$.
We call  $\widehat\TT_\bullet$  admissible if 
\begin{align}\label{def:admissible}
|\level(\widehat T)-\level(\widehat T')|\le 1\quad\text{for all }\widehat T,\widehat T'\in\widehat \TT_\bullet \text{ with }\widehat T'\in\NN_\bullet(\widehat T).
\end{align}
Let $\widehat\T$ be the set of all admissible hierarchical meshes in the parameter domain.
Clearly, $\widehat\TT^k\in\widehat\T$ for all $k\in\N_0$.
 Moreover, admissible meshes satisfy the following interesting properties which are also  important for an efficient implementation of IGAFEM with hierarchical splines.
 
\begin{proposition}\label{prop:bounded number}
Let $\widehat \TT_\bullet\in\widehat\T$. 
Then, the support of any basis function $\widehat \beta\in\widehat \HH_\bullet$ is the union  of at most  $2^d (p+1)^d$ elements $\widehat T'\in\widehat \TT_\bullet$.
Moreover, for any $\widehat T\in\widehat \TT_\bullet$, there are at most $2(p+1)^d$ basis functions $\widehat \beta'\in\widehat \HH_\bullet$ that have support on $\widehat T$, i.e., $|\supp(\widehat \beta')\cap \widehat T|>0$.
\end{proposition}
\begin{proof}

We abbreviate $k:=\level(\widehat\beta)$.
By \eqref{eq:level beta is}, there exists $\widehat T''\subseteq \supp(\widehat\beta)$ with $\level(\widehat T'')=k$.
Admissibility of $\widehat\TT_\bullet$ together with \eqref{eq:supp elements}  shows that $\level(\widehat T')\in\{k,k+1\}$ for all $\widehat T'\in\widehat\TT_\bullet$ with $\widehat T'\subseteq \supp(\widehat\beta)$.
Since $\widehat\beta$ is an element of $\widehat\BB^k$, its support is the union of at most $2^d (p+1)^d$ elements in $\widehat\TT^{k+1}$
This proves the first assertion.
For $\widehat\beta'\in\widehat\HH_\coarse$ and $\widehat T\in\widehat\TT_\coarse$ with $|\supp(\widehat\beta')\cap \widehat T|>0$, the characterization  \eqref{eq:supp elements} proves $\widehat T\subseteq\supp(\widehat\beta')$.
Hence, \eqref{eq:level beta is} together with admissibility of $\widehat\T$ proves that $\level(\widehat\beta')=\widetilde k:=\level(\widehat T)$ or $\level(\widehat\beta')=\widetilde k-1$.
With $\widehat \BB^{-1}:=\widehat \BB^0$, 
there are at most $(p+1)^d$ basis functions in $\widehat \BB^{\,\widetilde k-1}$ and $(p+1)^d$ basis functions in $\widehat \BB^{\,\widetilde k}$ that have support on the element $\widehat T$.
This concludes the proof.
\end{proof}

\begin{remark}\label{rem:connected}
Since the support of any $\widehat \beta\in\widehat \HH_\bullet$ is connected, Proposition~\ref{prop:bounded number} particularly shows that $\widehat T' \subseteq\supp(\widehat\beta)$ for an element $\widehat T'\in\widehat\TT_\coarse$ implies that $\supp(\widehat\beta)\subseteq\pi_\coarse^{2(p+1)}(\widehat T')$.
Moreover, we recall that $\widehat T' \subseteq\supp(\widehat\beta)$ is equivalent to $|\widehat T' \cap\supp(\widehat\beta)|>0$; see \eqref{eq:supp elements}.\qed%
\end{remark}

\subsection{Refinement   in the parameter domain $\bold{\widehat\Omega}$}\label{subsec:concrete refinement}
We define the initial mesh $\widehat\TT_0:=\widehat\TT^0$.
Note that $\widehat\TT_0$ is a hierarchical mesh with $\widehat\Omega_0^k=\emptyset$ for all $k>0$.
We say that a hierarchical mesh $\widehat\TT_\circ$ is finer than another hierarchical mesh $\widehat\TT_\bullet$ 
 if  $\widehat\Omega_\bullet^k\subseteq \widehat\Omega_\circ^k$ for all $k\in\N_0$.
This just means that $\widehat\TT_\circ$ is obtained from $\widehat\TT_\bullet$ by iterative dyadic bisections of the elements in $\widehat\TT_\bullet$.
To bisect an element $\widehat T\in\widehat\TT_\bullet$, one just has to add it to the set $\widehat\Omega_\bullet^{\level(\widehat T)+1}$, see \eqref{eq:bisection} below.
In this case, the corresponding spaces are nested, i.e., \begin{align}\label{eq:hierarchical nested}
\widehat\YY_\bullet\subseteq\widehat\YY_\circ\quad\text{and}\quad\widehat\XX_\bullet\subseteq\widehat\XX_\circ.
\end{align}
For a proof, see, e.g., \cite[Corollary 2]{speleers}.
In particular, this implies
 \begin{align}\label{eq:0 contained}
 \widehat\YY^0\subseteq\widehat\YY_\bullet\subseteq\widehat\YY^{M_\bullet-1}.
 \end{align}
 
Next, we present a concrete refinement algorithm to specify the setting of Section~\ref{subsec:general refinement}.
To this end, we first define for $\widehat T\in\widehat\TT_\bullet\in\T$ the set of its bad neighbors 
\begin{align}\label{eq:bad neighbors}
\NN^{\rm bad}_\bullet(\widehat T)&:=\set{\widehat T'\in\NN_\bullet(\widehat T)}{\level(\widehat T')=\level(\widehat T)-1}.
\end{align}
 \begin{algorithm}\label{alg:refinement}
\textbf{Input:} Hierarchical mesh $\widehat\TT_\bullet$ , marked elements $\widehat\MM_\bullet=:\widehat\MM_\bullet^{(0)}\subseteq\widehat\TT_\bullet$.
\begin{itemize}
\item[\rm(i)]  Iterate the following steps {\rm (a)--(b)} for $i=0,1,2,\dots$ until $\widehat\UU_\bullet^{(i)}=\emptyset$:
\begin{itemize}
\item[\rm (a)]Define $\widehat\UU_\bullet^{(i)}:=\bigcup_{\widehat T\in\widehat\MM_\bullet^{(i)}}\set{\widehat T'\in\widehat\TT_\bullet\setminus\widehat\MM_\bullet^{(i)}}{\widehat T'\in\NN^{\rm bad}_\bullet(T)}$.
\item[\rm(b)] 
Define  $\widehat\MM_\bullet^{(i+1)}:=\widehat\MM_\bullet^{(i)}\cup\widehat\UU_\bullet^{(i)}$.
\end{itemize}
\item[\rm(ii)] 
Dyadically bisect all $\widehat T\in\widehat \MM_\bullet^{(i)}$ by adding $\widehat T$ to the set $\widehat \Omega_\bullet^{\level(\widehat T)+1}$ and obtain a finer hierarchical mesh $\widehat\TT_\circ=\refine(\widehat\TT_\bullet,\widehat\MM_\bullet)$, where 
\begin{align}\label{eq:bisection}
\widehat\Omega_\circ^k=\widehat\Omega_\bullet^k\cup\bigcup\set{\widehat T\in\widehat\MM_\bullet^{(i)}}{\level(\widehat T)=k-1}\quad\text{for all }k\in\N.
\end{align}
\end{itemize}
\textbf{Output:} Refined mesh $\widehat\TT_\circ:=\refine(\widehat\TT_\bullet,\widehat\MM_\bullet)$.
\end{algorithm}
Clearly, $\refine(\widehat\TT_\bullet,\widehat\MM_\bullet)$ is finer than $\widehat\TT_\bullet$. 
For any hierarchical mesh $\widehat\TT_\bullet$, we define $\refine(\widehat\TT_\bullet)$ as the set of all hierarchical meshes $\widehat\TT_\circ$ such that there exist hierarchical meshes $\widehat\TT_{(0)},\dots,\widehat\TT_{(J)}$ and marked elements $\widehat\MM_{(0)},\dots,\widehat\MM_{(J-1)}$ with $\widehat\TT_\circ=\widehat\TT_{(J)}=\refine(\widehat\TT_{(J-1)},\widehat\MM_{(J-1)}),\dots,\widehat\TT_{(1)}=\refine(\widehat\TT_{(0)},\widehat\MM_{(0)})$, and $\widehat\TT_{(0)}=\widehat\TT_\bullet$. 
 Here, we formally allow $J=0$, i.e., $\TT_\coarse\in\refine(\TT_\coarse)$.
 Proposition ~\ref{prop:refineT subset T}  below will show that $\widehat\T=\refine(\widehat\TT_0)$, i.e., starting from $\widehat\TT_0=\widehat\TT^0$,  all admissible meshes $\widehat\TT_\coarse$ can be generated by iterative refinement via Algorithm~\ref{alg:refinement}.  

\begin{remark}\label{rem:morgenstern}
\cite{bg,morgenstern} studied a related refinement strategy, where $\NN_\bullet(\widehat T)$ of \eqref{eq:neigbors} and $\NN^{\rm bad}_\bullet(\widehat T)$ from \eqref{eq:bad neighbors} are replaced by 
\begin{align}
\begin{split}
{\widetilde\NN}_\bullet(\widehat T)&:=\set{\widehat T'\in\widehat\TT_\bullet}{\exists\widehat\beta\in\widehat\BB^{\level(\widehat T)}\text{ with }|\widehat T\cap\supp(\widehat \beta)|\neq 0 \neq |\widehat T'\cap\supp(\widehat \beta)|},\\
{\widetilde{\NN}}_\bullet^{\rm bad}(\widehat T)&:=\set{\widehat T'\in{\widetilde \NN}_\bullet}{\level(\widehat T')=\level(\widehat T)-1}.
\end{split}
\end{align}
There, the refinement strategy was designed for truncated hierarchical B-splines; see Section~\ref{subsec:trunc}.
Compared to the hierarchical B-splines $\widehat\HH_\bullet$, those have generically a smaller, but also more complicated  and not necessarily connected support.
\cite[Corollary 17]{bg} shows that the generated meshes are strictly admissible in the sense of \cite{bg,morgenstern}, i.e., for all $k\in\N$,  it holds that
\begin{align}\label{eq:strictly admissible}
\widehat\Omega_\bullet^k\subseteq \bigcup \set{\widehat T\in\widehat \TT^{k-1}}{\forall \widehat\beta\in\widehat\BB^{k-1}\quad\big(\widehat T\subseteq\supp(\widehat\beta)\,\Longrightarrow\, \supp(\widehat\beta)\subseteq\widehat\Omega_\bullet^{k-1}\big)}.
\end{align}
 This definition actually goes back to \cite[Appendix A]{gjs}.
According to \cite[Section~2.4]{bg}, strictly admissible meshes   satisfy 
a similar version of Proposition~\ref{prop:bounded number} for truncated hierarchical B-splines.
However, the example from Figure~\ref{fig:BGbad} shows that the proposition fails for hierarchical B-splines and the refinement strategy from \cite{bg}.
In particular, strictly admissible meshes are not necessarily admissible in the sense of Section~\ref{subsec:admissible meshes}.\qed%
\end{remark}

\begin{figure}[t] 
\psfrag{z}[r][r]{z}
\psfrag{T}[r][r]{T}
\psfrag{Algorithm (?2)}[b][c]{\tiny Algorithm \cite{bg}}
\psfrag{Algorithm (?1)}[b][c]{\tiny Algorithm \ref{alg:refinement}}
\begin{center}
\includegraphics[width=0.2\textwidth]{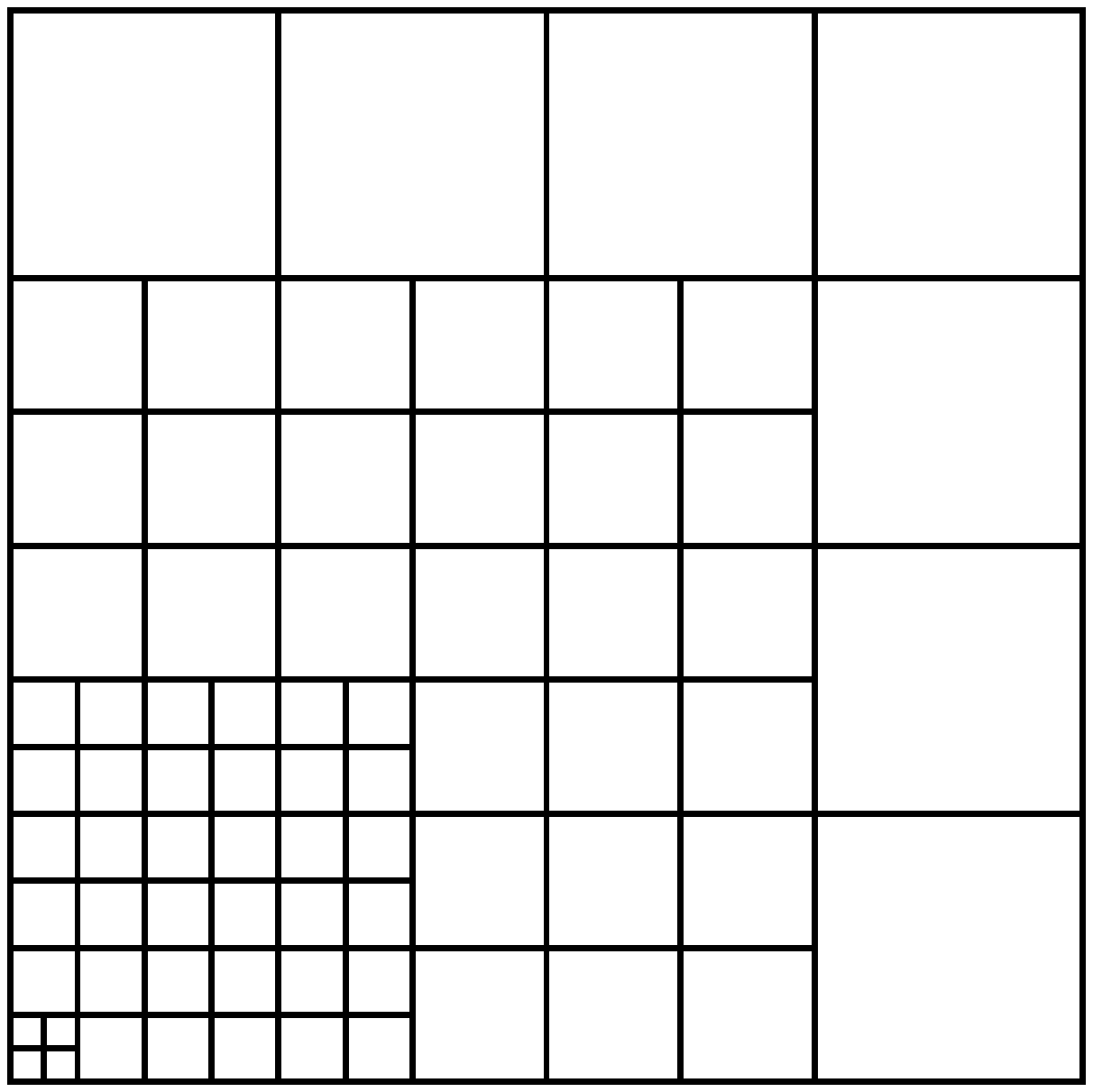}\qquad\qquad
\includegraphics[width=0.2\textwidth]{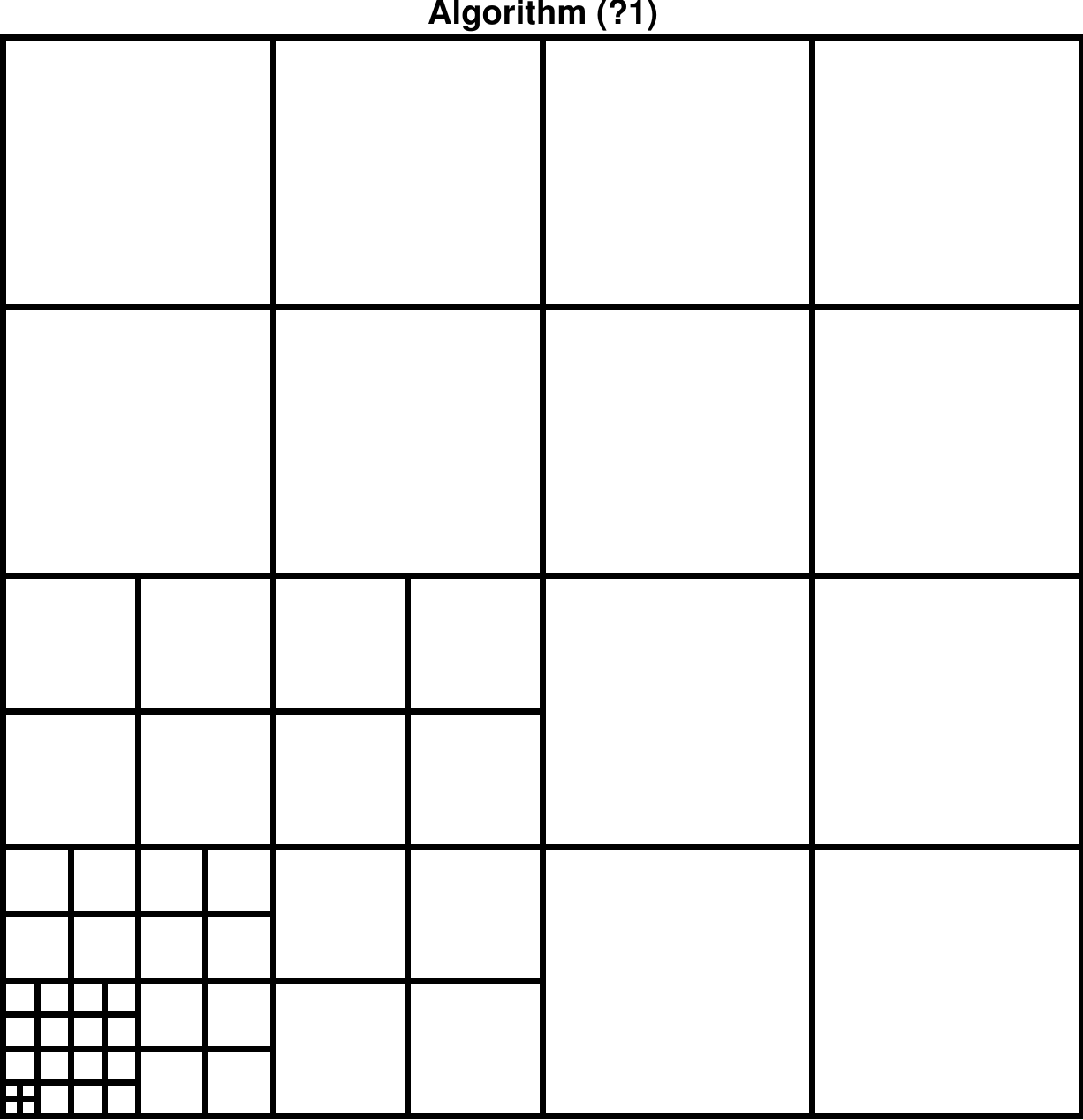}
\end{center}
\caption{
An initial  mesh $\widehat\TT_0$ with only one element $[0,1]^2$ is locally refined in the lower left corner using the  refinement of \cite{bg} (left) resp. the refinement of Algorithm~\ref{alg:refinement} (right); see Remark~\ref{rem:morgenstern} .
Consider the lowest-order case $(p_1,p_2)=(1,1)$. 
By repetitive refinement via \cite{bg}, the number of elements in the support of the hierarchical B-spline $\widehat B(s_1|0,1/2,1)\widehat B(s_2|0,1/2,1)$ grows to infinity.
Moreover, the number of hierarchical B-splines with support on the element in the lower left corner grows to infinity.
This is not the case if one uses Algorithm~\ref{alg:refinement}, see also Proposition~\ref{prop:bounded number}.
}
\label{fig:BGbad}
\end{figure}

\subsection{Hierarchical  meshes and splines in the physical domain $\bold{\Omega}$}\label{subsec:physical hsplines}
To transform the definitions in the parameter domain to the physical domain, we assume that we are given 
\begin{align}
\gamma:\overline{\widehat{\Omega}}\to\overline\Omega \quad\text{with}\quad  \gamma\in W^{1,\infty}({\widehat\Omega})\cap C^2(\widehat\TT_0)\quad\text{and}\quad\gamma^{-1}\in W^{1,\infty}(\Omega)\cap C^2(\TT_0),
\end{align}
where $C^2(\TT_0):=\set{v:\overline\Omega\to\R}{v|_T\in C^2(T)\text{ for all }T\in\TT_0}$.
Consequently, there exists
$C_{\gamma}>0$ such that for all 
$i,j,k\in\{1,\dots,d\}$
\begin{align}\label{eq:Cgamma}
\begin{split}
\Big\|\frac{\partial}{\partial t_j}\gamma_i\Big\|_{L^\infty(\widehat\Omega)}\le C_\gamma,\quad \Big\|\frac{\partial}{\partial x_j}(\gamma^{-1})_i\Big\|_{L^\infty(\Omega)}\le C_\gamma,\\
\Big\|\frac{\partial^2}{\partial t_j\partial t_k }\gamma_i\Big\|_{L^\infty(\widehat\Omega)}\le C_\gamma,\quad \Big\|\frac{\partial^2}{\partial x_j\partial x_k }(\gamma^{-1})_i\Big\|_{L^\infty(\Omega)}\le C_\gamma,
\end{split}
\end{align}
where $\gamma_i$ resp. $(\gamma^{-1})_i$ denotes the $i$-th component of $\gamma$ resp. $\gamma^{-1}$.
All previous definitions can now also be made in the physical domain, just by pulling them from the parameter domain via the diffeomorphism $\gamma$.
For these definitions, we drop the symbol~$\widehat\cdot$.
 If $\widehat\TT_\bullet\in\widehat\T$, we  define the corresponding mesh in the physical domain as  $\TT_\bullet:=\set{\gamma(\widehat T)}{\widehat T\in\widehat\TT_\bullet}$.
In particular, we have $\TT_0=\set{\gamma(\widehat T)}{\widehat T\in\widehat\TT_0}$.
Moreover, let  $\T:=\set{\TT_\bullet}{\widehat\TT_\bullet\in\widehat\T}$ be the  set of  admissible meshes in the physical domain.
If now $\MM_\bullet\subseteq\TT_\bullet$ with $\TT_\bullet\in\T$, we abbreviate $\widehat\MM_\bullet:=\set{\gamma^{-1}(T)}{T\in\MM_\bullet}$ and define $\refine(\TT_\bullet,\MM_\bullet):=\set{\gamma(\widehat T)}{\widehat T \in\refine(\widehat \TT_\bullet,\widehat\MM_\bullet)}$.
For $\TT_\bullet\in\T$, let  $\XX_\bullet:=\set{\widehat V_\bullet\circ\gamma^{-1}}{\widehat V_\bullet\in\widehat\XX_\bullet}$ be the the corresponding hierarchical spline space.
By  regularity of $\gamma$, we especially obtain
\begin{align}
\XX_\bullet\subset \set{v\in H_0^1(\Omega)}{v|_T\in  H^2(T)\text{ for all }T\in\TT_\bullet}.
\end{align}




\subsection{Main result}
Before we come to the main result of this work, we fix  polynomial orders $(q_1,\dots, q_d)$ and define for $\TT_\coarse \in\T$ the space of transformed polynomials 
\begin{align}\label{eq:polynomials}
\mathcal{P}(\Omega):=\set{\widehat V\circ\gamma}{\widehat V \text{ is a tensor polynomial of order }(q_1,\dots, q_d)}\end{align}

\begin{remark}
In order to obtain higher-order oscillations, the natural choice of the polynomial orders is $q_i\ge2p_i-1$; see, e.g., \cite[Section~3.1]{nv}.
If $\XX_\bullet\subset C^1(\overline\Omega)$, it is sufficient to choose $q_i\ge2p_i-2$; see Remark~\ref{rem:C12}.\qed
\end{remark}

Altogether, we have  specified the abstract framework of Section~\ref{sec:abstract setting} to hierarchical meshes and splines.
The following theorem is the main result of the present work. It shows that all assumptions of Theorem~\ref{thm:abstract} are satisfied for the present IGAFEM approach.
The proof is  given in Section \ref{sec:proof}.

\begin{theorem}\label{thm:main}
Hierarchical splines on admissible meshes satisfy the abstract assumptions  \eqref{M:shape}--\eqref{M:trace}, \eqref{R:sons}--\eqref{R:overlay}, and \eqref{S:inverse}--\eqref{S:grad} from Section~\ref{sec:abstract setting}, where the constants depend only on $d$, $\C{\gamma}$, $\widehat \TT_0$, and $(p_1,\dots,p_d)$.
Moreover, the piecewise polynomials $\mathcal{P}(\Omega)$ from \eqref{eq:polynomials} on admissible meshes satisfy the abstract assumptions \eqref{O:inverse}--\eqref{O:grad}, where the constants depend only on $d$, $\C{\gamma}$, $\widehat \TT_0$, and $(q_1,\dots,q_d)$.
By Theorem~\ref{thm:abstract}, this implies reliability \eqref{eq:reliable} as well as efficiency \eqref{eq:efficient} of the error estimator, linear convergence \eqref{eq:linear}, and quasi-optimal convergence rates \eqref{eq:optimal} for the adaptive strategy from Algorithm~\ref{the algorithm}.
\end{theorem}

\subsection{Generalization to rational hierarchical splines}\label{subsec:rational}
One can easily verify that all theoretical results of this work are still valid if one replaces the ansatz space $\XX_\bullet$ by rational hierarchical splines, i.e., by the set
\begin{align}
\XX_\bullet^{W_0}:=\Big\{\frac{V_\bullet}{W_0}:V_\bullet\in\XX_\bullet\Big\},
\end{align}
where $W_0$ is a fixed positive weight function in the initial ansatz space $\XX_0$.
In this case, the corresponding basis  consists of NURBS instead of B-splines.
Indeed, the mesh properties \eqref{M:shape}--\eqref{M:trace} as well as the refinement properties \eqref{R:sons}--\eqref{R:overlay} from Section~\ref{sec:abstract setting} are independent of the discrete spaces.
To verify the validity of Theorem~\ref{thm:main} in the NURBS setting, it thus only remains to verify the properties \eqref{S:inverse}--\eqref{S:grad} for the NURBS finite element spaces.
The inverse estimate \eqref{S:inverse} follows similarly as in Section ~\ref{subsec:E1.1 true} since we only consider a fixed and thus uniformly bounded weight function $W_0\in\XX_0$.
The properties \eqref{S:nestedness}--\eqref{S:local} depend only on the numerator of the NURBS functions and thus transfer.
To see \eqref{S:proj}--\eqref{S:grad}, one can proceed as in Section~\ref{subsec:E4.1 true}, where the corresponding Scott-Zhang type operator $J_\bullet^{W_0}:L^2(\Omega)\to \XX_\bullet^{W_0}$ now  reads 
\begin{align}
J_\bullet^{W_0}v:=\frac{J_\bullet(vW_0)}{W_0}\quad\text{for all }v\in L^2(\Omega).
\end{align}
With this definition,  Lemma~\ref{lem:Scott} holds accordingly, and \eqref{S:proj}--\eqref{S:grad} are proved as in Section~\ref{subsec:E4.1 true}.


\section{Sketch of proof of Theorem~\ref{thm:abstract}}
\label{sec:proof abstract}

In the following seven  subsections we sketch the  proof of Theorem~\ref{thm:abstract}, where we build  upon the analysis of~\cite{axioms}. 
Recall the residual {\sl a~posteriori} error estimator $\eta_\bullet$ from Section~\ref{subsec:estimator}.

\subsection{Discrete reliability}\label{subsec:discrete reliability}
Under the assumptions \eqref{M:patch}--\eqref{M:trace}, and \eqref{S:nestedness}--\eqref{S:grad}, we show that there exists $\Cdrel,\Cref\ge 1$ such that for all $\TT_\coarse\in\T$ and all $\TT_\fine\in\refine(\TT_\coarse)$, the subset $\RR_{\coarse\fine}:=\Pi_\coarse^\loc(\TT_\coarse\setminus\TT_\fine)\subseteq\TT_\coarse$ satisfies 
\begin{align*}
\norm{U_\fine-U_\coarse}{H^1(\Omega)}\le \Cdrel\,\eta_\coarse(\RR_{\coarse\fine}),
\quad
\TT_\coarse\setminus\TT_\fine\subseteq\RR_{\coarse\fine},
\quad\text{and}\quad
\# \RR_{\coarse\fine}\le\Cref \#(\TT_\coarse\setminus\TT_\fine).
\end{align*}
The last two properties are obvious with $\Cref=\Cpatch^{\kloc}$ by validity of~\eqref{M:patch} and~\eqref{S:local}. For the first property,
we argue as in~\cite[Theorem~4.1]{stevenson}: Ellipticity \eqref{eq:ellipticity}, $e_\fine:=U_\fine-U_\coarse\in\XX_\fine$ (which follows from \eqref{S:nestedness}), and Galerkin orthogonality~\eqref{eq:galerkin} with $V_\coarse := J_\coarse e_\fine\in\XX_\coarse$ prove 
\begin{align}
&\norm{U_\fine-U_\coarse}{H^1(\Omega)}^2
\lesssim\edual{e_\fine}{e_\fine}
= \edual{e_\fine}{(1-J_\coarse) e_\fine}
=\int_\Omega f(1-J_\coarse)e_\fine \,dx-\edual{U_\coarse}{(1-J_\coarse)e_\fine}.\notag
\intertext{We split $\Omega$ into elements $T\in\TT_\coarse$ and apply elementwise integration by parts, where we denote the outer  normal vector by $\nu$.
With $U_\coarse|_T\in H^2(T)$, this leads to}
\notag
&\qquad=\sum_{T\in\TT_\coarse} \bigg(\int_T f(1-J_\coarse)e_\fine\,dx-\int_T\big(-\div(\AA\nabla U_\coarse)+\bb\nabla U_\coarse+cU_\coarse\big)(1-J_\coarse)e_\fine\,dx\\
\label{eq:dpr1}
&\qquad\qquad\qquad+\int_{\partial T} (\AA \nabla U_\coarse\cdot \nu)\,(1-J_\coarse)e_\fine \,ds\bigg).
\end{align}
The properties \eqref{S:local}--\eqref{S:proj} immediately prove for any $V_\fine\in\XX_\fine$
\begin{align*}
 J_\coarse V_\fine
 = V_\fine\quad\text{on }\overline{\Omega\setminus\pi_\coarse^\loc(\TT_\coarse\setminus\TT_\fine)}
 = \overline{\Omega\backslash\bigcup\RR_{\coarse\fine}}.
 = \bigcup(\TT_\coarse\backslash\RR_{\coarse\fine}).
\end{align*}
Hence, the sum in~\eqref{eq:dpr1} reduces from $T\in\TT_\coarse$ to $T\in\RR_{\coarse\fine}$.
Recall that $(1-J_\coarse)e_\fine\in\XX_\fine\subset H^1_0(\Omega)$ with $(1-J_\coarse)e_\fine=0$ on $\partial(\bigcup \RR_{\coarse\fine})$.
We define  the set of facets $\mathcal{E}_{\coarse\fine}:= \set{ T_1\cap  T_2}{T_1,T_2\in\RR_{\coarse\fine}\text{ with }T_1\neq T_2\wedge |T_1\cap T_2|>0}$, where $|\cdot|$ denotes the $(d-1)$-dimensional measure. 
Almost  all $x\in\bigcup\mathcal{E}_{\coarse\fine} $   belong to precisely two elements with opposite normal vectors. Hence, 
\begin{align*}
&\sum_{T\in\RR_{\coarse\fine}}\int_{\partial T} (\AA \nabla U_\coarse\cdot \nu)\,(1-J_\coarse)e_\fine \,ds=\sum_{T\in\RR_{\coarse\fine}}\int_{\partial T\cap \Omega} (\AA \nabla U_\coarse\cdot \nu)\,(1-J_\coarse)e_\fine \,ds\\
&\quad\le\sum_{E\in\mathcal{E}_{\coarse\fine}}\int_{E} \big|[\AA \nabla U_\coarse\cdot \nu]\,(1-J_\coarse)e_\fine\big| \,ds= \frac{1}{2} \sum_{T\in\RR_{\coarse\fine}} \int_{\partial T\cap \Omega}\big|[\AA \nabla U_\coarse\cdot \nu](1-J_\coarse)e_\fine\big|\,ds.
\end{align*}
Altogether, we have derived
\begin{align}
&\norm{U_\fine-U_\coarse}{H^1(\Omega)}^2
\lesssim \sum_{ T\in\RR_{\coarse\fine}}\bigg(\int_T (f-\LL U_\coarse)(1-J_\coarse)e_\fine\,dx+\int_{\partial T\cap \Omega}\big|[\AA \nabla U_\coarse\cdot \nu](1-J_\coarse)e_\fine\big|\,ds\bigg)\notag\\
&\qquad\le \sum_{ T\in\RR_{\coarse\fine}} \bigg(|T|^{1/d}\norm{f-\LL U_\coarse}{L^2(T)} \, |T|^{-1/d} \norm{(1-J_\coarse)e_\fine}{L^2(T)}\label{eq:E4 equation}\\
&\qquad\qquad\qquad\quad+|T|^{1/(2d)}\norm{[\AA \nabla U_\coarse\cdot \nu]}{L^2(\partial T\cap \Omega)} \,  |T|^{-1/(2d)}\norm{(1-J_\coarse)e_\fine}{L^2(\partial T\cap \Omega)}\bigg).\notag
\end{align}
We abbreviate $\pi_\coarse^{\max}:=\pi_\coarse^{\max(\kapp,\kgrad)}$.
By \eqref{M:trace}, \eqref{S:app}, and \eqref{S:grad}, we have 
\begin{align*}
|T|^{-1/d} \norm{(1-J_\coarse)e_\fine}{L^2(T)}
+ |T|^{-1/(2d)} \norm{(1-J_\coarse)e_\fine}{L^2(\partial T\cap\Omega)}
\lesssim \norm{U_\fine-U_\coarse}{H^1(\pi^{\rm max}_\coarse(T))}.
\end{align*}
Plugging this into \eqref{eq:E4 equation} and using the Cauchy-Schwarz inequality, we obtain
\begin{align*}
\norm{U_\fine-U_\coarse}{H^1(\Omega)}^2\lesssim \Big(\sum_{ T\in\RR_{\coarse\fine}}\eta_\coarse(T)^2\Big)^{1/2}  \Big(\sum_{T\in\RR_{\coarse\fine}}\norm{U_\fine-U_\coarse}{H^1(\pi^{\rm max}_\coarse(T))}^2\Big)^{1/2}.
\end{align*}
With~\eqref{M:patch}, the second factor is controlled  by $\norm{U_\fine-U_\coarse}{H^1(\Omega)}$.
This concludes the current section, and $\Cdrel$ depends only on $\Cell$, \eqref{M:patch}--\eqref{M:trace}, and \eqref{S:nestedness}--\eqref{S:grad}.


\subsection{Reliability (\ref{eq:reliable})} \label{subsec:reliability}
Note that $\osc_\bullet\lesssim \eta_\bullet$ follows immediately from their definitions~\eqref{eq:eta} --\eqref{eq:osc}.
If one replaces $U_\fine\in\XX_\fine$ by the exact solution $u\in H_0^1(\Omega)$ and $\RR_{\coarse\fine}$ by $\TT_\coarse$, reliability \eqref{eq:reliable} follows along the lines of  Section~\ref{subsec:discrete reliability}, but now,  \eqref{S:nestedness}--\eqref{S:proj} are not needed for the proof.

\subsection{Efficiency (\ref{eq:efficient})}
As in \cite[Theorem 7]{nv}, the assumptions \eqref{M:shape}--\eqref{M:patch} and \eqref{O:inverse}--\eqref{O:grad} imply that
\begin{align}\label{eq:efficienttwo}
\eta_\coarse\lesssim \norm{u-U_\coarse}{H^1(\Omega)}+\osc_\coarse(U_\coarse).
\end{align}
As in \cite[Proposition 3.3]{ckns}, the assumptions~\eqref{M:shape}--\eqref{M:trace} and~\eqref{S:inverse}  imply that
\begin{align}\label{eq:osc monoton}
\osc(U_\coarse) \lesssim \osc_\coarse(V_\coarse) +\norm{U_\coarse-V_\coarse}{H^1(\Omega)}\quad\text{for all }V_\coarse\in \XX_\coarse.
\end{align}
The C\'ea lemma~\eqref{eq:cea} and \eqref{eq:osc monoton} show that 
\begin{align*}
\norm{u-U_\coarse}{H^1(\Omega)}+\osc_\coarse(U_\coarse) &\stackrel{\eqref{eq:osc monoton}}{\lesssim} \norm{u-U_\coarse}{H^1(\Omega)} +\osc_\coarse(V_\coarse) +\norm{u-V_\coarse}{H^1(\Omega)}\\
&\stackrel{\eqref{eq:cea}}{\lesssim} \norm{u-V_\coarse}{H^1(\Omega)} +\osc_\coarse(V_\coarse)\quad\text{for all } V_\coarse \in\XX_\coarse.
\end{align*}
This proves $\norm{u-U_\coarse}{H^1(\Omega)}+\osc_\coarse(U_\coarse)\simeq\inf_{V_\coarse\in\XX_\coarse}\big(\norm{u-V_\coarse}{H^1(\Omega)}+\osc_\coarse(V_\coarse)\big)$.
Combining these observations, we conclude \eqref{eq:efficient}, where $\C{eff}$ depends only on $\C{ell}$, \eqref{M:shape}--\eqref{M:trace}, \eqref{S:inverse} and ~\eqref{O:inverse}--\eqref{O:grad}.

\subsection{Stability on non-refined elements}\label{subsec:stability}
As in~\cite[Corollary~3.4]{ckns}, the assumptions~\eqref{M:shape}--\eqref{M:trace} and~\eqref{S:inverse} imply the existence of $\Cstab\ge1$ such that for all $\TT_\coarse\in\T$, all $\TT_\fine\in\refine(\TT_\coarse)$, and all subsets $\SS\subseteq\TT_\coarse\cap\TT_\fine$ of non-refined elements, it holds that 
\begin{align*}
|\eta_\coarse(\SS)-\eta_\fine(\SS)|\le \Cstab\norm{U_\coarse-U_\fine}{H^1(\Omega)}.
\end{align*}
The constant $\Cstab$ depends only on  \eqref{M:shape}--\eqref{M:trace}, \eqref{S:inverse},  as well as on $\norm{\AA}{W^{1,\infty}(\Omega)}$, $\norm{\bb}{L^\infty(\Omega)}$, $\norm{c}{L^\infty(\Omega)}$, and $\diam(\Omega)$.

\subsection{Reduction on refined elements}\label{subsec:reduction}
As in~\cite[Corollary 3.4]{ckns}, the assumptions~\eqref{M:shape}--\eqref{M:trace}, \eqref{R:union}--\eqref{R:reduction}, and \eqref{S:inverse} imply the existence of $\Cred\ge1$ and $0<\qred<1$ such that for all $\TT_\coarse\in\T$ and all $\TT_\fine\in\refine(\TT_\coarse)$, it holds that
\begin{align*}
\eta_\fine(\TT_\fine\setminus\TT_\coarse)^2
\le \qred \, \eta_\coarse(\TT_\coarse\backslash\TT_\fine)^2 
+ \Cred\norm{U_\fine-U_\coarse}{H^1(\Omega)}^2.
\end{align*}
The constants $\Cred$ and $\qred=\qson^{1/d}$ depend only on  \eqref{M:shape}--\eqref{M:trace}, \eqref{R:union}--\eqref{R:reduction},  \eqref{S:inverse}  as well as on $\norm{\AA}{W^{1,\infty}(\Omega)}$, $\norm{\bb}{L^\infty(\Omega)}$, $\norm{c}{L^\infty(\Omega)}$, and $\diam(\Omega)$. 
Note that $[\AA\nabla U_\coarse\cdot \nu]=0$ on $(\partial T'\setminus\partial T)\cap \Omega$ for all sons $T'\subsetneqq T$ of an element $T\in\TT_\coarse$, since $U_\coarse|_T\in H^2(T)$.

\subsection{Estimator reduction principle}\label{subsec:estimator reduction}
Choose sufficiently small $\delta>0$ such that $0<\qest := (1+\delta)\big(1-(1-\qred)\theta\big)<1$ and define $\Cest:=\Cred+(1+\delta^{-1})\Cstab^2$. With Section~\ref{subsec:stability}--\ref{subsec:reduction}, elementary calculation shows that
\begin{align}\label{eq:estred}
 \eta_{\ell+1}^2 \le \qest\,\eta_\ell^2 + \Cest\,\norm{U_{\ell+1}-U_\ell}{H^1(\Omega)}^2
 \quad\text{for all }\ell\in\N_0;
\end{align}
see~\cite[Lemma~4.7]{axioms}. Nestedness~\eqref{S:nestedness} ensures that $\XX_\infty:=\overline{\bigcup_{\ell\in\N_0}\XX_\ell}$ is a closed subspace of $H^1_0(\Omega)$ and hence admits a unique Galerkin solution $U_\infty\in\XX_\infty$. Note that $U_\ell$ is also a Galerkin approximation of $U_\infty$. Hence, the C\'ea lemma~\eqref{eq:cea} with $u$ replaced by $U_\infty$ and a density argument prove $\norm{U_\infty-U_\ell}{H^1(\Omega)}\to0$ as $\ell\to\infty$. Elementary calculus, estimator reduction~\eqref{eq:estred}, and Section~\ref{subsec:reliability} thus prove $\norm{u-U_\ell}{H^1(\Omega)} \lesssim \eta_\ell \to0$ as $\ell\to\infty$. This proves $U_\ell\to U_\infty=u$; see~\cite[Section~2]{MR2908795} for the detailed argument.

\subsection{General quasi-orthogonality} \label{subsec:quasi-orthogonality}
By use of  reliability from Section~\ref{subsec:reliability}, the plain convergence result from Section~\ref{subsec:estimator reduction} and a perturbation argument (since the non-symmetric part of $\LL$ is compact), it is shown in~\cite[Proof of Theorem~8]{ffp14} that
\begin{align}\label{eq:quasi-orthogonality}
 \sum_{k=\ell}^N\big(\norm{U_{k+1}\!-\!U_k}{H^1(\Omega)}^2-\varepsilon\,\norm{u\!-\!U_k}{H^1(\Omega)}^2\big)
 \le C(\eps)\,\eta_\ell^2 \text{ for all }0\le\ell\le N \text{ and all }\varepsilon>0.
\end{align}
The constant $C(\eps)$ depends on $\eps>0$, the operator $\LL$, the sequence $(U_\ell)_{\ell\in\N_0}$, and $\Crel$.
See also~\cite[Proposition~6.1]{axioms} for a short paraphrase of the proof.

\begin{remark}\label{rem:E3}
If the bilinear form $\edual{\cdot}{\cdot}$ is symmetric,~\eqref{eq:quasi-orthogonality} follows from the Pythagoras theorem in the $\LL$-induced energy norm $\enorm{v}^2:=\edual{v}{v}$ and norm equivalence
\begin{align*}
 \sum_{k=\ell}^N\norm{U_{k+1}-U_k}{H^1(\Omega)}^2
 \simeq \sum_{k=\ell}^N\enorm{U_{k+1}-U_k}^2
 = \enorm{u-U_\ell}^2-\enorm{u-U_N}^2 
 \lesssim \norm{u-U_\ell}{H^1(\Omega)}^2.
\end{align*}
Together with reliability~\eqref{eq:reliable}, this proves~\eqref{eq:quasi-orthogonality} even for $\varepsilon=0$, and $C(\eps)\simeq\Crel^2$ is independent of the sequence $(U_\ell)_{\ell\in\N_0}$.\qed
\end{remark}%

\subsection{Linear convergence with optimal rates}
The remaining claims \eqref{eq:linear}--\eqref{eq:optimal} follow from~\cite[Theorem~4.1]{axioms} which only relies on \eqref{R:sons}, \eqref{eq:R:refine}, \eqref{R:closure}--\eqref{R:overlay}, \emph{stability} (Section~\ref{subsec:stability}), \emph{reduction} (Section~\ref{subsec:reduction}), \emph{quasi-orthogonality} (Section~\ref{subsec:quasi-orthogonality}), \emph{discrete reliability} (Section~\ref{subsec:discrete reliability}), and reliability (Section~\ref{subsec:reliability}).


\section{Proof of Theorem~\ref{thm:main}}\label{sec:proof}

\subsection{Admissibility and $\refine$}
In this section, we  show that, given a mesh $\widehat\TT_\coarse\in\widehat\T$, iterative application of the refinement  Algorithm  \ref{alg:refinement} generates exactly the set of all admissible meshes $\widehat\TT_\fine$ that are finer than $\widehat\TT_\coarse$. 
In particular, this implies that $\widehat \T$ coincides with the set of all admissible hierarchical meshes that are finer than $\widehat\TT_0$, which we has already been mentioned in Section~\ref{subsec:concrete refinement}.
We start with the following lemma.
\begin{lemma}\label{lem:active supports shrink}
Let $\widehat\TT_\coarse$ and $\widehat\TT_\circ$ be hierarchical meshes such that $\widehat\TT_\circ$ is finer than $\widehat\TT_\coarse$, i.e., $\widehat\Omega_\coarse^k\subseteq \widehat\Omega_\circ^k$ for all $k\in\N_0$.
Then, for all $\widehat\beta_\circ\in\widehat\HH_\circ$ there exists $\widehat\beta_\coarse\in\widehat\HH_\coarse$ with $\supp(\widehat\beta_\circ)\subseteq\supp(\widehat\beta_\coarse)$.
\end{lemma}

\begin{proof}
Clearly, we may assume $\widehat\beta_\circ\in\widehat\HH_\circ\setminus \widehat\HH_\coarse$.
Let $k:=\level(\widehat\beta_\circ)$ and define $\widehat\beta^k:=\widehat\beta_\circ$.
Since $\widehat\beta^k\in\widehat\HH_\circ$, \eqref{eq:short cHH} implies that  $\supp(\widehat\beta^k)\setminus \widehat\Omega_\circ^{k+1}\neq\emptyset$ and $\supp(\widehat\beta^k)\subseteq\widehat\Omega_\circ^k$.
Since $\widehat\beta^k\not \in\widehat\HH_\coarse$, \eqref{eq:short cHH} implies that  $\supp(\widehat\beta^k)\setminus\widehat\Omega_\coarse^{k+1}=\emptyset$ or $\supp(\widehat\beta^k)\not\subseteq\widehat\Omega_\coarse^k$.
However, $\widehat\Omega_\coarse^{k+1}\subseteq\widehat\Omega_\circ^{k+1}$ and $\supp(\widehat\beta^k)\setminus\widehat\Omega_\circ^{k+1}\neq\emptyset$ imply that  $\supp(\widehat\beta^k)\setminus\widehat\Omega_\coarse^{k+1}=\emptyset$.
Hence, we have $\supp(\widehat\beta^k)\not\subseteq\widehat\Omega_\coarse^k$, which especially implies $k>0$.
This is equivalent to $\supp(\widehat\beta^k)\setminus\widehat\Omega_\coarse^k\neq \emptyset$.
Clearly, there exists $\widehat\beta^{k-1}\in\BB^{k-1}$ with $\supp(\widehat\beta^k)\subseteq\supp(\widehat\beta^{k-1})$.
If $\widehat\beta^{k-1}\in\widehat\HH_\coarse$, we are done. Otherwise, \eqref{eq:short cHH} implies that  $\supp(\widehat\beta^{k-1})\setminus\widehat\Omega_\coarse^k=\emptyset$ or $\supp(\widehat\beta^{k-1})\not\subseteq\widehat\Omega_\coarse^{k-1}$.
Again, the first case is not possible because 
\begin{align*}
\supp(\widehat\beta^{k-1})\setminus\widehat\Omega_\coarse^k\supseteq\supp(\widehat\beta^k)\setminus\widehat\Omega_\coarse^k\neq\emptyset.
\end{align*}
Hence, we have $\supp(\widehat\beta^{k-1})\not\subseteq\widehat\Omega_\coarse^{k-1}$ which especially implies $k-1>0$.
This is equivalent to $\supp(\widehat\beta^{k-1})\setminus\widehat\Omega_\coarse^{k-1}\neq\emptyset$.
Inductively, we obtain a sequence $\widehat\beta^k,\dots,\widehat\beta^K$ with $\widehat\beta^j\in\BB^j$ and $\supp(\widehat\beta^K)\supseteq\dots\supseteq\supp(\widehat\beta^k)$, where $\widehat\beta^K\in\widehat\HH_\coarse$ for some $K\ge 0$.
\end{proof}


\begin{proposition}\label{prop:refineT subset T}
If $\widehat\TT_\coarse\in\widehat\T$, then $\refine(\widehat\TT_\coarse)$ coincides with the set of all admissible hierarchical meshes $\widehat\TT_\circ\in\widehat\T$ that are finer than $\widehat\TT_\coarse$.
\end{proposition}

\begin{proof}
We prove the assertion in four steps.

\textbf{Step 1:}
We show that $\widehat\TT_\circ:=\refine(\widehat\TT_\coarse,\widehat\MM_\coarse)\in\widehat\T$ for any $\widehat\MM_\coarse\subseteq\widehat\TT_\coarse$.
Let $\widehat T,\widehat T'\in\TT_\circ$ with $\widehat T'\in\NN_\circ(\widehat T)$, i.e., there exists $\widehat \beta_\fine\in\widehat \HH_\fine$ with $|\widehat T\cap\supp(\widehat \beta_\fine)|\neq0\neq  |\widehat T'\cap\supp(\widehat \beta_\fine)|$; see \eqref{eq:neigbors}.
By Lemma \ref{lem:active supports shrink}, there exists some (not necessarily unique) $ \widehat \beta_\coarse\in\widehat \HH_\coarse$ with $\supp(\widehat \beta_\fine)\subseteq\supp( \widehat \beta_\coarse)$.
We consider four different cases.
\begin{enumerate}
\item[(i)]
Let $\widehat T,\widehat T'\in\widehat \TT_\coarse$.
 Then, $|\widehat T\cap\supp(\widehat \beta_\coarse)|\neq0\neq| \widehat T'\cap\supp(\widehat \beta_\coarse)|$, i.e., $\widehat T'\in\NN_\coarse(\widehat T)$ and hence $|\level(\widehat T)-\level(\widehat T')|\le 1$ by $\widehat \TT_\coarse\in\widehat \T$.
 \item[(ii)]
Let $\widehat T,\widehat T'\in\widehat \TT_\circ\setminus\widehat \TT_\coarse$.
Let $\widehat T_\coarse,\widehat T_\coarse '\in\widehat\TT_\coarse$ with $\widehat T\subsetneqq\widehat T_\coarse$, $\widehat T'\subsetneqq\widehat T_\coarse'$.
Then, it holds that  $\level(\widehat T)=\level(\widehat T_\coarse)+1$, $\level(\widehat T')=\level(\widehat T_\coarse')+1$ as well as  $|\widehat T_\coarse\cap\supp(\widehat \beta_\coarse)|\neq0\neq |\widehat T_\coarse'\cap\supp(\widehat \beta_\coarse)|$.
By definition, it follows that $\widehat T_\coarse'\in\NN_\coarse(\widehat T_\coarse)$ and hence $|\level(\widehat T)-\level(\widehat T')|=|\level(\widehat T_\coarse)-\level(\widehat T'_\coarse)|\le 1$ by $\widehat \TT_\coarse\in\widehat \T$.
\item[(iii)]
Let $\widehat T\in\widehat \TT_\circ\setminus\widehat \TT_\coarse, \widehat T'\in\widehat \TT_\coarse$.
Let $\widehat T_\coarse\in\widehat\TT_\coarse$ with $\widehat T\subsetneqq\widehat T_\coarse$.
Then, 
$|\widehat T_\coarse\cap\supp(\widehat \beta_\coarse)|\neq 0\neq |\widehat T'\cap\supp(\widehat \beta_\coarse)|$, and  $|\level(\widehat T_\coarse)-\level(\widehat T')|\le 1$ by  $\widehat \TT_\coarse\in\widehat \T$.
We argue by contradiction and assume $|\level(\widehat T)-\level(\widehat T')|>1$. 
Together with $\level(\widehat T_\coarse)+1=\level(\widehat T)$, this yields  $\level(\widehat T_\coarse)-1=\level(\widehat T')$. 
Hence, $\widehat T'\in\NN_\coarse^{\rm bad}(\widehat T_\coarse)$ with $\widehat T_\coarse\in\widehat\MM_\coarse^{(\rm end)}$.
By Algorithm \ref{alg:refinement} {(\rm i)}, we get $\widehat T'\in\widehat\MM_\coarse^{(\rm end)}$.
This contradicts $\widehat T'\in\widehat \TT_\coarse$ and hence proves $|\level(\widehat T)-\level(\widehat T')|\le 1$.
\item[(iv)]
Let $\widehat T\in\widehat \TT_\coarse, \widehat T'\in\widehat \TT_\circ\setminus \widehat \TT_\coarse$.
Since $\widehat T'\in\NN_\coarse(\widehat T)$ is equivalent to $\widehat T\in\NN_\coarse(\widehat T')$, we argue as in (iii) to conclude $|\level(\widehat T)-\level(\widehat T')|\le 1$.
\end{enumerate}

\textbf{Step 2:} 
It is clear that an arbitrary $\widehat\TT_\circ\in \refine(\widehat\TT_\coarse)$ is finer than $\widehat\TT_\coarse$.
By induction, Step~1 proves the inclusion $\refine(\widehat\TT_\coarse)\subseteq\widehat\T$.

\textbf{Step 3:}
To prove the converse inclusion, let $\widehat\TT_\circ\in\widehat\T$ be an admissible mesh that is finer than $\widehat\TT_\coarse$.
Moreover, let $\widehat T\in\widehat\TT_\coarse\setminus\widehat\TT_\circ$.
We show that $\widehat\TT_\circ$ is also finer than $\widehat\TT_\star:=\refine(\widehat\TT_\coarse,\{\widehat T\})$.
We argue by contradiction and suppose that  $\widehat\TT_\circ$ is not finer than $\widehat\TT_\star$.
Since $\refine$ bisects each element of $\widehat\TT_\coarse$ at most once, there  exists a refined element $\widehat T^{(0)}\in\widehat\TT_\bullet\backslash\widehat\TT_\star$ which is also in $\widehat\TT_\fine$, i.e., $\widehat T^{(0)}\in(\widehat\TT_\coarse\setminus\widehat\TT_\star)\cap\widehat\TT_\fine$.
In particular, $\widehat T^{(0)}\neq\widehat T\in\widehat\TT_\bullet\backslash\widehat\TT_\circ$.
Thus, Algorithm~\ref{alg:refinement} shows that $\widehat T^{(0)}\in\NN^{\rm bad}_\coarse(\widehat T^{(1)})$  for some $\widehat T^{(1)}\in\widehat\TT_\coarse\setminus\widehat\TT_\star$.
If $\widehat T^{(1)}\in\widehat\TT_\fine$ and hence $\widehat T^{(1)}\in (\widehat\TT_\coarse\setminus\widehat\TT_\star)\cap\widehat\TT_\fine$, we have again $\widehat T^{(1)}\neq \widehat T$ as well as $\widehat T^{(1)}\in\NN^{\rm bad}_\coarse(\widehat T^{(2)})$  for some $\widehat T^{(2)}\in\widehat\TT_\coarse\setminus\widehat\TT_\star$.
Inductively, we see the existence of $\widehat\TT^{(J-1)}\in(\widehat\TT_\coarse\setminus\widehat\TT_\star)\cap\widehat\TT_\fine$ such that $\widehat T^{(J-1)}\in\NN^{\rm bad}_\coarse(\widehat T^{(J)})$  for some $\widehat T^{(J)}\in\widehat\TT_\coarse\setminus\widehat\TT_\star$ with $\widehat T^{(J)}\not\in\widehat\TT_\fine$.
In particular, this implies the existence of $\widehat T_\fine^{(J)}\in\widehat\TT_\fine$ with $\widehat T_\fine^{(J)}\subsetneqq \widehat T^{(J)}$.

By  definition of $\NN_\coarse^{\rm bad}(\cdot)$, we have 
   $\widehat T^{(J)},\widehat T^{(J-1)}\subseteq\supp(\widehat\beta)$ for some $\widehat\beta\in\widehat\HH_\coarse$ as well as $\level(\widehat T^{(J-1)})=\level(\widehat T^{(J)})-1$.
Hence, \eqref{eq:level beta is} and $\widehat\TT_\coarse\in\widehat\T$ show $k:=\level(\widehat\beta)=\level(\widehat T^{(J-1)})$.
Since $\widehat T^{(J-1)}\in\widehat \TT_\fine$,
\eqref{eq:parameter mesh} implies $\widehat T^{(J-1)}\not\subseteq\widehat\Omega_\fine^{k+1}$ and hence $\supp(\widehat\beta)\not\subseteq \widehat\Omega_\fine^{k+1}$.
Moreover, \eqref{eq:short cHH} shows $\supp(\widehat\beta)\subseteq\widehat\Omega_\coarse^k\subseteq\widehat\Omega_\fine^k$.
Using \eqref{eq:short cHH} again, we see  $\widehat\beta\in\widehat\HH_\circ$. 
Together with $\widehat T_\fine^{(J)},\widehat T^{(J-1)}\subseteq\supp(\widehat\beta)$ and $\level( \widehat T_\fine^{(J)})\ge\level(\widehat T^{(J)})+1=\level(\widehat T^{(J-1)})+2$, this contradicts  admissibility of $\widehat\TT_\circ\in\widehat\T$, and concludes the proof.

\textbf{Step 4:}
Let again $\widehat\TT_\circ\in\widehat\T$ be an arbitrary admissible mesh that is finer than $\widehat\TT_\coarse$.
Step 3 together with Step 2 shows that we can iteratively refine $\widehat \TT_\coarse$ and obtain a sequence $\widehat\TT_{(0)},\dots,\widehat\TT_{(J)}$ with $\widehat\TT_\coarse=\widehat\TT_{(0)}$, $\widehat\TT_{(j+1)}=\refine(\widehat\TT_{(j)},\{\widehat T_{(j)}\})$ with some $\widehat T_{(j)}\in\widehat\TT_{(j)}\setminus\widehat\TT_{(j+1)}$  for $j=1,\dots,J-1$ and $\widehat\TT_{(J)}=\widehat\TT_\fine$.
By definition, this proves $\widehat\TT_\fine\in\refine(\widehat\TT_\coarse)$.
\end{proof}

\subsection{Verification of (\ref{M:shape})--(\ref{M:trace})}\label{subsec:M true}
The  mesh properties   \eqref{M:shape}--\eqref{M:trace} essentially follow from admissibility in the sense of Section~\ref{subsec:admissible meshes} in combination with  the following lemma. 
\begin{lemma}\label{lem:patch2neigbor}
Let $\widehat\TT_\coarse$ be an arbitrary hierarchical mesh in the parameter domain.
Then, 
%
\begin{align}\label{eq:patch2neigbor}
\Pi_\coarse(\widehat T)\subseteq \NN_\coarse(\widehat T)\quad\text{for all } \widehat T\in\widehat\TT_\coarse.
\end{align}
\end{lemma}
\begin{proof}
Let $\widehat T'\in\Pi_\coarse(\widehat T)$, i.e., $\widehat T'\in \widehat \TT_\coarse$ with $\widehat T\cap \widehat T'\neq\emptyset$. 
We abbreviate $k:=\level(\widehat T)$. 
Since all knot multiplicities are smaller that $p+1$, there exists $\widehat \beta^k\in\widehat \BB^k$ such that $|\widehat T\cap\supp(\widehat \beta^k)|\neq 0\neq|\widehat T'\cap\supp(\widehat \beta^k)|$.
If $\widehat \beta^k\in\widehat \HH_\coarse$, then $\widehat T'\in\widehat \NN_\coarse(\widehat T)$.
If $\widehat \beta^k\not\in\widehat \HH_\coarse$, the characterization  \eqref{eq:short cHH} shows that  $\supp(\widehat \beta^k)\not\subseteq\widehat \Omega^k_\coarse$ or $\supp(\widehat \beta^k)\subseteq\widehat \Omega^{k+1}_\coarse$.
By choice of $k$, it  holds that $\widehat T\subseteq\supp(\widehat\beta^k)$.
In view of \eqref{eq:parameter mesh}, $\widehat T\in\widehat\TT_\coarse$ implies $\widehat T\not \subseteq \widehat\Omega_\coarse^{k+1}$.
Hence, $\supp(\widehat\beta^k)\not \subseteq \widehat\Omega_\coarse^k$ and, in particular, $k>0$.
Next, there exists $\widehat \beta^{k-1}\in\widehat \BB^{k-1}$ such that $\supp(\widehat \beta^{k})\subseteq\supp(\widehat \beta^{k-1})$.
If $\widehat \beta^{k-1}\in\widehat \HH_\coarse$, then $\widehat T'\in\widehat \NN_\coarse(\widehat T)$.
If $\widehat \beta^{k-1}\not\in\widehat \HH_\coarse$, there holds again either $\supp(\widehat \beta^{k-1})\not\subseteq\widehat \Omega^{k-1}_\coarse$ or  $\supp(\widehat \beta^{k-1})\subseteq\widehat \Omega^{k}_\coarse$. 
Due to $\supp(\widehat \beta^k)\not\subseteq\widehat \Omega^k_\coarse$, the second case is not possible.
Hence, $\supp(\widehat\beta^{k-1})\not \subseteq \widehat\Omega_\coarse^{k-1}$ and, in particular, $k-1>0$.
We proceed in the same way to get  a sequence $\widehat \beta^k,\dots,\widehat \beta^K$ with $\widehat \beta^j\in\widehat \BB^j$ and $\supp(\widehat \beta^K)\supseteq\dots\supseteq\supp(\widehat \beta^k)$, where $\widehat \beta^K\in\widehat\HH_\coarse$ for some $K\ge 0$.
\end{proof}

We define the patches $\pi_\coarse(\cdot)$ and $\Pi_\coarse(\cdot)$ in the parameter domain analogously to the patches  in the physical domain, see Section \ref{subsec:general mesh}.

With Lemma~\ref{lem:patch2neigbor}, one can easily verify that $\T$ satisfies \eqref{M:shape}--\eqref{M:trace}: 
Let $\TT_\coarse\in\T$.
We start with \eqref{M:shape}.  Let $T\in\TT_\coarse$ and $T'\in\Pi_\coarse(T)$.
Lemma \ref{lem:patch2neigbor} and admissibility show for the corresponding elements $\widehat T,\widehat T'$ in the parameter domain that $|\level(\widehat T)-\level(\widehat T')|\le 1$, wherefore $|\widehat T|\simeq|\widehat T'|$.
Regularity \eqref{eq:Cgamma} of the transformation $\gamma$ finally yields $|T|\simeq| T'|$.
The constant $\Cshape$ depends only on $d$, $\C{\gamma}$, and $\TT_0$.

To prove \eqref{M:patch}, let $T\in\TT_\coarse$ and $T'\in\Pi_\coarse(T)$.
As before, we have  $|\level(\widehat T)-\level(\widehat T')|\le 1$ for the corresponding elements in the parameter domain.
With this, one  easily sees that  $\#\Pi_\coarse(T)\le  \Cpatch$ with a constant $\Cpatch>0$ that depends only on the dimension $d$.

Regularity \eqref{eq:Cgamma} of $\gamma$ shows that it is sufficient to prove \eqref{M:trace} for hyperrectangles $\widehat T$ in the parameter domain.
There, the trace inequality \eqref{M:trace} is well-known; see, e.g., \cite[Satz~3.4.5]{erath}.
The constant $\Ctrace$ depends only on  $d$, $\C{\gamma}$, and $\TT_0$.

\subsection{Verification of  (\ref{R:sons})--(\ref{R:reduction})}
Let $\TT_\coarse\in\T$, $\TT_\circ\in\refine(\TT_\coarse)$, and $T\in\TT_\coarse$.
 \eqref{R:sons} is trivially satisfied with $\Cson=2^d$, since each refined element is split into exactly $2^d$ elements.
Moreover, the union of sons property \eqref{R:union} holds by definition.

To see the reduction property \eqref{R:reduction}, let $T'\in\TT_\circ$ with $T'\subsetneqq T$. 
Since each refined element is split it into $2^d$ elements, we have for the corresponding elements in the parameter domain $|\widehat T'|\le 2^{-d}|\widehat T|$.
Next, we prove $|T'|\le q_{\rm son}|T|$ with a constant $0<q_{\rm son}<1$ which depends only on $d$ and $C_\gamma$.
Indeed, we even prove for arbitrary measurable sets $\widehat S'\subseteq \widehat S\subseteq\overline{\widehat\Omega}$ and  $S:=\gamma(\widehat S)$, $S':=\gamma(\widehat S')$ that $0<|\widehat S'|\le 2^{-d} |\widehat S|$ implies $|S'|\le q_{\rm son} |S|$.
To see this, we argue by contradiction and assume that  there are two sequences of such sets $(\widehat S_n)_{n\in\N}$ and $(\widehat S_n')_{n\in\N}$ with $|S_n'|/|S_n|\to 1$.
This implies $|S_n\setminus S_n'|/|S_n|\to 0$ and yields the contradiction
\begin{align*}
 1-2^{-d}\le \frac{|\widehat S_n\setminus\widehat S_n'|}{|\widehat S_n|}\simeq\frac{\int_{\widehat S_n\setminus \widehat S_n'}{|\det D \gamma(t)|}\,dt}{\int_{\widehat S_n}{|\det D \gamma(t)|}\,dt}= \frac{|S_n\setminus S_n'|}{|S_n|}\stackrel{n\to\infty}{\longrightarrow}0.
\end{align*}

\subsection{Verification of (\ref{R:closure})}\label{subsec:R3}
The proof of the closure estimate \eqref{R:closure} goes back to the seminal works \cite{bdd,stevenson08}.
Our analysis builds on  \cite[Section 3]{morgenstern} which proves \eqref{R:closure} for the refinement strategy of \cite{bg}; see also Remark~\ref{rem:morgenstern}.
The following auxiliary result  states that $\refine(\cdot,\cdot)$ is equivalent to iterative refinement of one single  element.
For a mesh in the parameter domain  $\widehat \TT_\coarse\in\widehat\T$ and an arbitrary set $\widehat\MM_\coarse$,  we define $\refine (\widehat\TT_\coarse,\widehat\MM_\coarse):=\refine(\widehat\TT_\coarse,\widehat\MM_\coarse\cap\widehat\TT_\coarse)$ and  note that $\refine(\widehat\TT_\coarse,\emptyset)=\widehat\TT_\coarse$.

\begin{lemma}\label{lem:refine commutes}
Let $\widehat\TT_\coarse\in\widehat\T$ and $\widehat\MM_\coarse=\{\widehat T_1,\dots, \widehat T_n\}\subseteq\widehat\TT_\coarse$.
Then, it holds that
\begin{align}\label{eq:iterative refinement}
\refine(\widehat\TT_\coarse,\widehat\MM_\coarse)=\refine(\refine(\dots\refine(\widehat\TT_\coarse,\{\widehat T_1\})\dots, \{\widehat T_{n-1}\}),\{\widehat T_{n}\}).
\end{align}
\end{lemma}
\begin{proof}We only show that  $\refine(\widehat\TT_\coarse,\widehat\MM_\coarse)=\refine(\refine(\widehat\TT_\coarse,\{\widehat T_1\}),\widehat\MM_\coarse\setminus\{\widehat T_1\})$, and  then \eqref{eq:iterative refinement} follows by induction.
We define 
\begin{align*}
\widehat\TT_{(1)}:=\refine(\widehat\TT_\coarse,\{\widehat T_1\}),\quad \widehat\TT_{(2)}:=\refine(\widehat\TT_{(1)},\widehat\MM_\coarse\setminus\{\widehat T_1\}),\\
\widehat\MM_{(0)}:=\widehat\MM_{(0)}^{(0)}:=\{\widehat T_1\},\quad\widehat\MM_{(1)}:=\widehat\MM_{(1)}^{(0)}:=\widetilde\MM_{(1)}:=\widetilde\MM_{(1)}^{(0)}:=\widehat\MM_\coarse\setminus\{\widehat T_1\}.
\end{align*}
For $i\in\N_0$, we introduce the following notation which is conform with that of Algorithm~\ref{alg:refinement}:
\begin{align*}
\widehat\MM_{(0)}^{(i+1)}&:=\widehat\MM_{(0)}^{(i)}\cup \bigcup_{\widehat T\in \widehat\MM_{(0)}^{(i)}}
\NN^{\rm bad}_\coarse(\widehat T),\quad
\widehat\MM_{(1)}^{(i+1)}:=\widehat\MM_{(1)}^{(i)}\cup\bigcup_{T_\in\widehat\MM_{(1)}^{(i)}}
\NN^{\rm bad}_{(1)}(\widehat T),\\\widetilde{\MM}_{(1)}^{(i+1)}&:=\widetilde\MM_{(1)}^{(i)}\cup\bigcup_{\widehat T\in \widetilde\MM_{(1)}^{(i)}}
\NN^{\rm bad}_\coarse(\widehat T).
\end{align*}
Finally, we set
\begin{align*}
\widehat\MM_{(0)}^{\rm (end)}:=\bigcup_{i\in\N_0}\widehat\MM_{(0)}^{(i)}, \quad\widehat\MM_{(1)}^{(\rm end)}:=\bigcup_{i\in\N_0}\widehat\MM_{(1)}^{(i)},\quad\widetilde\MM_{(1)}^{(\rm end)}:=\bigcup_{i\in\N_0}\widetilde\MM_{(1)}^{(i)}.
\end{align*}
With these notations, we have
\begin{align*}
\widehat\TT_\coarse\setminus\widehat\TT_{(1)}=\widehat\MM_{(0)}^{(\rm end)},\quad\widehat\TT_{(1)}\setminus\widehat\TT_{(2)}=\widehat\MM_{(1)}^{(\rm end)},\\
\widehat\TT_\coarse\setminus\refine(\widehat\TT_{(1)},\widehat\MM_\coarse\setminus\{\widehat T_1\})=\widehat\MM_{(0)}^{(\rm end)}\cup\widetilde\MM_{(1)}^{(\rm end)}.
\end{align*}
To conclude the proof, we will  prove that  $\widehat\MM_{(0)}^{(\rm end)}\cup\widehat\MM_{(1)}^{(\rm end)}=\widehat\MM_{(0)}^{(\rm end)}\cup\widetilde\MM_{(1)}^{(\rm end)}$
To this end, we split the proof into three steps.

\textbf{Step 1:}
We first prove $\widehat\MM_{(1)}^{(\rm end)}\subseteq\widehat\TT_\coarse$ by induction.
Clearly, we have $\widehat\MM_{(1)}^{(0)}\subseteq\widehat\TT_\coarse$.
Now, let $i\in\N_0$ and suppose $\widehat\MM_{(1)}^{(i)}\subseteq\widehat\TT_\coarse$.
To see $\widehat\MM_{(1)}^{(i+1)}\subseteq\widehat\TT_\coarse$, we argue by contradiction and assume that there exists $\widehat T\in\widehat\MM_{(1)}^{(i)}$ and $\widehat T'\in\NN_{(1)}^{\rm bad}(\widehat T)\setminus\widehat\TT_\coarse$.
By Lemma~\ref{lem:active supports shrink}, the  unique father element $\widehat T_\coarse'\in\widehat\TT_\coarse$ with $\widehat T'\subsetneqq \widehat T_\coarse'$ satisfies  $\widehat T_\coarse'\in\NN_\coarse(\widehat T)$.
Therefore, admissibility of $\widehat\TT_\coarse$ proves $|\level(\widehat T)-\level(\widehat T_\coarse')|\le 1$, which contradicts 
\begin{align*}\level(\widehat T_\coarse')=\level(\widehat T')-1=\level(\widehat T)-2.\end{align*}

 \textbf{Step 2:} 
Let $\widehat T\in \widehat\MM_{(1)}^{(\rm end)}$. In this step, we will prove that 
\begin{align}\label{eq:M1ip2}
\widehat\MM_0^{\rm (end)}\cup\NN^{\rm bad}_{(1)}(\widehat T)=\widehat\MM_{(0)}^{\rm (end)}\cup\NN^{\rm bad}_\coarse(\widehat T).
\end{align}
By Step 1, we have $\widehat T\in\widehat\TT_\coarse$.
Lemma~\ref{lem:active supports shrink}  proves $\NN^{\rm bad}_{(1)}(\widehat T)\cap\widehat\TT_\coarse\subseteq\NN^{\rm bad}_{\coarse}(\widehat T)$.
Using Step~1 again, we see $\NN_{(1)}^{\rm bad}(\widehat T)\subseteq\widehat\MM_{(1)}^{\rm (end)}\subseteq \widehat\TT_\coarse$ and conclude ``$\subseteq$" in \eqref{eq:M1ip2}.
To see ``$\supseteq$", let $\widehat T'\in \NN^{\rm bad}_\coarse(\widehat T)\setminus\widehat\MM_{(0)}^{\rm (end)}$.
Note that $\widehat T'\in\widehat\TT_\coarse\cap\widehat\TT_{(1)}$ since $\widehat\TT_\coarse\setminus\widehat\TT_{(1)}=\widehat\MM_{(0)}^{\rm (end)}.$
There exists $\widehat  \beta\in\widehat  \HH_\coarse$ with $\widehat T,\widehat T'\subseteq\supp(\widehat \beta)$.
By admissibility of  $\widehat\TT_\coarse\in\widehat\T$, $\level(\widehat T')=\level(\widehat T)-1$, and \eqref{eq:level beta is}, we see that $\level(\widehat \beta)=\level(\widehat T')=:k'$. 
Hence, \eqref{eq:short cHH} yields that $\supp(\widehat \beta)\subseteq\widehat \Omega_\coarse^{k'}$ as well as $\supp(\widehat  \beta)\not\subseteq\widehat \Omega_\coarse^{k'+1}$.
The definition of $k'$ and \eqref{eq:parameter mesh} show that  $\widehat T'\not\subseteq\widehat\Omega_{(1)}^{k'+1}$.
We conclude $\supp(\widehat \beta)\subseteq\widehat\Omega_\coarse^{k'}\subseteq\widehat \Omega_{(1)}^{k'}$ and $\supp(\widehat \beta)\not\subseteq\widehat \Omega_{(1)}^{k'+1}$, since $\widehat\TT_{(1)}\ni \widehat T'\subseteq\supp(\widehat \beta)$.
Therefore, \eqref{eq:short cHH}  shows $\widehat \beta\in\widehat \HH_{(1)}$.
Altogether, we have $\widehat T'\in\NN^{\rm bad}_{(1)}(\widehat T)$.

 \textbf{Step 3:} 
Finally, we prove $\widehat\MM_{(0)}^{(\rm end)}\cup\widehat\MM_{(1)}^{(i)}=\widehat\MM_{(0)}^{(\rm end)}\cup\widetilde\MM_{(1)}^{(i)}$ by induction on $i\in\N_0$.
In particular, this will imply $\widehat\MM_{(0)}^{(\rm end)}\cup\widehat\MM_{(1)}^{(\rm end)}=\widehat\MM_{(0)}^{(\rm end)}\cup\widetilde\MM_{(1)}^{(\rm end)}$.
For $i=0$,  the claim follows from $\widehat\MM_{(1)}^{(0)}=\widetilde\MM_{(1)}^{(0)}$.
By Step 2,
the induction step works as follows:
\begin{align*}
\widehat\MM_{(0)}^{(\rm end)}\cup\widehat\MM_{(1)}^{(i+1)}&\,\,=\widehat\MM_{(0)}^{(\rm end)}\cup\widehat\MM_{(1)}^{(i)}\cup\bigcup_{\widehat T\in \widehat\MM_{(1)}^{(i)}}
\NN^{\rm bad}_{(1)}(\widehat T)\\
&\stackrel{\eqref{eq:M1ip2}}{=}\widehat\MM_{(0)}^{(\rm end)}\cup\widehat\MM_{(1)}^{(i)}\cup\bigcup_{\widehat T\in \widehat\MM_{(0)}^{(\rm end)}\cup \widehat\MM_{(1)}^{(i)}}
\NN^{\rm bad}_\coarse(\widehat T)\\
&\,\,=\widehat\MM_{(0)}^{(\rm end)}\cup\widetilde \MM_{(1)}^{(i)}\cup\bigcup_{\widehat T\in \widehat\MM_{(0)}^{(\rm end)}\cup\widetilde\MM_{(1)}^{(i)}}
\NN^{\rm bad}_\coarse(\widehat T)\\
&\,\,=\widehat\MM_{(0)}^{(\rm end)}\cup\widetilde\MM_{(1)}^{(i+1)}.
\end{align*}
This concludes the proof.
\end{proof}

Let $\widehat\TT_\coarse\in\widehat\T$. 
For $\widehat  T,\widehat T'\in\widehat\TT_\coarse$, let $\dist(\widehat T,\widehat T')$ be the Euclidean distance of their midpoints in the parameter domain.
Let $\widehat T\in \widehat\TT_\coarse$ and  $\widehat T'\in\NN_\coarse(\widehat T)$.
Hence, there is $\widehat\beta\in\widehat\HH_\coarse$  such that $\widehat T,\widehat T'\subseteq\supp(\beta)$.
In particular, it holds that $\dist(\widehat T,\widehat T')\le\diam(\supp(\widehat\beta))$.
By admissibility  of $\widehat\TT_\coarse$ and \eqref{eq:level beta is}, we see $|\level(\widehat\beta)-\level(\widehat T)|\le 1$.
This proves
\begin{align}\label{eq:distineq}
\dist(\widehat T,\widehat T')\le \C{diam} 2^{-\level(\widehat T)},
\end{align}
where $\C{diam}>0$  depends only on $d$, ${\widehat\TT}_0$ and $(p_1,\dots,p_d)$.
With this observation, we can prove the following lemma.
The proof follows the lines of \cite[Lemma 11]{morgenstern}, but is also included here  for completeness.
\begin{lemma}\label{lem:morgenstern}
Let $\widehat\TT_\coarse\in\widehat\T$ and $\widehat T'\in\widehat\TT_\coarse$.
With $\widehat\TT_\circ=\refine(\widehat\TT_\coarse,\{\widehat T'\})$, it holds that 
\begin{align}
\dist(\widehat T,\widehat T')\le 2^{-\level(\widehat T)}\Cdist \quad \text{for all }\widehat T\in \widehat\TT_\circ\setminus\widehat\TT_\coarse, 
\end{align}
where $\Cdist>0$ depends only on $d$, ${\widehat\TT}_0$ and $(p_1,\dots,p_d)$.
\end{lemma}
\begin{proof}
$\widehat T\in \widehat\TT_\circ\setminus\widehat\TT_\coarse$ implies the existence of a sequence $\widehat T'=\widehat T_J, \widehat T_{J-1},\dots, \widehat T_0$ such that $\widehat T_{j-1}\in\NN^{\rm bad}_\coarse(\widehat T_j)$ and $\widehat T$ is a child of $\widehat T_0$, i.e., $\widehat T\subsetneqq\widehat  T_0$ and $\level(\widehat T)=\level(\widehat T_0)+1$.
Since $\level(\widehat T_{j-1})=\level(\widehat T_j)-1$, it follows 
\begin{align}\label{eq:level Tj}
\level(\widehat T_j)=\level(\widehat T_0)+j.
\end{align}
The triangle inequality proves 
\begin{align*}
\dist(\widehat T,\widehat T')\le \dist(\widehat T,\widehat T_0)+\dist(\widehat T_0,\widehat T')\le\dist(\widehat T,\widehat T_0) +\sum_{j=1}^J\dist(\widehat T_j,\widehat T_{j-1})
\end{align*}
Further, there exists a constant $C>0$ which depends only on $\widehat{\TT}_0$ and $d$, such that
\begin{align*}
\dist(\widehat T,\widehat T_0)\le C 2^{-\level(\widehat T)}.
\end{align*}
With \eqref{eq:distineq} and \eqref{eq:level Tj}, we see
\begin{align*}
\sum_{j=1}^J\dist(\widehat T_j,\widehat T_{j-1})&\stackrel{\eqref{eq:distineq}}{\le}\C{diam}\sum_{j=1}^J 2^{-\level(\widehat T_j)}\stackrel{\eqref{eq:level Tj}} =\C{diam}\sum_{j=1}^J 2^{-\level(\widehat T_0)-j} {\le} \C{diam}2^{-\level(\widehat T)-1},
\end{align*}
which concludes the proof.
\end{proof}

Finally, let $\widehat\TT_\coarse\in\widehat\T$ and $\widehat T\in \widehat\TT_\coarse$.
We abbreviate $\widehat\TT_\circ=\refine(\widehat\TT_\coarse,\{\widehat T\})$.
Then, there holds 
\begin{align}\label{eq:remark morgenstern}
\level(\widehat T')\le\level(\widehat T)+1\quad\text{for all refined elements }\widehat T'\in\widehat\TT_\circ\setminus\widehat\TT_\coarse.
\end{align}
To see this, note that all elements $\widehat T''\in\widehat\TT_\coarse\setminus\widehat\TT_\circ$ which are refined, satisfy $\widehat T''=\widehat T$ or $\level(\widehat T'')\le\level(\widehat T)-1$.
Therefore, their children satisfy $\level(\widehat T')\le\level(\widehat T)+1$.
With this last observation, we can argue as in the proof of \cite[Theorem 12]{morgenstern} to show the closure estimate \eqref{R:closure}.
The constant $\Cclos>0$ depends only on $d,\widehat\TT_0$, and $(p_1,\dots,p_d)$.

\subsection{Verification of (\ref{R:overlay})}
We prove  \eqref{R:overlay} in the parameter domain $\widehat\Omega$. 
  Let $\widehat \TT_\coarse,\widehat\TT_\star\in\widehat\T$ be two admissible hierarchical meshes.
We define the overlay
\begin{align}
\widehat\TT_{\fine}:=\set{\widehat T_\coarse\in\widehat\TT_\coarse}{\exists \widehat T_\star\in\widehat\TT_\star\quad \widehat T_\coarse\subseteq \widehat T_\star}\cup\set{\widehat T_\star\in\widehat\TT_\star}{\exists\widehat T_\coarse \in\widehat\TT_\coarse\quad\widehat T_\star\subseteq \widehat T_\coarse}.
\end{align}
Note that $\widehat\TT_\fine$ is a hierarchical mesh with hierarchical domains ${\widehat\Omega}_\fine^k={\widehat\Omega}_\coarse^k\cup \widehat\Omega_\star^k$ for $k\in\N_0$.
In particular, $\widehat\TT_\fine$  is finer than $\widehat\TT_\coarse$ and $\widehat\TT_\star$.
Moreover, the overlay estimate  easily follows from the definition of $\widehat\TT_\fine$.
It remains to prove that $\widehat\TT_\fine$ is admissible.
To see this, let $\widehat T,\widehat T'\in\widehat\TT_\fine$ with $\widehat T'\in\NN_\fine(\widehat T)$, i.e., there exists $\widehat\beta_\fine\in\widehat\HH_\fine$ such that $|\widehat T\cap \supp(\widehat\beta_\fine)|\neq 0\neq| \widehat T'\cap \supp(\widehat\beta_\fine)|$.
Without loss of generality, we suppose $\level(\widehat T)\ge\level(\widehat T')$ and $\widehat T\in\widehat\TT_\coarse$.
If $\widehat T'\in\widehat\TT_\coarse$, Lemma~\ref{lem:active supports shrink}  shows $\widehat T'\in\NN_\coarse(\widehat T)$, and admissibility of $\widehat\TT_\coarse$ implies $|\level(\widehat T)-\level(\widehat T')|\le 1$.
Now, let $\widehat T'\in\widehat\TT_\star$.
By definition of the overlay, there exists $\widehat T_\coarse'\in\widehat\TT_\coarse$ with $\widehat T'\subseteq \widehat T_\coarse'$ and $\level(\widehat T_\coarse')\le\level(\widehat T')$.
Further,  Lemma~\ref{lem:active supports shrink} provides some (not necessarily unique) $\widehat\beta_\coarse\in\widehat\HH_\coarse$ such that $\supp(\widehat\beta_\fine)\subseteq\supp(\widehat\beta_\coarse)$.
Hence, $|\widehat T\cap\supp(\widehat\beta_\coarse)|\neq0\neq |\widehat T_\coarse'\cap\supp(\widehat\beta_\coarse)|$, i.e., $\widehat T_\coarse'\in\NN_\coarse(\widehat T)$.
Since  $\widehat\TT_\coarse\in\widehat\T$, it follows that  $|\level(\widehat T)-\level(\widehat T_\coarse')|\le 1$.
Altogether, we see that  
\begin{align*}|\level(\widehat T)-\level(\widehat T')|=\level(\widehat T)-\level(\widehat T')&\le \level(\widehat T)-\level(\widehat T_\coarse')\\
&\le |\level(\widehat T)-\level(\widehat T_\coarse')|\le 1.
\end{align*}
This concludes the proof of \eqref{R:overlay}.



\subsection{Verification of (\ref{S:inverse})}\label{subsec:E1.1 true}
Let $T\in\TT_\coarse\in\T$.
Let $V_\coarse\in\XX_\coarse$.
Define $\widehat V_\coarse:=V_\coarse\circ\gamma\in\widehat\XX_\coarse\subseteq\widehat\YY_\coarse$ and $\widehat T:=\gamma^{-1}(T)\in\widehat\TT_\coarse$.
Regularity \eqref{eq:Cgamma} of $\gamma$ proves for $i\in\{0,1,2\}$ 
\begin{align}\label{eq:sobolev equivalent}
\norm{V_\bullet}{H^i(T)}
\simeq\norm{\widehat V_\bullet}{H^i(\widehat T)},
\end{align}
where the hidden constants depend only on $d$ and $C_\gamma$.
Since $\widehat V_\coarse$ is a $\widehat \TT_\coarse$-piecewise tensor polynomial, there holds for $i,j\in\{0,1,2\}$ with $j\le i$ that
\begin{align}\label{eq:parameter invest}
|\widehat T|^{(i-j)/d} \norm{\widehat V_\bullet}{H^i(\widehat T)}\lesssim \norm{\widehat V_\bullet}{H^j(\widehat T)},
\end{align}
where  the hidden constant  depends only on $d$, $\widehat\TT_0$, and $(p_1,\dots,p_d)$.
Together, \eqref{eq:sobolev equivalent}--\eqref{eq:parameter invest} conclude the proof of \eqref{S:inverse}, where $C_{\rm inv}$ depends only on $d$, $C_\gamma$, $\widehat\TT_0$,  and $(p_1,\dots,p_d)$.

\subsection{Verification of (\ref{S:nestedness})}\label{subsec:nestedness}
 In \eqref{eq:hierarchical nested}, we already saw that $\TT_\fine\in\refine(\TT_\coarse)$ with $\TT_\coarse\in\T$ implies nestedness of the corresponding ansatz spaces $\XX_\coarse\subseteq\XX_\fine$.

\subsection{Verification of (\ref{S:local})}\label{subsec:E4.2 true}
We show the assertion in the parameter domain.
For arbitrary but fixed $\kproj\in\N_0$ (which will be fixed later in Section~\ref{subsec:E4.1 true} to be $k_{\rm proj}:= 2(p+1)$), we  set $k_{\rm loc}:=k_{\rm proj}+2(p+1)$. 
Let $\widehat \TT_\bullet\in\widehat\T$,  $\widehat\TT_\circ\in\refine(\widehat\TT_\bullet)$, and $\widehat V_\circ\in\widehat\XX_\circ$.
We define the patch functions $\pi_\bullet$ and $\Pi_\bullet$ in the parameter domain analogously to the patch functions in the physical domain, see Section \ref{subsec:general mesh}.
Let $\widehat T\in\widehat\TT_\bullet\setminus\Pi_\bullet^{{\rm loc}}(\widehat\TT_\bullet\setminus\widehat\TT_\circ)$, where $\Pi_\bullet^{\rm loc}:=\Pi_\bullet^{k_{\rm loc}}$.
First, we show that
\begin{align}\label{eq:omegainv}
\Pi_\bullet^{{\rm loc}}(\widehat T)\subseteq \widehat\TT_\bullet\cap\widehat\TT_\circ.
\end{align}
To this end, we argue by contradiction and assume  that there exists $\widehat T'\in\Pi_\bullet^{{\rm loc}}(\widehat T)$ with $\widehat T'\not \in \widehat\TT_\bullet\cap\widehat\TT_\circ$.
This is equivalent to $\widehat T\in\Pi_\bullet^{{\rm loc}}(\widehat T')$ and  $\widehat T'\in\widehat\TT_\bullet\setminus\widehat\TT_\circ$.
This implies $\widehat T\in\Pi_\bullet^{{\rm loc}}(\widehat\TT_\bullet\setminus\widehat\TT_\circ)$,  contradicts $\widehat T\in\widehat\TT_\bullet\setminus\Pi_\bullet^{{\rm loc}}(\widehat\TT_\bullet\setminus\widehat\TT_\circ)$, and hence proves \eqref{eq:omegainv}.
Again, we abbreviate $\pi_\bullet^{\rm proj}:=\pi_\bullet^{k_{\rm proj}}$.
According to Corollary~\ref{cor:basis of X}, 
it holds that
\begin{align*}
\set{V_\coarse|_{\pi_\bullet^{\rm proj}(T)}}{V_\coarse\in \XX_\coarse}={\rm span}\set{\widehat\beta|_{\pi_\bullet^{\rm proj}(T)}}{\widehat \beta\in\widehat \HH_\bullet\wedge\widehat \beta|_{\partial\widehat \Omega}=0\wedge|\supp(\widehat \beta)\cap\pi_\bullet^{\rm proj}(\widehat T)|>0},
\end{align*}
as well as
\begin{align*}
\set{V_\fine|_{\pi_\bullet^{\rm proj}(T)}}{V_\fine\in \XX_\fine}={\rm span}\set{\widehat\beta|_{\pi_\bullet^{\rm proj}(T)}}{\widehat \beta\in\widehat \HH_\fine\wedge\widehat \beta|_{\partial\widehat \Omega}=0\wedge|\supp(\widehat \beta)\cap\pi_\bullet^{\rm proj}(\widehat T)|>0}.
\end{align*}
We will prove 
\begin{align}\label{eq:local ansatz prove}
\set{\widehat \beta\in\widehat \HH_\bullet}{|\supp(\widehat \beta)\cap\pi_\bullet^{\rm proj}(\widehat  T)|>0}=\set{\widehat \beta\in\widehat \HH_\circ}{|\supp(\widehat \beta)\cap\pi_\bullet^{\rm proj}(\widehat T)|>0},
\end{align}
which will conclude \eqref{S:local}.
First let $\widehat \beta$ be an element of the left set.
By Remark~\ref{rem:connected}, this implies $\supp(\widehat \beta)\subseteq\pi_\bullet^{{\rm loc}}(\widehat T)$.
Together with \eqref{eq:omegainv}, we see $\supp(\widehat \beta)\subseteq \bigcup(\widehat\TT_\bullet\cap \widehat\TT_\circ)$.
This proves that no element within $\supp(\widehat \beta)$ is changed during refinement, i.e., $\widehat \Omega_\bullet^k\cap\supp(\widehat \beta)=\widehat \Omega_\circ^k\cap\supp(\widehat \beta)$ for all $k\in\N_0$.
Thus, \eqref{eq:short cHH}  proves $\widehat \beta\in\widehat \HH_\circ$.
The proof works the same if we start with some $\widehat \beta$ in the right set.
This proves \eqref{eq:local ansatz prove} and therefore \eqref{S:local}.


\subsection{Truncated hierarchical B-splines}\label{subsec:trunc}
We will define some Scott-Zhang type operator $J_\bullet$ similarly as in \cite{speleers} with the help of so-called truncated hierarchical B-splines (THB-splines) introduced in \cite{juttler2}.
In this section, we recall their definition and list some basic properties.
For a more detailed  presentation, we refer to, e.g., \cite{juttler2,speleers}.

Let $\widehat\TT_\bullet$ be an arbitrary hierarchical mesh in the parameter domain.
For $k\in\N_0$, we define the truncation $\trunc_\coarse^{k+1}:\widehat\YY^k\to\widehat\YY^{k+1}$ as follows:
\begin{align}
\trunc_\bullet^{k+1}(\widehat V^k):=\sum_{\widehat \beta\in\widehat\BB^{k+1}\atop \supp(\widehat\beta)\not\subseteq\widehat\Omega_\bullet^{k+1}} c_{\widehat\beta} \widehat\beta
\quad\text{for all }\widehat V^k=\sum_{\widehat\beta\in\widehat\BB^{k+1}} c_{\widehat\beta}\widehat\beta\in\widehat\YY^k\subset\widehat\YY^{k+1},
\end{align}
i.e., truncation is defined via the (unique) basis representation of $\widehat V^k\in\widehat\YY^k$ with respect $\widehat\BB^{k+1}$.
Recall that $M_\bullet\in\N$ is the minimal integer such that $\widehat\Omega_\bullet^{M_\bullet}=\emptyset$.
For all  $\widehat\beta\in\widehat\HH_\bullet$, the corresponding truncated hierarchical B-spline reads
\begin{align}
\Trunc_\bullet(\widehat\beta):=\trunc_\bullet^{M_\bullet-1}\Big(\trunc_\bullet^{M_\bullet-2}\Big(\dots\Big(\trunc_\bullet^{\level(\widehat\beta)+1}(\widehat \beta)\Big)\dots\Big)\Big),
\end{align}
As the set $\widehat\HH_\bullet$, the set of  THB-splines $\set{\Trunc_\bullet(\widehat\beta)}{\widehat\beta\in\widehat\HH_\bullet}$ forms a basis of the space of hierarchical splines $\widehat\YY_\bullet$.
In Section \ref{subsec:splines}, we mentioned that each basis function in $\widehat\BB^k$ is the linear combination of basis functions $\widehat\BB^{k+1}$, where the corresponding coefficients are nonnegative; see \eqref{eq:twoscale1}--\eqref{eq:twoscale}.
For $\widehat\beta\in\widehat\HH_\bullet$, this proves
\begin{align}\label{eq:bounds for Trunc}
0\le \Trunc_\bullet(\widehat\beta)\le\widehat\beta\le 1,
\end{align}
and in particular $\supp(\Trunc_\bullet(\widehat\beta))\subseteq \supp(\widehat\beta)$.
With this and the fact that the THB-splines are a basis of $\widehat\YY_\bullet$, Corollary~\ref{cor:basis of X} proves
\begin{align}\label{eq:Truncbasis of X}
\widehat\XX_\bullet={\rm span} \set{\Trunc_\bullet(\widehat\beta)}{\widehat\beta\in\widehat\HH_\bullet\wedge \widehat\beta|_{\partial\widehat\Omega}=0},
\end{align}
where the set on the right-hand side  is even a basis of $\widehat\XX_\coarse$.
The following proposition shows that for an admissible mesh $\widehat\TT_\coarse\in\widehat\TT_\coarse$, the full truncation $\Trunc_\coarse$ reduces to  $\trunc_\bullet^{\level(\widehat\beta)+1}$.
\begin{proposition}\label{prop:trunc}
Let $\widehat\TT_\bullet\in\widehat\T$ and $\widehat\beta\in\widehat\HH_\bullet$.
Then, it holds that
\begin{align}\label{eq:trunc one}
\Trunc_\bullet(\widehat\beta)=\trunc^{\level(\widehat\beta)+1}_\bullet(\widehat\beta).
\end{align}
\end{proposition}
\begin{proof}
We prove the proposition in three steps.

 \textbf{Step 1:}
Let $k'< k''\in\N_0$ and  $\widehat\beta'\in\widehat\BB^{k'}$ with representation 
$\widehat\beta'=\sum_{\widehat\beta''\in\widehat\BB^{k''}}c_{\widehat\beta''}\widehat\beta''$.
Let $\widehat\beta''\in\widehat\BB^{k''}$  such that  $c_{\widehat\beta''}\neq 0$.
Then, local linear independence (with the open set $O=(0,1)^d \setminus\supp(\widehat\beta')$) of $\widehat\BB^{k''}$ implies $\supp(\widehat\beta'')\subseteq\supp(\widehat\beta')$.

\textbf{Step 2:}
We prove \eqref{eq:trunc one}.
We abbreviate $k=\level(\widehat\beta)$.
Let $\widehat\beta=\sum_{\widehat\beta'\in\widehat\BB^{k+1}}c_{\widehat\beta'}\widehat\beta'$.
Let $\widehat\beta'\in\widehat\BB^{k+1}$ with $\supp(\widehat\beta')\not\subseteq\widehat\Omega_\bullet^{k+1}$ and $c_{\widehat\beta'}\neq 0$.
By Step 1, this proves $\supp(\widehat\beta')\subseteq\supp(\widehat\beta)$. 
For $k''>k+1$, we consider  the representation 
\begin{align*}
\trunc_\bullet^{k''}(\widehat\beta') =\sum_{\widehat\beta''\in\widehat\BB^{k''}\atop \supp(\widehat\beta'')\not\subseteq\widehat\Omega_\bullet^{k''}}c_{\widehat\beta''}\widehat\beta'',\quad\text{where } \widehat\beta'=\sum_{\widehat\beta''\in\widehat\BB^{k''}}c_{\widehat\beta''}\widehat\beta''.
\end{align*}
If $\widehat\beta''\in\widehat\BB^{k''}$ with $\supp(\widehat\beta'')\subseteq \widehat\Omega_\bullet^{k''}$, let $\widehat T''\in\widehat\TT^{k''}$ with $\widehat T''\subseteq\supp(\widehat\beta)$. \eqref{eq:parameter mesh} shows the  existence of an element $\widehat T\in\widehat\TT_\bullet$ with $\level(\widehat T)\ge k''$ such that $\widehat T\subseteq \widehat T''$.
To see $c_{\widehat \beta''}=0$, we argue by contradiction and assume $c_{\widehat\beta''}\neq 0$.
By Step 1, this  implies $\widehat T\subseteq\supp(\widehat\beta'')\subseteq\supp(\widehat\beta')\subseteq\supp(\widehat\beta)$.
Due to \eqref{eq:level beta is}, this contradicts admissibility of $\widehat\TT_\bullet$.
This proves $c_{\widehat\beta''}=0$.
Overall, we  conclude $\trunc_\bullet^{k''}(\widehat\beta')=\widehat\beta'$, and thus
$\trunc_\bullet^{k''}(\trunc_\bullet^{k+1}(\widehat\beta))=\trunc_\bullet^{k+1}(\widehat\beta)$ as well as \eqref{eq:trunc one}.

\end{proof}
\begin{remark}\label{rem:trunc}
Actually, the proposed refinement strategy of Algorithm~\ref{alg:refinement} was designed for hierarchical B-splines; see also Proposition~\ref{prop:bounded number}.
However, \eqref{eq:bounds for Trunc} implies that  Proposition~\ref{prop:bounded number} holds accordingly for truncated hierarchical B-splines.
Moreover,  if one applies the refinement strategy of Algorithm~\ref{alg:refinement},
\eqref{eq:trunc one} shows that the computation of the truncated hierarchical B-splines greatly simplifies.\qed%
\end{remark}



\subsection{Verification of (\ref{S:proj})--(\ref{S:grad})}\label{subsec:E4.1 true}
Given $\TT_\bullet\in\T$, we are finally able to introduce a suitable Scott-Zhang-type operator $J_\bullet:H_0^1(\Omega)\to\XX_\bullet$  which satisfies \eqref{S:proj}--\eqref{S:grad}.
To this end, it is sufficient to construct a corresponding operator $\widehat J_\bullet:H_0^1(\widehat\Omega)\to\widehat\XX_\bullet$ in the parameter space, and to define 
\begin{align}
J_\bullet v:=\big(\widehat J_\bullet (v\circ\gamma)\big)\circ\gamma^{-1}\quad\text{for all }v\in H_0^1(\Omega).
\end{align}
By regularity \eqref{eq:Cgamma} of $\gamma$, the properties \eqref{S:proj}--\eqref{S:grad} immediately transfer from the parameter domain $\widehat\Omega$ to the physical domain $\Omega$.
Recall that $\widehat\BB^k\cap\widehat\BB^{k'}=\emptyset$ for $k\neq k'$.
For $k\in\N_0$ and $\widehat\beta\in\widehat\BB^k$, let $\widehat T_{\widehat\beta}\in \widehat\TT^k$ be an arbitrary but fixed element with $\widehat T_{\widehat\beta}\subseteq\supp(\widehat\beta)$. 
If $\widehat\beta\in\widehat\HH_\bullet$, we additionally require\footnote{Therefore, the elements $\widehat T_{\widehat\beta}$ depend additionally on the considered mesh $\widehat\TT_\bullet$.} $\widehat T_{\widehat\beta}\in\widehat\TT_\bullet$, which is possible due to \eqref{eq:level beta is}.
By local linear independence and continuity
 of $\widehat \BB^k$ (see Section \eqref{subsec:splines}), also the restricted basis functions $\set{\widehat\beta|_{\widehat T_{\widehat\beta}}}{\widehat\beta\in\widehat \BB^k\wedge\widehat\beta|_{\widehat T_{\widehat\beta}}\neq 0}$ are linearly independent.  Hence, the Riesz  theorem guarantees the existence and uniqueness of some  $\widehat\beta^*\in\set{\widehat V^k|_{\widehat T_{\widehat\beta}}}{\widehat V^k\in\widehat\YY^k}$ such that
\begin{align}\label{eq:dual}
\int_{\widehat T_{\widehat\beta}}\widehat\beta^*\widehat \beta'\,dt= \delta_{\widehat\beta,\widehat\beta'}\quad\text{for all }\widehat\beta'\in\widehat\BB^k.
\end{align}
These dual basis functions $\widehat\beta^*$ satisfy the following scaling property.
\begin{lemma}\label{lem:dual}
There exists  $\C{dual}>0$  such that for all $k\in\N_0$ and all $\widehat\beta\in\widehat\BB^k$, it holds that 
\begin{align}\label{eq:dual bound}
\norm{\widehat\beta^*}{L^\infty(\widehat T_{\widehat\beta})}\le \C{dual} |\widehat T_{\widehat\beta}|^{-1}.
\end{align}
The constant $\C{dual}$ depends only on $d$, $\widehat\TT_0$ and $(p_1,\dots,p_d)$.
\end{lemma}
\begin{proof}
Recall that $\widehat T_{\widehat\beta}$ is a rectangle of the form $\widehat T^k_{\ell_1,\dots,\ell_d}=[t^k_{1,\ell_1},t^k_{1,\ell_1+1}]\times\dots\times [t^k_{d,\ell_d},t^k_{d,\ell_d+1}]$.
We abbreviate $C:=|\widehat T_{\widehat\beta}|^{1/d}$, $(a_1,\dots,a_d):=(t_{1,\ell_1}^k,\dots,t_{d,\ell_d}^k)$ and define the normalized element $\widetilde T_{\widehat\beta}:=(\widehat T_{\widehat\beta}-(a_1,\dots,a_d))/C$ and the corresponding affine transformation $\Phi:  \widetilde T_{\widehat\beta}\to \widehat T_{\widehat\beta}$.
We apply the transformation formula to see
$\int_{\widehat T_{\widehat\beta}}\widehat\beta^*\widehat \beta'\,dt=C^d \int_{\widetilde T_{\widehat\beta}} (\widehat\beta^*\circ\Phi)( \widehat\beta\circ\Phi)\,dt.$
Therefore, the Riesz  theorem implies that $\widehat\beta^*=(\widetilde\beta^*\circ\Phi^{-1})/C^d$, where $\widetilde \beta^*$ is the unique element in $\widetilde\BB^k:=\set{\widehat\beta'\circ\Phi}{\widehat\beta'\in\widehat\BB^k}\setminus\{0\}$ such that 
\begin{align*}
\int_{\widetilde T_{\widehat\beta}}\widetilde\beta^*\widetilde \beta'\,dt= \delta_{\widetilde\beta,\widetilde\beta'}\quad\text{for all }\widetilde\beta'\in\widetilde\BB^k.
\end{align*}
By \eqref{eq:B-spline}, each $\widetilde\beta'\in\widetilde\BB^k$ is of the form 
\begin{align*}
\widetilde\beta'(\widetilde s_1,\dots,\widetilde s_d)=\prod_{i=1}^d B(s_i|t_{i,j_i}^k,\dots,t_{i,j_i+p_i+1}^k)\text{ with }(s_1,\dots,s_d)=(\widetilde s_1,\dots,\widetilde s_d)C+(a_1,\dots,a_d).
\end{align*}
We only have to consider $\widetilde\beta'$ that are supported on $\widetilde T_{\widehat\beta}$.
As the support of any $B(\cdot|t_{i,j_i}^k,\dots,t_{i,j_i+p_i+1}^k)$ is just $[t_{i,j_i}^k,t_{i,j_i+p_i+1}^k]$, it is sufficient to consider $j_i=\ell_i-p_i,\dots,\ell_i$.
By the definition of B-splines, one immediately sees that an affine transformation in the parameter domain can just be passed to the knots, i.e.,
\begin{align*}
B(s_i|t_{i,j_i}^k,\dots,t_{i,j_i+p_i+1}^k)=B\big(\widetilde s_i\big|(t_{i,j_i}^k-a_1)/C,\dots,(t_{i,j_i+p_i+1}^k-a_d)/C\big).
\end{align*}
Altogether, we see that $\widetilde\beta^*$ depends only on the  knots 
\begin{align*}
\Big(\frac{t_{i,j_i}^k-a_1}{C},\dots,\frac{t_{i,j_i+p_i+1}^k- a_d}{C}:i=1,\dots,d, j_i=\ell_i-p_i,\dots,\ell_i\Big).
\end{align*}
Since we only use  global dyadic bisection between two consecutive levels, we see that these knots depend only on $d$, $\widehat\TT_0$ and $(p_1,\dots,p_d)$ but not on the level $k$.
This shows $\norm{\widetilde\beta}{L^\infty(\widehat T_{\widehat\beta})}\lesssim 1$, where the hidden constant depends only on $d$, $\widehat\TT_0$ and $(p_1,\dots,p_d)$.
\end{proof}
We adopt the approach of \cite{speleers}.
For $\widehat v\in L^2(\widehat\Omega)$, we abbreviate $\dual{\widehat\beta^*}{\widehat v}:=\int_{\widehat T_{\widehat\beta}}\widehat \beta^* \widehat v\,dt$ and define
\begin{align}
&\widehat J^k:L^2(\widehat \Omega)\to \widehat\YY^k, \quad \widehat J^k\widehat v:= \sum_{\widehat\beta\in\widehat\BB^k}  \dual{\widehat\beta^*}{\widehat v} \,\widehat\beta,\\
&\widehat J_\bullet:L^2(\widehat\Omega)\to \widehat\XX_\bullet,\quad \widehat J_\coarse \widehat v:= \sum_{k=0}^{M_\bullet-1}\sum_{\widehat\beta\in\widehat\HH_\bullet\wedge \widehat\beta|_{\partial\widehat\Omega}= 0\atop\level(\widehat\beta)=k}\dual{\widehat\beta^*}{\widehat v} \,\Trunc_\bullet(\widehat\beta).
\end{align}
Before we prove the properties \eqref{S:proj}--\eqref{S:grad}, we collect some basic properties of $\widehat J_\bullet$.

\begin{lemma} \label{lem:Scott}
Let $\widehat\TT_\bullet\in\widehat\T$.
Then, $\widehat J_\bullet$ is a projection, i.e.,    
\begin{align}\label{eq:projection}
\widehat J_\bullet\widehat V_\bullet=\widehat V_\bullet\quad\text{for all }\widehat V_\bullet\in\XX_\bullet.
\end{align}
Moreover, $\widehat J_\bullet$ is locally $L^2$-stable, i.e., there exists $C_J>0$ such that for all $\widehat T\in \widehat\TT_\bullet$
\begin{align}\label{eq:L2stability}
\norm{\widehat J_\bullet\widehat v}{L^2(\widehat T)}\le C_J \norm{\widehat v}{L^2(\pi_\bullet^{2(p+1)}(\widehat T))}\quad\text{for all }\widehat v\in L^2(\widehat\Omega).
\end{align}
The constant $C_J$  depends only  on $d$, $\widehat\TT_0$, and $(p_1,\dots,p_d)$.
\end{lemma}
\begin{proof}
We prove the lemma in three steps.

\textbf{Step 1:}
The projection property \eqref{eq:projection} can be proved  as in \cite[Theorem 4]{speleers}.
There, a corresponding projection onto $\widehat\YY_\bullet$ instead of $\widehat\XX_\bullet=\widehat\YY_\bullet\cap H_0^1(\widehat\Omega)$ is considered.
However, with \eqref{eq:Truncbasis of X}, the proof works exactly the same.


\textbf{Step 2:}
We prove \eqref{eq:L2stability}.
The triangle inequality proves that 
\begin{align*}
\norm{\widehat J_\bullet\widehat v}{L^2(\widehat T)}\le \sum_{k=0}^{M_\bullet-1}\sum_{\substack{\widehat\beta\in\widehat\HH_\bullet\wedge \widehat\beta|_{\partial\widehat\Omega}= 0\\Ê|\supp(\widehat\beta)\cap \widehat T|>0\atop\level(\widehat\beta)=k}}\norm{\widehat \beta^*}{L^2(\widehat T_{\widehat\beta})}\norm{\widehat v}{L^2(\widehat T_{\widehat\beta})}\norm{\Trunc_\bullet(\widehat\beta)}{L^2(\widehat T)}.
\end{align*}
By Remark~\ref{rem:connected}, it holds that $\supp(\widehat\beta)\subseteq \pi_\bullet^{2(p+1)}(\widehat T)$ if $\widehat\beta\in\widehat\HH_\bullet$ with $|\supp(\widehat\beta)\cap \widehat T|>0$.
Therefore, we obtain that
\begin{align*}
\norm{\widehat J_\coarse \widehat v}{L^2(\widehat T)}\le \norm{\widehat v}{L^2(\pi_\bullet^{2(p+1)}(\widehat T))}\sum_{\substack{\widehat\beta\in\widehat\HH_\bullet\\ \supp(\widehat\beta)\subseteq\pi_\bullet^{2(p+1)}(\widehat T)}}\norm{\widehat \beta^*}{L^2(\widehat T_{\widehat\beta})}\norm{\Trunc_\bullet(\widehat\beta)}{L^2(\widehat T)}.
\end{align*}
We consider the set $\set{\widehat\beta\in\widehat\HH_\bullet}{\supp(\widehat\beta)\subseteq\pi_\bullet^{2(p+1)}(\widehat T)}$.
Since the support of each basis function in $\widehat\HH_\bullet$ consists of elements in $\widehat\TT_\bullet$, see \eqref{eq:supp elements}, this set is a subset of $\set{\widehat\beta\in\widehat\HH_\bullet}{\exists \widehat T'\in\Pi_\bullet^{2(p+1)}(\widehat T)\text{ with }\widehat T'\subseteq \supp(\widehat\beta)}$.
By (M2) and Proposition \ref{prop:bounded number}, the cardinality of the latter set is bounded by a constant $C>0$ that depends only on $d$ and $(p_1,\dots,p_d)$.
With \eqref{M:shape}, \eqref{eq:bounds for Trunc}, and \eqref{eq:dual bound}, we see that for $\widehat\beta\in\widehat\HH_\bullet$ with $\supp(\widehat\beta)\subseteq\pi_\bullet^{2(p+1)}(\widehat T)$, it holds that 
\begin{align*}
\norm{\widehat \beta^*}{L^2(\widehat T_{\widehat\beta})}\norm{\Trunc_\bullet(\widehat\beta)}{L^2(\widehat T)}\le |\widehat T_{\widehat\beta}|^{1/2}|\widehat T|^{1/2} \norm{\widehat\beta^*}{L^\infty (\widehat T_{\widehat \beta})}\lesssim 1.
\end{align*} 
The hidden constant  depends only on $d$, $\widehat\TT_0$, and $(p_1,\dots,p_d)$.

\end{proof}

We prove \eqref{S:proj} in the parameter domain.
Let $\widehat T\in\widehat\TT_\bullet, \widehat v\in H_0^1(\widehat\Omega)$, and $\widehat V_\bullet\in \widehat\XX_\bullet$ such that 
$\widehat v|_{\pi^{\rm proj}_\bullet(\widehat T)}=\widehat V_\bullet|_{\pi^{\rm proj}_\bullet(\widehat T)}$ with $k_{\rm proj}:= 2(p+1)$.
Remark~\ref{rem:connected} shows that for $\widehat\beta\in\widehat\HH_\bullet$ with $|\supp(\widehat\beta)\cap\widehat T|>0$, it holds that  $\supp(\widehat\beta)\subseteq\pi_\bullet^{{\rm proj}}(\widehat T)$.
With this, \eqref{eq:bounds for Trunc}, and the projection property \eqref{eq:projection} of $\widehat J_\bullet$, we conclude that 
\begin{eqnarray*}
(\widehat J_\bullet \widehat v)|_{\widehat T}&=&\sum_{k=0}^{M_\bullet-1}\sum_{\widehat\beta\in\widehat\HH_\bullet\wedge \widehat\beta|_{\partial\widehat\Omega}= 0\atop\level(\widehat\beta)=k}\dual{\widehat\beta^*}{\widehat v} \,\Trunc_\bullet(\widehat\beta)|_{\widehat T}\\
&\stackrel{\eqref{eq:bounds for Trunc}}{=}&\sum_{k=0}^{M_\bullet-1}\sum_{\substack{\widehat\beta\in\widehat\HH_\bullet\wedge \widehat\beta|_{\partial\widehat\Omega}= 0\\ \supp (\widehat\beta)\subseteq\pi_\bullet^{\rm proj} (\widehat T)\atop\level(\widehat\beta)=k}}\dual{\widehat\beta^*}{\widehat V_\bullet} \,\Trunc_\bullet(\widehat\beta)|_{\widehat T}= (\widehat J_\bullet \widehat V_\bullet)|_{\widehat T}\stackrel{\eqref{eq:projection}}= \widehat V_\bullet|_{\widehat T}=\widehat v|_{\widehat T}.
\end{eqnarray*}

Next, we prove \eqref{S:app}. 
Let $\widehat T\in\widehat\TT_\bullet$, $\widehat v\in H_0^1(\widehat\Omega)$, and  $\widehat V_\bullet\in \widehat \XX_\bullet$.
By  \eqref{eq:projection}--\eqref{eq:L2stability}, it holds that 
\begin{align*}
\norm{(1-\widehat J_\bullet)\widehat v}{L^2(\widehat T)} \stackrel{\eqref{eq:projection}}{=}\norm{\widehat v-\widehat V_\bullet}{L^2(\widehat T)}+\norm{\widehat J_\bullet (\widehat v-\widehat V_\bullet)}{L^2(\widehat T)}\stackrel{\eqref{eq:L2stability}}{\lesssim} \norm{  \widehat v-\widehat V_\bullet}{L^2(\pi_\bullet^{2(p+1)}(\widehat T))}.
\end{align*}
To proceed, we distinguish between two cases, first, $\pi_\bullet^{4(p+1)}(\widehat T)\cap \partial\widehat\Omega=\emptyset$ and, second, $\pi_\bullet^{4(p+1)}(\widehat T)\cap \partial\widehat\Omega\neq\emptyset$, i.e., if $\widehat T$ is far away from the boundary or not.
Note that these cases are equivalent to $|\pi_\bullet^{4(p+1)}(\widehat T)\cap \partial\widehat\Omega|=0$ resp. $|\pi_\bullet^{4(p+1)}(\widehat T)\cap \partial\widehat\Omega|>0$, since the elements in the parameter domain are rectangular.

In the first case, we proceed as follows:
\eqref{eq:0 contained} especially proves $1\in\widehat\YY_\bullet$ with $1=\sum_{\widehat\beta\in\widehat\HH_\bullet}c_{\widehat\beta}\widehat\beta$ on $\widehat\Omega$.
With Remark \ref{rem:connected}, we see that $|\supp(\widehat\beta)\cap\pi_\bullet^{2(p+1)}(\widehat T)|>0$ implies $\supp(\widehat\beta)\subseteq\pi_\bullet^{4(p+1)}(\widehat T)$.
Therefore, the  restriction satisfies
\begin{align*}
1=\sum_{\widehat\beta\in\widehat\HH_\bullet}c_{\widehat\beta}\widehat\beta|_{\pi_\bullet^{2(p+1)}(\widehat T)}=\sum_{\widehat\beta\in\widehat\HH_\bullet\atop|\supp(\widehat\beta)\cap \pi_\bullet^{2(p+1)}(\widehat T)|>0}c_{\widehat\beta}\widehat\beta|_{\pi_\bullet^{2(p+1)}(\widehat T)}=\sum_{\widehat\beta\in\widehat\HH_\bullet\atop \supp(\widehat \beta)\subseteq \pi_\bullet^{4(p+1)}(\widehat T)}c_{\widehat\beta}\widehat\beta|_{\pi_\bullet^{2(p+1)}(\widehat T)}.
\end{align*}
We define 
\begin{align*}
\widehat V_\bullet:=\widehat v_{\pi_\bullet^{2(p+1)}(\widehat T)}\sum_{\substack{\widehat\beta\in\widehat\HH_\bullet\\\supp(\widehat\beta)\subseteq \pi_\bullet^{4(p+1)}(\widehat T)}}c_{\widehat\beta}\widehat\beta,\quad\text{where }\widehat v_{\pi_\bullet^{2(p+1)}(\widehat T)}:=| \pi_\bullet^{2(p+1)}(\widehat T)|^{-1}\int_{\pi_\bullet^{2(p+1)}(\widehat T)}\widehat v\,dt.
\end{align*}
In the second case, we set $\widehat V_\bullet:=0$.
For the first case, the Poincar\'e inequality combined with admissibility concludes the proof, whereas we use the Friedrichs inequality combined with admissibility in the second case. 
In either case, we obtain $\widehat V_\bullet\in\widehat\XX_\bullet$ and 
\begin{align}\label{eq:PF}
\norm{\widehat v-\widehat V_\coarse}{L^2(\pi_\coarse ^{2(p+1)}(\widehat T))}\lesssim\diam(\pi_\bullet^{4(p+1)}(\widehat T))\norm{\nabla \widehat v}{L^2(\pi_\bullet^{4(p+1)}(\widehat T))}\simeq |\widehat T|^{1/d} \norm{\nabla\widehat v}{L^2(\pi_\bullet^{4(p+1)}(\widehat T))}.
\end{align}

The hidden constants depend only on $\widehat\TT_0$, $(p_1,\dots,p_d)$, and  the shape of the patch $\pi_\bullet^{2(p+1)}(\widehat T)$ resp. the shape of $\pi_\bullet^{4(p+1)}(\widehat T)$ and of $\pi_\bullet^{2(p+1)}(\widehat T)\cap \partial\widehat \Omega$.
However, by admissibility, the number of different patch shapes is bounded itself by a constant which again depends only on $d$, $\widehat\TT_0$ and  $(p_1,\dots,p_d)$.

Finally, we prove \eqref{S:grad}.
Let again $\widehat T\in\widehat\TT_\bullet$, $\widehat v\in H_0^1(\widehat\Omega)$. 
For all $\widehat V_\bullet\in\widehat \XX_\bullet$ that are constant on $\widehat T$, the projection property \eqref{eq:projection} implies that 
\begin{align*}
\norm{\nabla \widehat J_\bullet\widehat v}{L^2(\widehat T)}\stackrel{\eqref{eq:projection}}{=}\norm{\nabla \widehat J_\bullet(\widehat v-\widehat V_\bullet)}{L^2(\widehat T)}&\stackrel{\eqref{eq:parameter invest}}{\lesssim} |\widehat T|^{-1/d}\norm{\widehat J_\coarse (v-\widehat V_\coarse)}{L^2(\widehat T)}\\
&\stackrel{\eqref{eq:L2stability}}{\lesssim}|\widehat T|^{-1/d}\norm{\widehat v-\widehat V_\bullet}{L^2(\pi_\bullet^{2(p+1)}(\widehat T))}.
\end{align*}
Arguing as before and using \eqref{eq:PF}, we conclude the proof.

\subsection{Verification of (\ref{O:inverse})}
The inverse estimate  \eqref{O:inverse} follows from a standard scaling argument together with the regularity \eqref{eq:Cgamma} of $\gamma$ and the Poincar\'e-Friedrichs inequality on $\gamma^{-1}(T)$.
The constant $\C{inv}'$ depends only on $d$, $\C{\gamma}$, $\widehat \TT_0$, and $( q_1,\dots, q_d)$.

\subsection{Verification of (\ref{O:dual})--(\ref{O:grad})}
This section adapts \cite[Section~3.4]{nv}.
Let $W\in\mathcal{P}(\Omega)$, $\TT_\coarse\in\T$,  and $T, T'\in\TT_\coarse$ with $(d-1)$-dimensional intersection $E:=T\cap T'$.
We set $\widehat W:= W\circ \gamma$, $\widehat T:=\gamma^{-1}(T)$, $\widehat T':=\gamma^{-1}(T')$, and $\widehat E:=\gamma^{-1}(E)$.
Let $\gamma_{\widehat T}:\R^{d}\to\R^{d}$ be the affine transformation with $\widetilde T:=\gamma^{-1}_{\widehat T}(\widehat T)=[0,1]^{d}$.
Due to admissibility of $\widehat \TT_\coarse$, the number of  different configurations for the set $\widetilde T':=\gamma_{\widehat T}^{-1}(\widehat T')$ is uniformly bounded by a constant that depends only on $d$ and $\widehat \TT_0$.
We fix  some cut-off function $\widetilde\varphi\in H_0^1(\widetilde T\cup\widetilde T')$ with $\widetilde\varphi>0$ almost everywhere on $\widetilde E:=\gamma^{-1}_{\widehat T}(\widehat E)$.
We define $\varphi:=\widetilde\varphi\circ\gamma_{\widehat T'}^{-1}\circ\gamma^{-1}$, and $\widetilde W:= W\circ\gamma \circ \gamma_{\widehat T}$.
Due to the finite dimension of the polynomial space $\widetilde{\mathcal{P}}(\widetilde T\cup \widetilde T'):=\set{W'\circ\gamma \circ \gamma_{\widehat T}}{W'\in\mathcal{P}(\Omega)}$, there exists $\widetilde W'\in \widetilde{\mathcal{P}}(\widetilde T\cup \widetilde T')$ with $ \widetilde W|_{\widetilde E}= \widetilde W'|_{\widetilde E}$ and $\norm{ \widetilde W'\widetilde \varphi}{L^2(\widetilde T\cup\widetilde T')}\lesssim \norm{ \widetilde W}{L^2(\widetilde E)}$.
Finally, we set $W':= \widetilde W'\circ\gamma_{\widehat T'}^{-1}\circ\gamma^{-1}$, and $J_{\coarse,E}:=W'\varphi$.
Standard scaling arguments prove that \eqref{O:dual}--\eqref{O:grad} are satisfied, where the constants depend only on $d$, $\C{\gamma}$, $\widehat \TT_0$, and $( q_1,\dots, q_d)$.





\section{Numerical examples}\label{sec:numerics}

\noindent In this section, we apply our adaptive algorithm to the 2D Poisson model problem
\begin{align}\label{eq:Poisson}
\begin{split}
-\Delta u&=f\quad \text{in }\Omega,\\
u&=0\quad\text{on }\partial\Omega.
\end{split}
\end{align}
on different NURBS surfaces $\Omega\subset\R^2$.
 On the parameter domain $\widehat{\Omega}=(0,1)^d$, let $p_1^\gamma,p_2^\gamma\ge 1$ be  fixed polynomial degrees and $\widehat\KK^\gamma$ be  an arbitrary fixed $2$-dimensional vector of $p_i^\gamma$-open knot vectors with multiplicity smaller or equal to $p_i^\gamma$ for the interior knots, i.e.,
\begin{align}
\widehat\KK^\gamma=(\widehat\KK^\gamma_1,\widehat\KK^\gamma_2);
\end{align}
see also Section~\ref{subsec:splines}. 
Let $\widehat\BB^\gamma$ be the corresponding tensor-product B-spline basis,
i.e., 
\begin{align}
\widehat\BB^\gamma=\set{\widehat{B}^\gamma_{j_1,j_2}}{j_i\in\{0,\dots,N_i^\gamma-1\}}, 
\end{align}
where $\widehat{B}^\gamma_{j_1,j_2}$ is defined as in \eqref{eq:B-spline}.
For given points in the plane $c^\gamma_{j_1,j_2}\in\R^2$ and positive weights $w^\gamma_{j_1,j_2}>0$ for $j_i\in\{0,\dots,N_i^\gamma-1\}$, a NURBS-surface $\Omega$ is  defined via a parametrization $\gamma:\widehat\Omega\to\Omega$ of the form
\begin{align}
\gamma(s_1,s_2) = \frac{\sum\limits_{\substack{i\in\{1,2\}\\j_i\in\{0,\dots,N_i^\gamma-1\}}}w^\gamma_{j_1,j_2}
c^\gamma_{j_1,j_2}\widehat{B}^\gamma_{j_1,j_2}(s_1,s_2)}{\sum\limits_{\substack{i\in\{1,2\}\\j_i\in\{0,\dots,N_i^\gamma-1\}}}w^\gamma_{j_1,j_2} \widehat{B}^\gamma_{j_1,j_2}(s_1,s_2)}.
\end{align}

\subsection{Square} 
\label{subsec:square}
In the first experiment, we consider the unit square 
$\Omega = (0,1)^2$, where  we choose $p^\gamma_1,p^\gamma_2=1$ and $\widehat\KK^\gamma_1=\widehat\KK^\gamma_2 = (0,0,1,1)$. We set the control points
\begin{align*}
c^\gamma_{1,1} = (0,0),~ c^\gamma_{2,1} = (1,0),~ c^\gamma_{1,2} = (0,1),~ c^\gamma_{2,2} = (1,1)
\end{align*} and all weights equal to $1$. We choose $f$ as in \cite[Example 7.4]{garau16} such  that the exact solution of \eqref{eq:Poisson} is given by $u(x_1,x_2)=x_1^{2.3}(1-x_1)x_2^{2.9}(1-x_2)$, which implies $u\in H^2(\Omega)\setminus H^3(\Omega).$ To create the initial ansatz space with spline degrees $p_1=p_2 \in \{2,3,4\}$, we choose the initial knot vectors $\widehat\KK^0_1=\widehat\KK^0_2=(0,\dots 0,1,\dots,1)$, where the multiplicity of $0$ and $1$ is $p_1+1=p_2+1$.
We apply Algorithm~\ref{the algorithm} for the spline degrees $2,3$, and $4$ with uniform ($\theta=1$) and  adaptive ($\theta=0.5$) mesh-refinement.
Some adaptively generated hierarchical meshes are shown in Figure~\ref{fig:square_mesh}.
In Figure~\ref{fig:square}, we plot  the energy error $\norm{\nabla u-\nabla U_\ell}{L^2(\Omega)}$ as well as the error estimator $\eta_\ell$ against the number of elements $\#\TT_\ell$.
Due to the regularity  of $u$, we obtain a suboptimal convergence rate  for $p_1=p_2>2$. 
The adaptive strategy on the other hand recovers the optimal convergence of the
energy error and the error estimator, i.e.,  $\mathcal{O}((\# \TT_\ell)^{-p/2})$.

\subsection{L-shape} 
\label{subsec:L-shape}
To obtain the L-shaped domain 
$\Omega = (0,1)^2 \setminus ([0,0.5] \times [0,0.5])$, we choose 
 $p^\gamma_1,p^\gamma_2=1$ and $\widehat\KK^\gamma_1 = (0,0,0.5,1,1),~\widehat\KK^\gamma_2 = (0,0,1,1)$. 
 Moreover, we choose the control points 
 \begin{align*}
c^\gamma_{1,1} &= (0,0.5),~ c^\gamma_{2,1} = (0.5,0.5),~
c^\gamma_{3,1} = (0.5,0),\\ c^\gamma_{1,2} &= (0,1),~ c^\gamma_{2,2} = (1,1),~
c^\gamma_{3,2} = (1,0),
\end{align*}
and set all weights equal to $1$. We consider the Poisson problem \eqref{eq:Poisson} with $f=1$. For the initial ansatz space with spline degrees $p_1=p_2 \in \{2,3,4\}$, we choose the initial knot vectors $\widehat\KK^0_1=(0,\dots,0,0.5,\dots,0.5,1,\dots,1)$ and $\widehat\KK^0_2=(0,\dots,0,1\dots,1)$, where  the multiplicity of $0$ and $1$ is $p_1+1=p_2+1$, whereas the multiplicity of $0.5$ is $p_1$.
As a consequence, the ansatz functions are only continuous at $\{0.5\}\times[0,1]$,  but not continuously differentiable.
We compare uniform  ($\theta=1$) and adaptive ($\theta=0.4$) mesh-refinement.
 In Figure~\ref{fig:Lshape_mesh}, one can see some adaptively generated hierarchical meshes.
 In Figure~\ref{fig:Lshape}
we plot again   the energy error $\norm{\nabla u-\nabla U_\ell}{L^2(\Omega)}$ and the error estimator $\eta_\ell$ against the  number of elements $\#\TT_\ell$.  
The uniform approach leads to a suboptimal convergence rate, since the reentrant corner at  $(0.5,0.5)$ causes a generic singularity of the solution $u$.
 However, the adaptive strategy recovers the optimal convergence rate $\mathcal{O}((\# \TT_\ell)^{-p/2})$.

\subsection{Quarter ring}\label{subsec:ring}
We construct the NURBS-surface  given in polar coordinates
$\Omega = \{(r,\varphi)|~0.5 < r < 1~\wedge~0<\varphi<\pi/2\}$ by choosing
$p^\gamma_1=2,p^\gamma_2=1$ and $\widehat\KK^\gamma_1 = (0,0,0,1,1,1),~\widehat\KK^\gamma_2 = (0,0,1,1)$. 
Moreover, we choose the control points \begin{align*}
c^\gamma_{1,1} &= (0,0.5),~ c^\gamma_{2,1} = (0.5,0.5),~
c^\gamma_{3,1} = (0.5,0),\\ 
c^\gamma_{1,2} &= (0,1),~ c^\gamma_{2,2} = (1,1),~
c^\gamma_{3,2} = (1,0),
\end{align*}
and set all weights,  except $w^\gamma_{2,1},w^\gamma_{2,2} := 1/\sqrt{2}$, equal to $1$. 
As right-hand side in \eqref{eq:Poisson},  we choose the indicator function $f=\chi_S$, where $S=\{(r,\varphi)|~0.5 < r < 0.75~\wedge~0<\varphi<\pi/4\} = \gamma([0.5,1]\times[0,0.5])$. 
For the initial ansatz space with spline degrees $p_1=p_2 \in \{2,3,4\}$, we choose  the initial knot vectors $\widehat\KK^0_1=\widehat\KK^0_2=(0,\dots,0,0.5,\dots,0.5,1\dots,1)$,  where  the multiplicity of $0$ and $1$ is $p_1+1=p_2+1$, whereas the multiplicity of $0.5$ is $p_1=p_2$.
As a consequence the ansatz functions are only continuous at $\{0.5\}\times[0,1]$ and $[0,1]\times\{0.5\}$. 
We  compare uniform ($\theta=1$)   and adaptive ($\theta=0.8$) mesh-refinement.
Some adaptively generated hierarchical meshes are shown  in Figure~\ref{fig:bendshape_mesh}.
In Figure~\ref{fig:bendshape}, we plot  the energy error $\norm{\nabla u-\nabla U_\ell}{L^2(\Omega)}$ and the error estimator $\eta_\ell$ against the  number of elements $\#\TT_\ell$.
For $p_1=p_2>2$, uniform refinement leads to suboptimal convergence rate.
However, the adaptive approach regains the  optimal rate $\mathcal{O}((\#\TT_\ell)^{-p/2})$.

\begin{figure}
\centering

\includegraphics[width=0.32\textwidth]{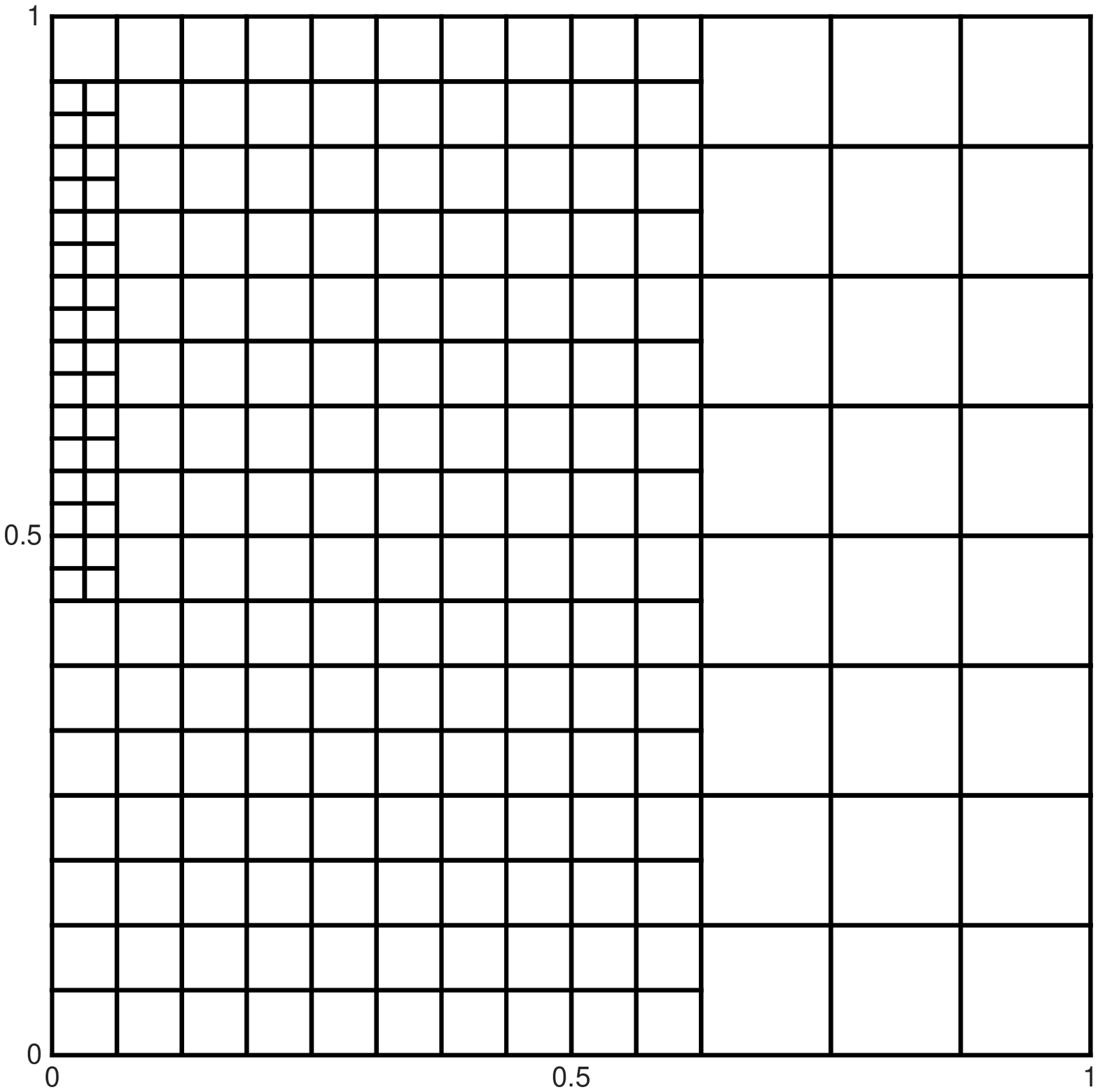}
\includegraphics[width=0.32\textwidth]{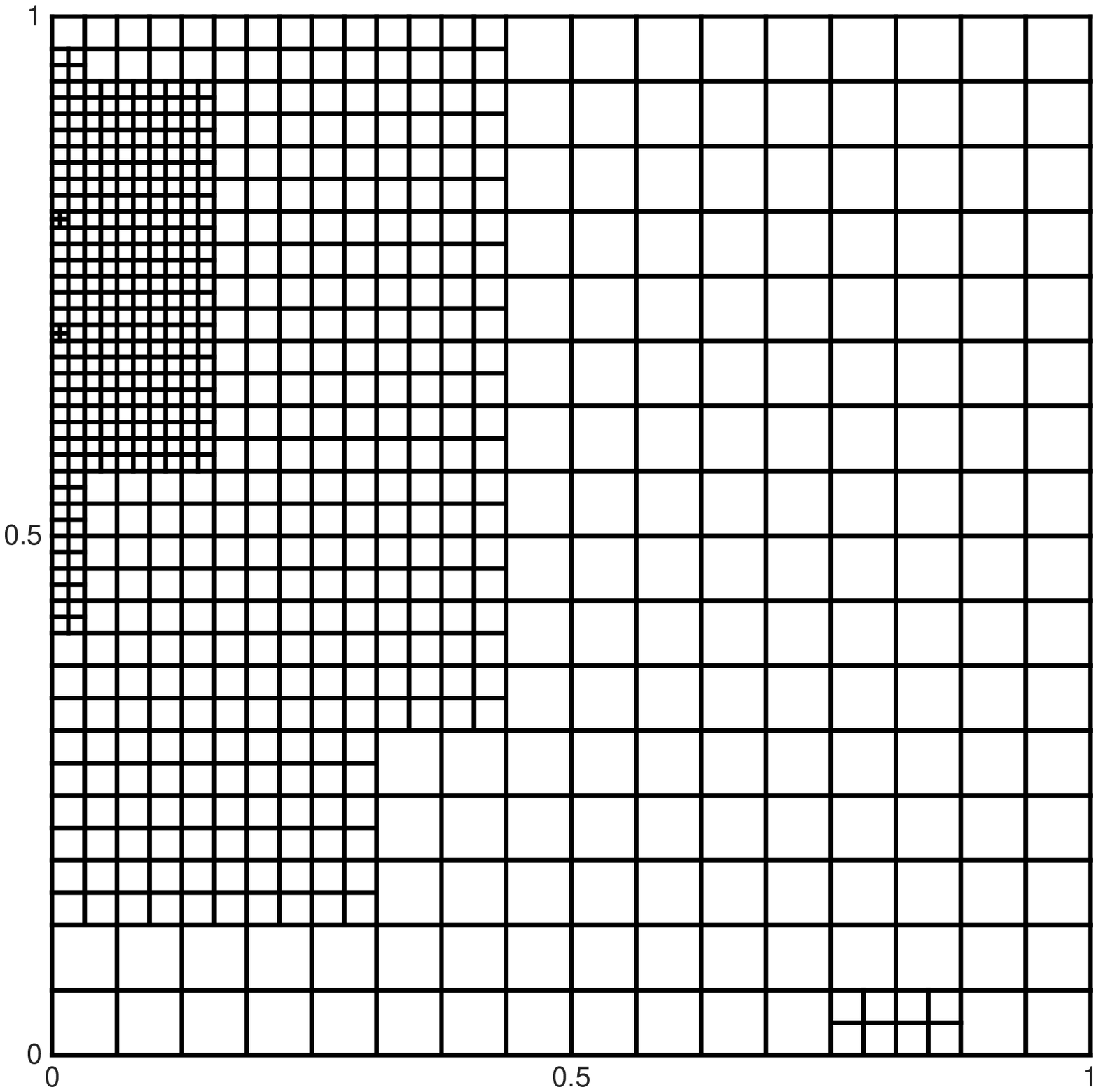}
\includegraphics[width=0.32\textwidth]{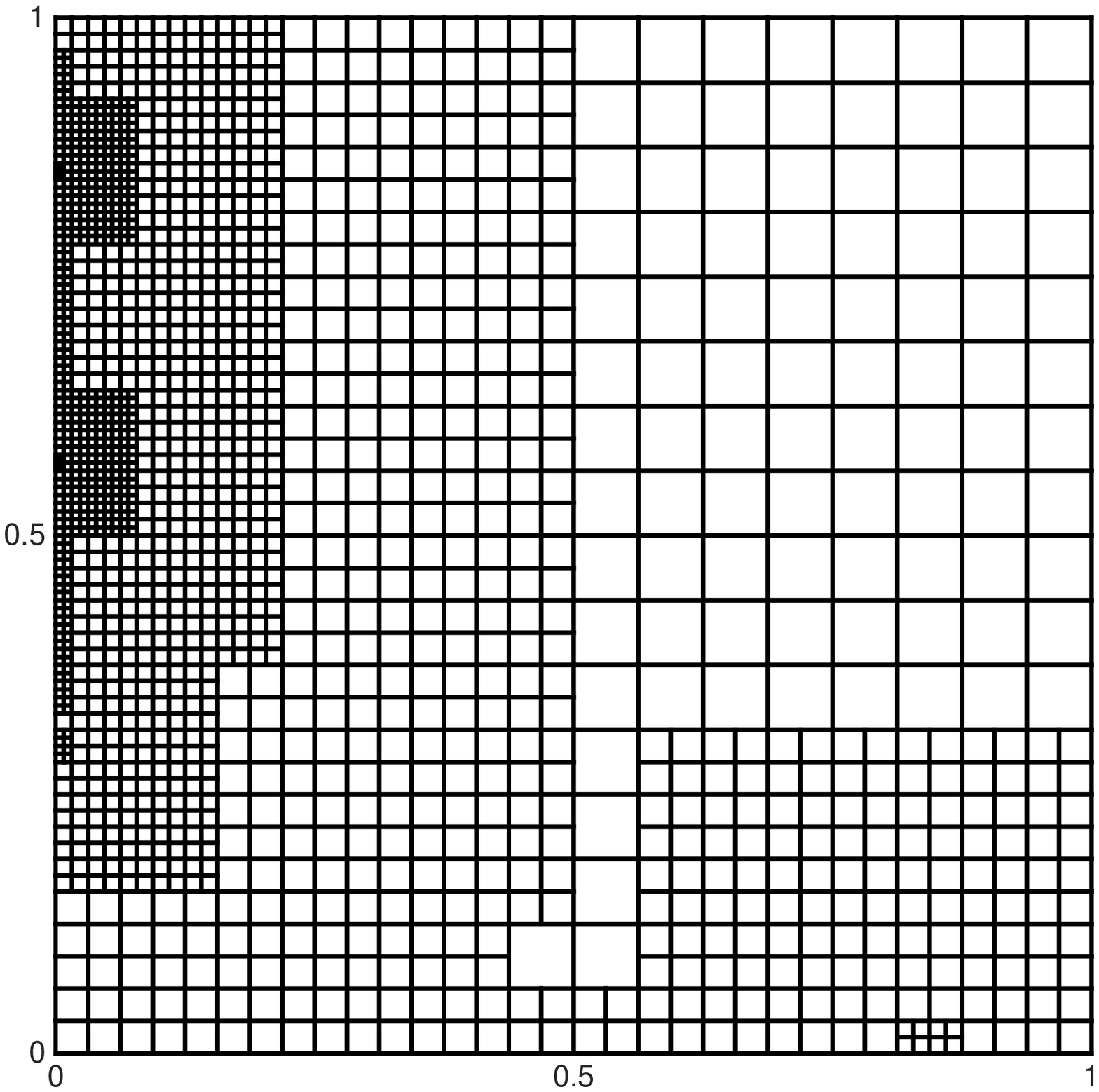}
\begin{minipage}{0.32\textwidth}
\begin{center}$\#\TT_6=208$\end{center}
\end{minipage}
\begin{minipage}{0.32\textwidth}
\begin{center}$\#\TT_8=742$\end{center}
\end{minipage}
\begin{minipage}{0.32\textwidth}
\begin{center}$\#\TT_{10}=1633$\end{center}
\end{minipage}

\caption{Hierarchical meshes generated by  Algorithm~\ref{the algorithm} with $\theta=0.5$ for the problem of Section~\ref{subsec:square}   for $p_1=p_2=4$. 
}
\label{fig:square_mesh}
\end{figure}

\begin{figure}
\psfrag{ON-2}[r][r]{\hspace{-10pt}\scalebox{0.5}{$\tiny \mathcal{O}(\#\mathcal{T_\ell}^{-2/2})$}}
\psfrag{ON-3}[r][r]{\hspace{-10pt}\scalebox{0.5}{$\tiny \mathcal{O}(\#\mathcal{T_\ell}^{-3/2})$}}
\psfrag{ON-4}[r][r]{\hspace{-10pt}\scalebox{0.5}{$\tiny \mathcal{O}(\#\mathcal{T_\ell}^{-4/2})$}}

\psfrag{Square, degree=2}{}
\psfrag{Square, degree=3}{}
\psfrag{Square, degree=4}{}

\psfrag{Number of elements}{\hspace{-24pt}\fontsize{5pt}{6pt}\selectfont number of elements $\#\TT_\ell$}
\psfrag{Energy error/Error estimator}{\hspace{-20pt}\fontsize{5pt}{6pt}\selectfont energy error/error estimator}

\psfrag{Error uniform}{\fontsize{5pt}{6pt}\selectfont err. unif.}
\psfrag{Error adaptive}{\fontsize{5pt}{6pt}\selectfont err. adap.}
\psfrag{Estimator uniform}{\fontsize{5pt}{6pt}\selectfont est. unif.}
\psfrag{Estimator adaptive}{\fontsize{5pt}{6pt}\selectfont est. adap.}

\centering
 
\includegraphics[width=0.32\textwidth]{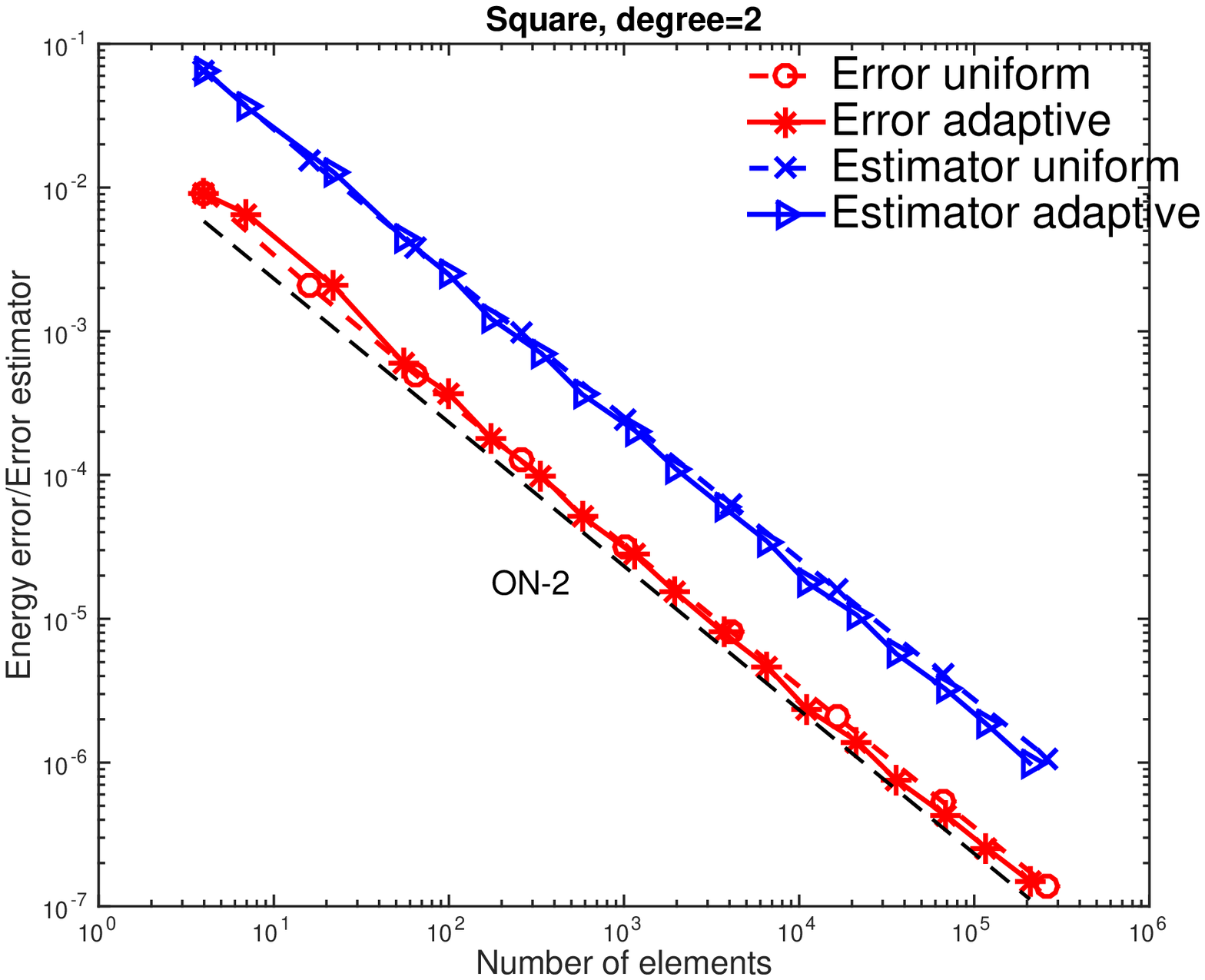}
\includegraphics[width=0.32\textwidth]{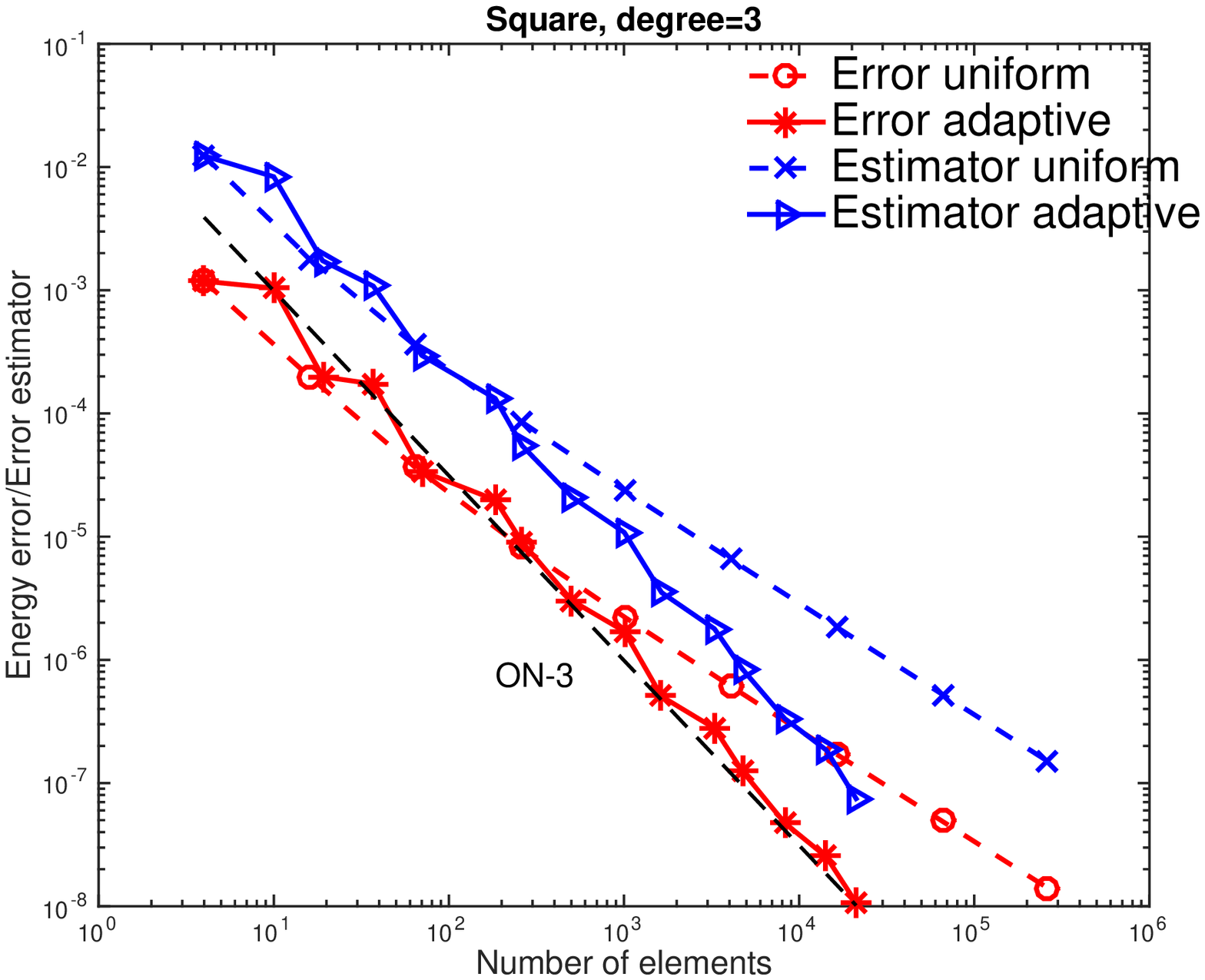}
\includegraphics[width=0.32\textwidth]{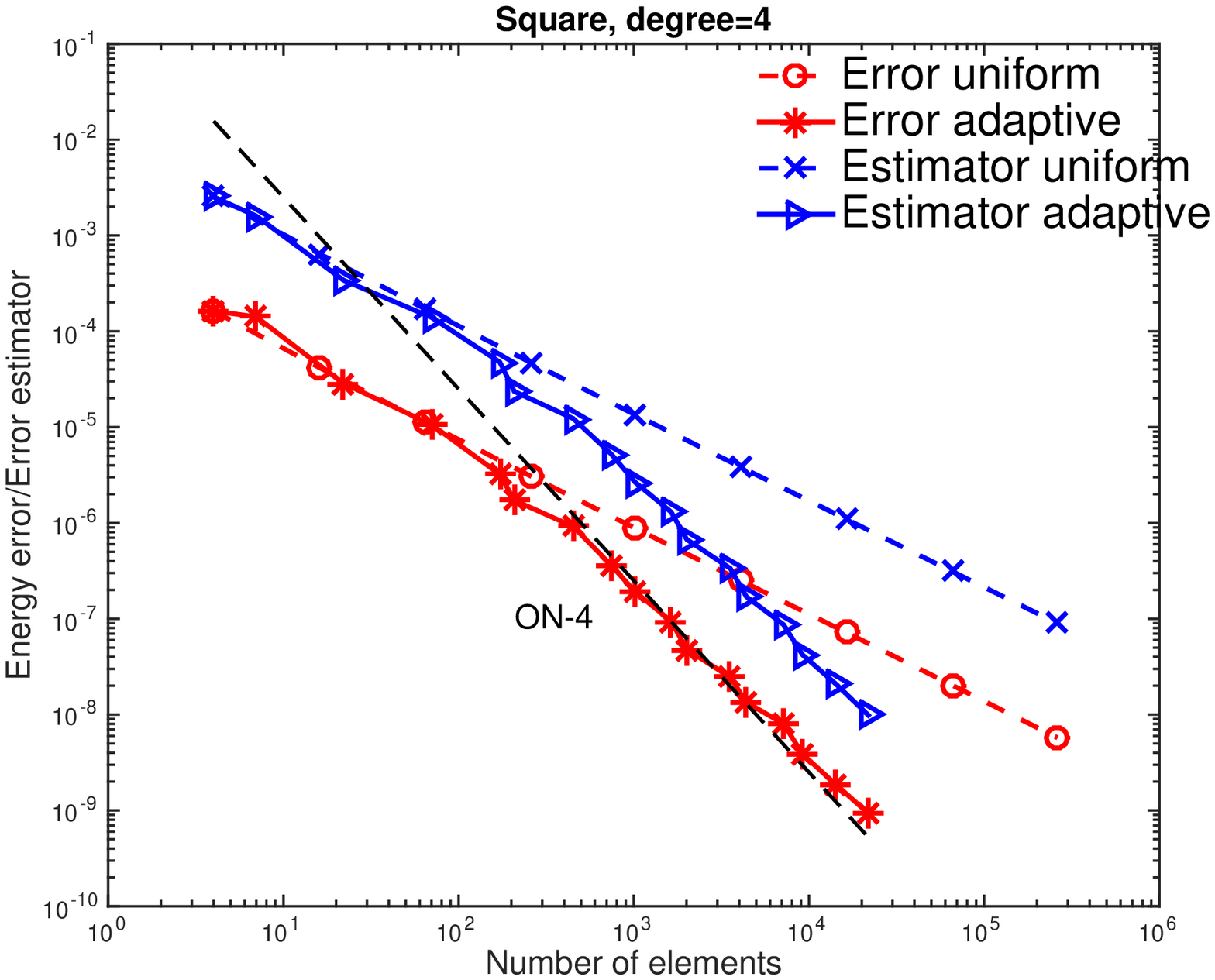}
\begin{minipage}{0.32\textwidth}
\begin{center}$p=2$\end{center}
\end{minipage}
\begin{minipage}{0.32\textwidth}
\begin{center}$p=3$\end{center}
\end{minipage}
\begin{minipage}{0.32\textwidth}
\begin{center}$p=4$\end{center}
\end{minipage}

\caption{
Energy error $\norm{\nabla u-\nabla U_\ell}{L^2(\Omega)}$ and  error estimator $\eta_\ell$ for the problem of Section~\ref{subsec:square} on the unit square for uniform ($\theta=1$) and adaptive ($\theta=0.5$) mesh-refinement and different spline orders $p_1=p_2=p\in\{2,3,4\}$, where adaptivity always regains the respective optimal convergence rate.
}
\label{fig:square} 
\end{figure}

\begin{figure}
\centering

\includegraphics[width=0.32\textwidth]{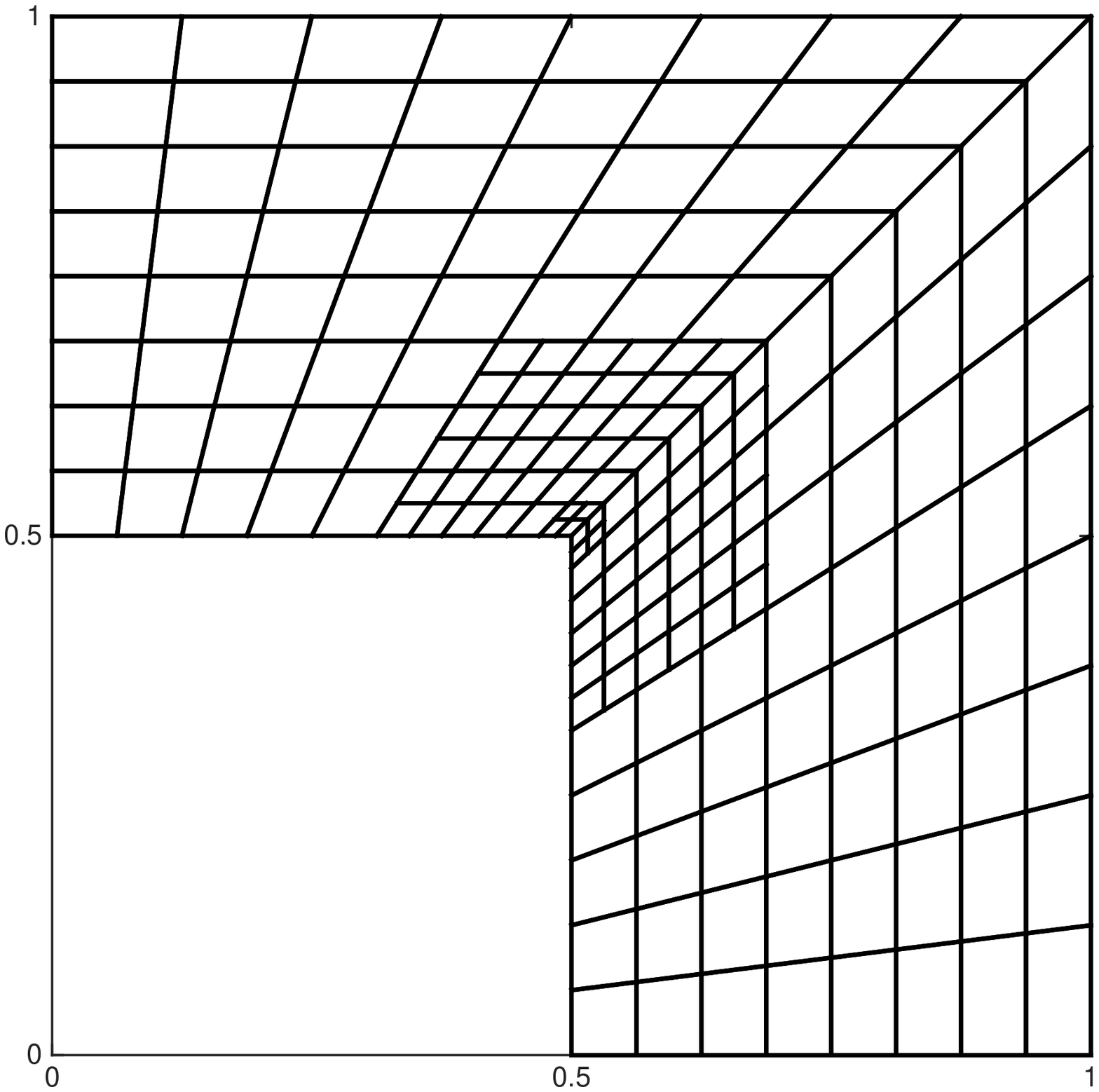}
\includegraphics[width=0.32\textwidth]{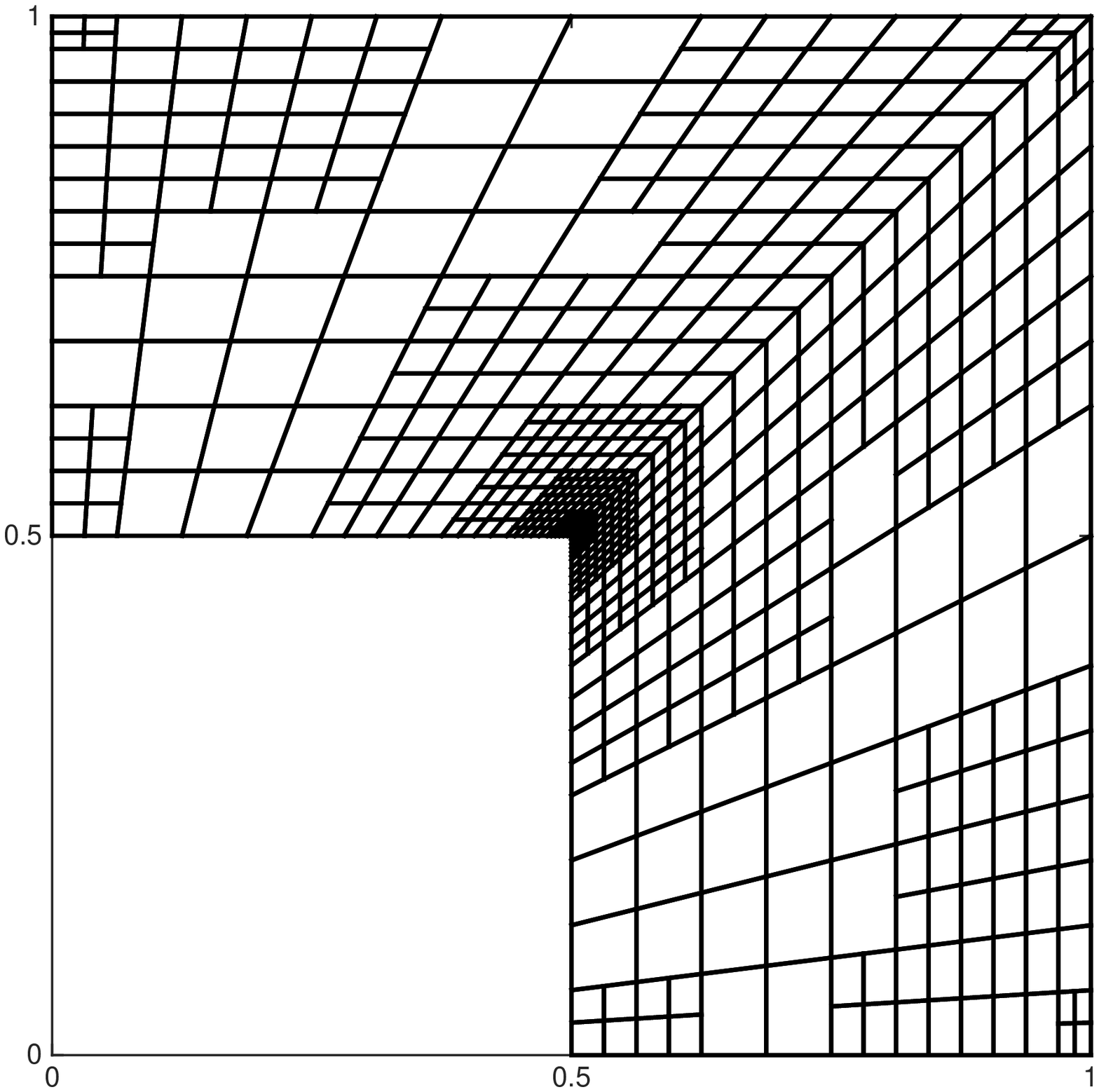}
\includegraphics[width=0.32\textwidth]{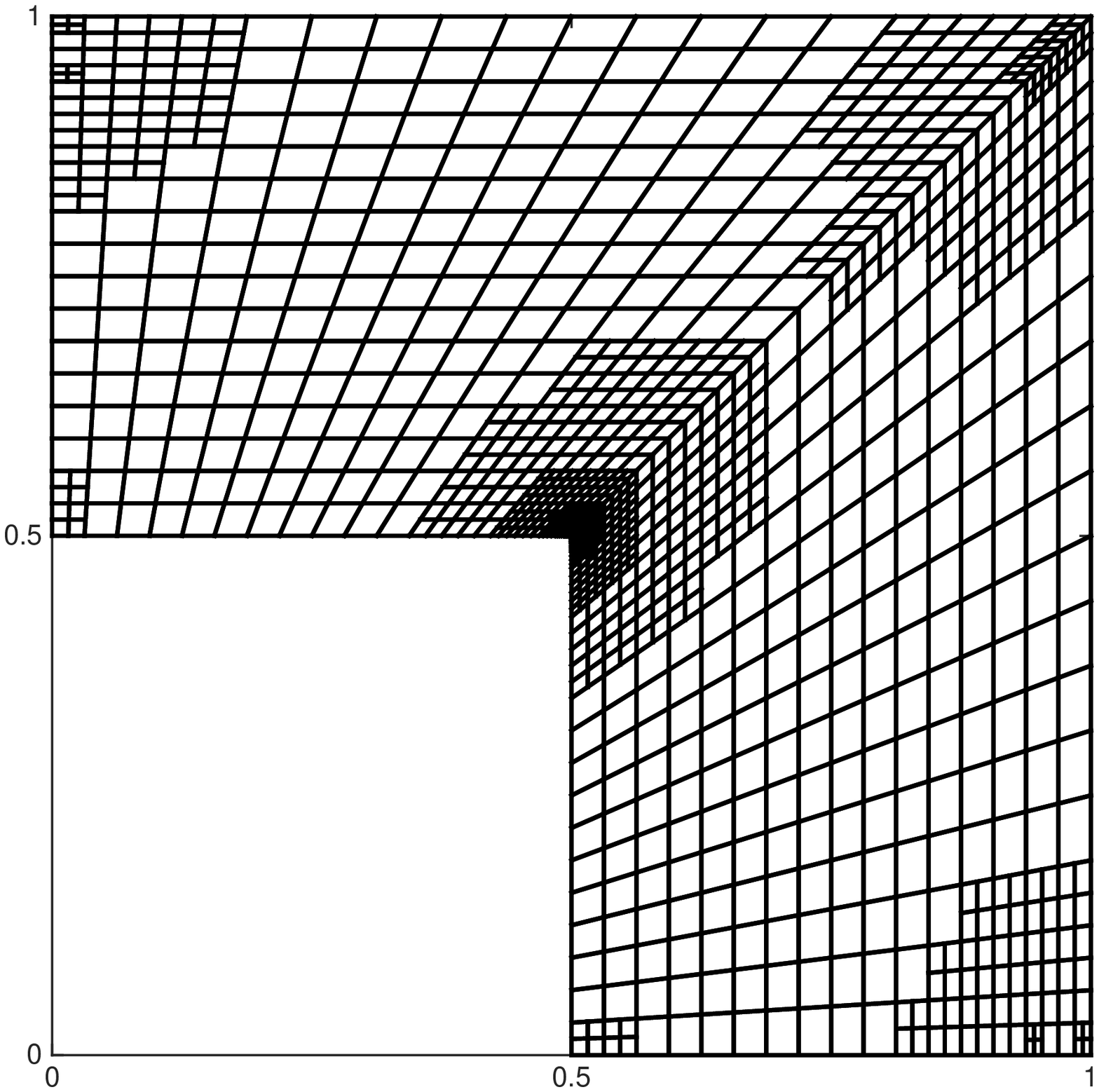}
\begin{minipage}{0.32\textwidth}
\begin{center}$\#\TT_6=188$\end{center}
\end{minipage}
\begin{minipage}{0.32\textwidth}
\begin{center}$\#\TT_9=632$\end{center}
\end{minipage}
\begin{minipage}{0.32\textwidth}
\begin{center}$\#\TT_{11}=1364$\end{center}
\end{minipage}

\caption{Hierarchical meshes generated by  Algorithm~\ref{the algorithm} with $\theta=0.4$ for the problem of Section~\ref{subsec:L-shape}   for $p_1=p_2=2$. 
}

\label{fig:Lshape_mesh}
\end{figure}

\begin{figure}
\psfrag{ON-2}[r][r]{\hspace{-10pt}\scalebox{0.5}{$\tiny \mathcal{O}(\#\mathcal{T_\ell}^{-2/2})$}}
\psfrag{ON-3}[r][r]{\hspace{-10pt}\scalebox{0.5}{$\tiny \mathcal{O}(\#\mathcal{T_\ell}^{-3/2})$}}
\psfrag{ON-4}[r][r]{\hspace{-10pt}\scalebox{0.5}{$\tiny \mathcal{O}(\#\mathcal{T_\ell}^{-4/2})$}}

\psfrag{L-shape, degree=2}{}
\psfrag{L-shape, degree=3}{}
\psfrag{L-shape, degree=4}{}

\psfrag{Number of elements}{\hspace{-24pt}\fontsize{5pt}{6pt}\selectfont number of elements $\#\TT_\ell$}
\psfrag{Energy error/Error estimator}{\hspace{-20pt}\fontsize{5pt}{6pt}\selectfont energy error/error estimator}

\psfrag{Error uniform}{\fontsize{5pt}{6pt}\selectfont err. unif.}
\psfrag{Error adaptive}{\fontsize{5pt}{6pt}\selectfont err. adap.}
\psfrag{Estimator uniform}{\fontsize{5pt}{6pt}\selectfont est. unif.}
\psfrag{Estimator adaptive}{\fontsize{5pt}{6pt}\selectfont est. adap.}

\centering

\includegraphics[width=0.32\textwidth]{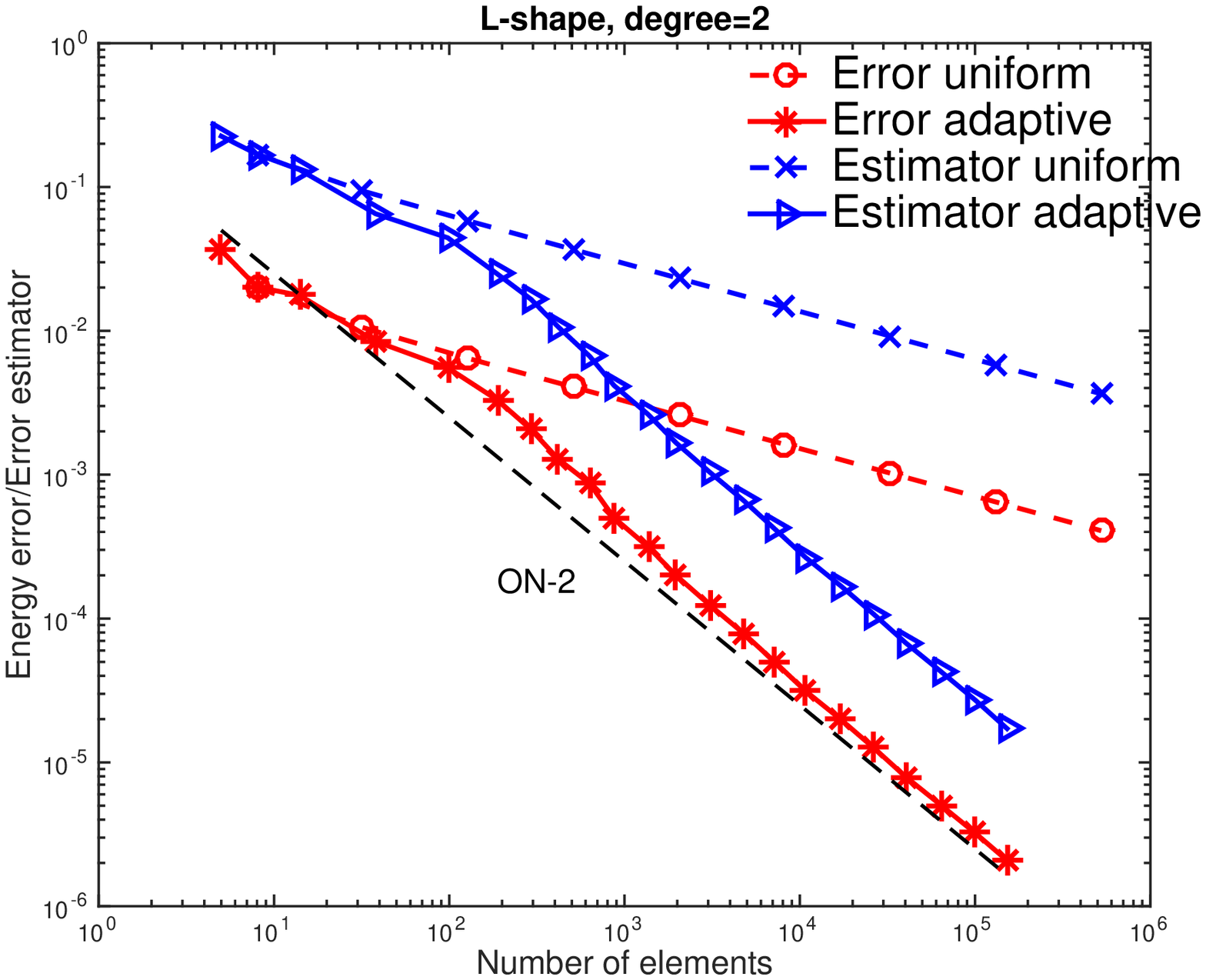}
\includegraphics[width=0.32\textwidth]{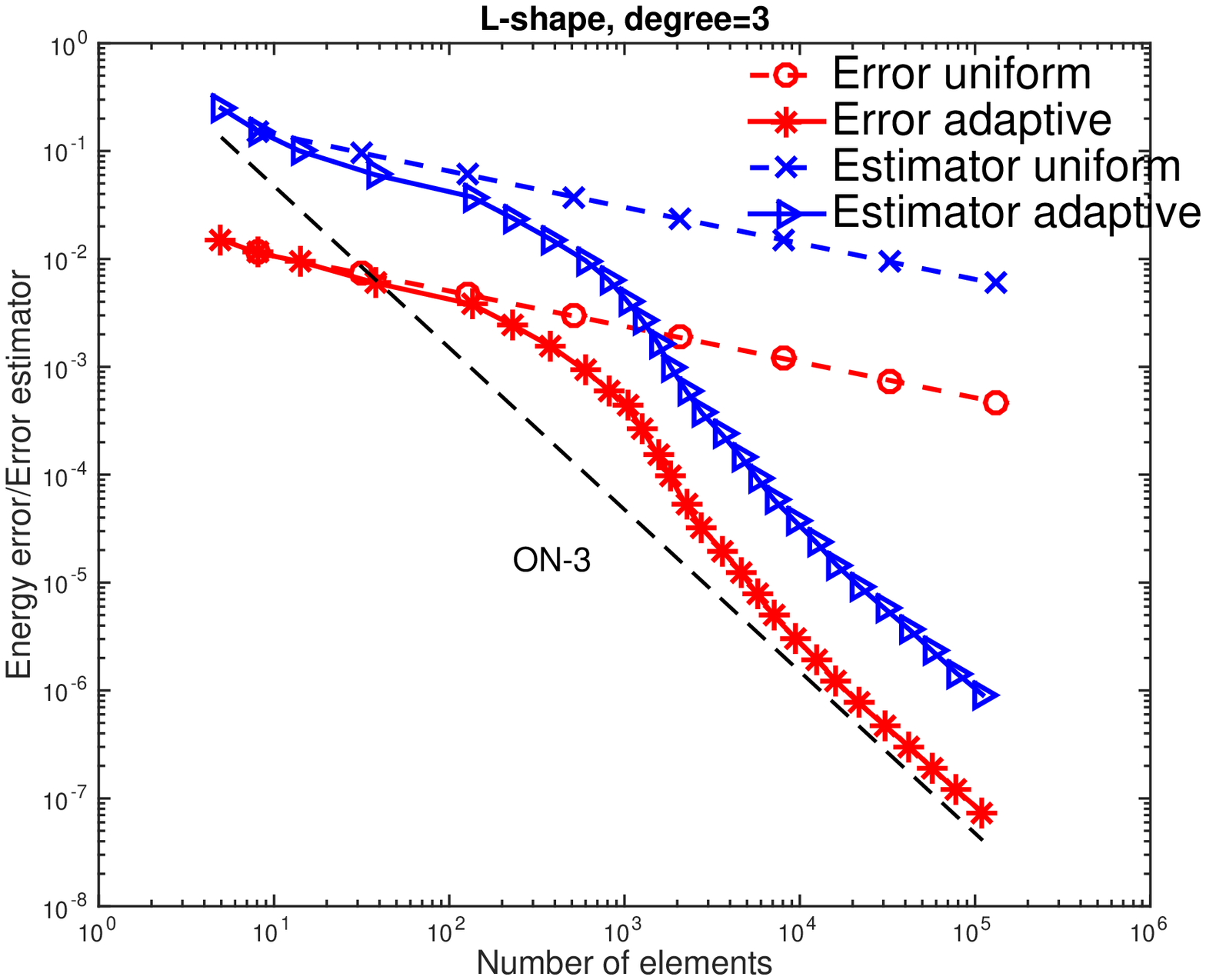}
\includegraphics[width=0.32\textwidth]{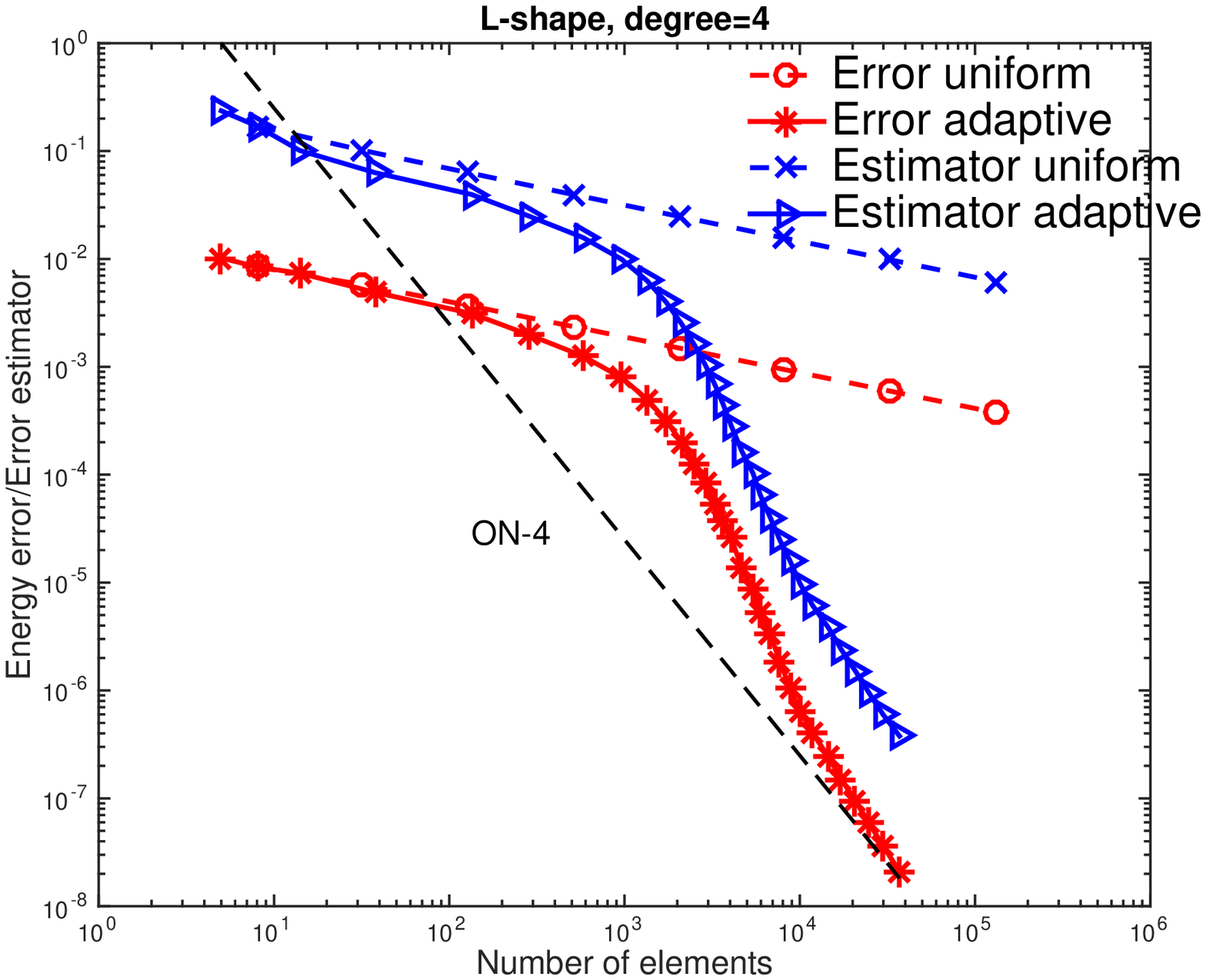}
\begin{minipage}{0.32\textwidth}
\begin{center}$p=2$\end{center}
\end{minipage}
\begin{minipage}{0.32\textwidth}
\begin{center}$p=3$\end{center}
\end{minipage}
\begin{minipage}{0.32\textwidth}
\begin{center}$p=4$\end{center}
\end{minipage}

\caption{
Energy error $\norm{\nabla u-\nabla U_\ell}{L^2(\Omega)}$ and  error estimator $\eta_\ell$ for the problem of Section~\ref{subsec:ring} on the L-shape   for uniform ($\theta=1$) and adaptive ($\theta=0.5$) mesh-refinement and different spline orders $p_1=p_2=p\in\{2,3,4\}$, where adaptivity always regains the respective optimal convergence rate.
}
\label{fig:Lshape} 
\end{figure}

\begin{figure}
\centering

\includegraphics[width=0.32\textwidth]{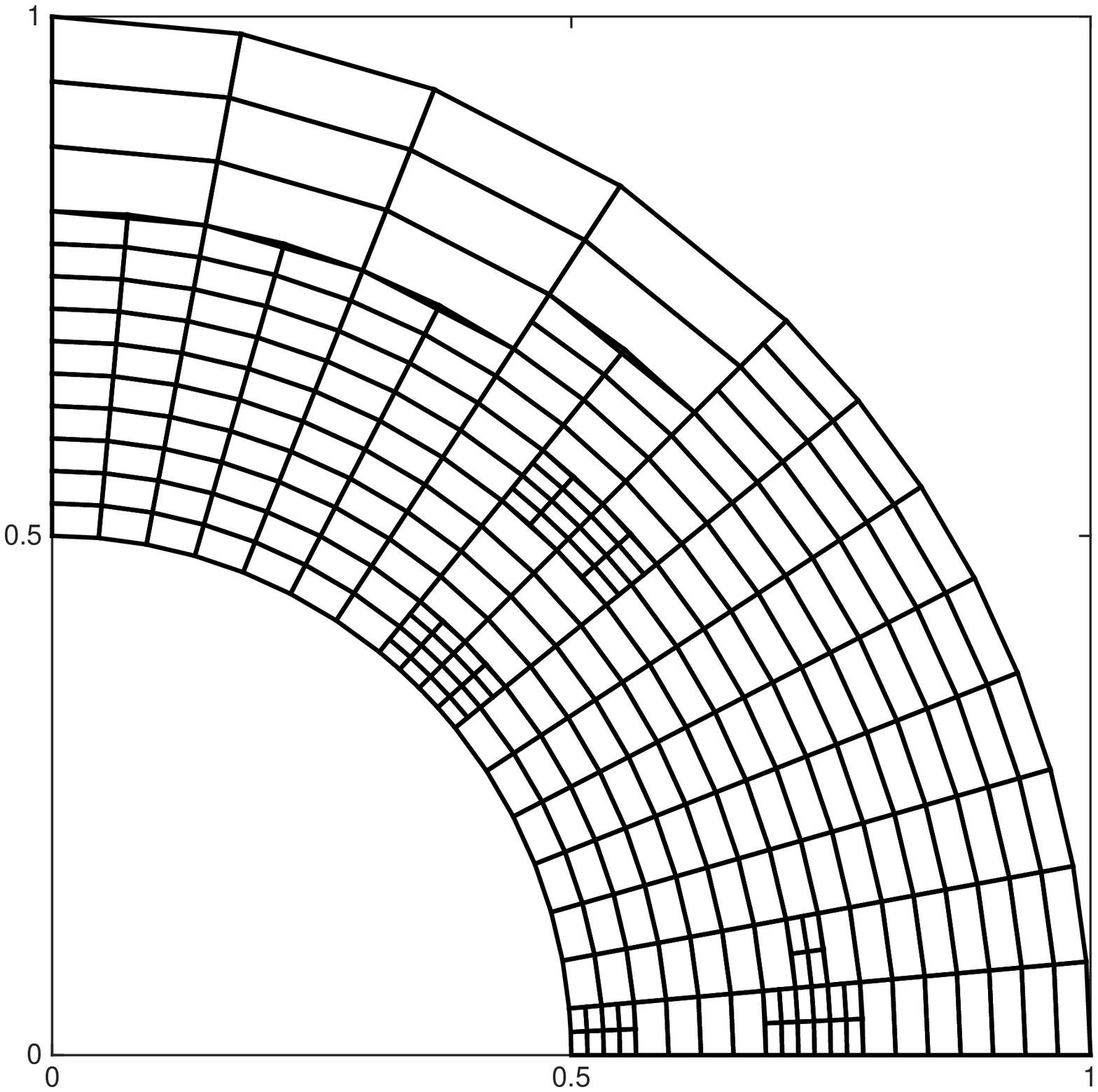}
\includegraphics[width=0.32\textwidth]{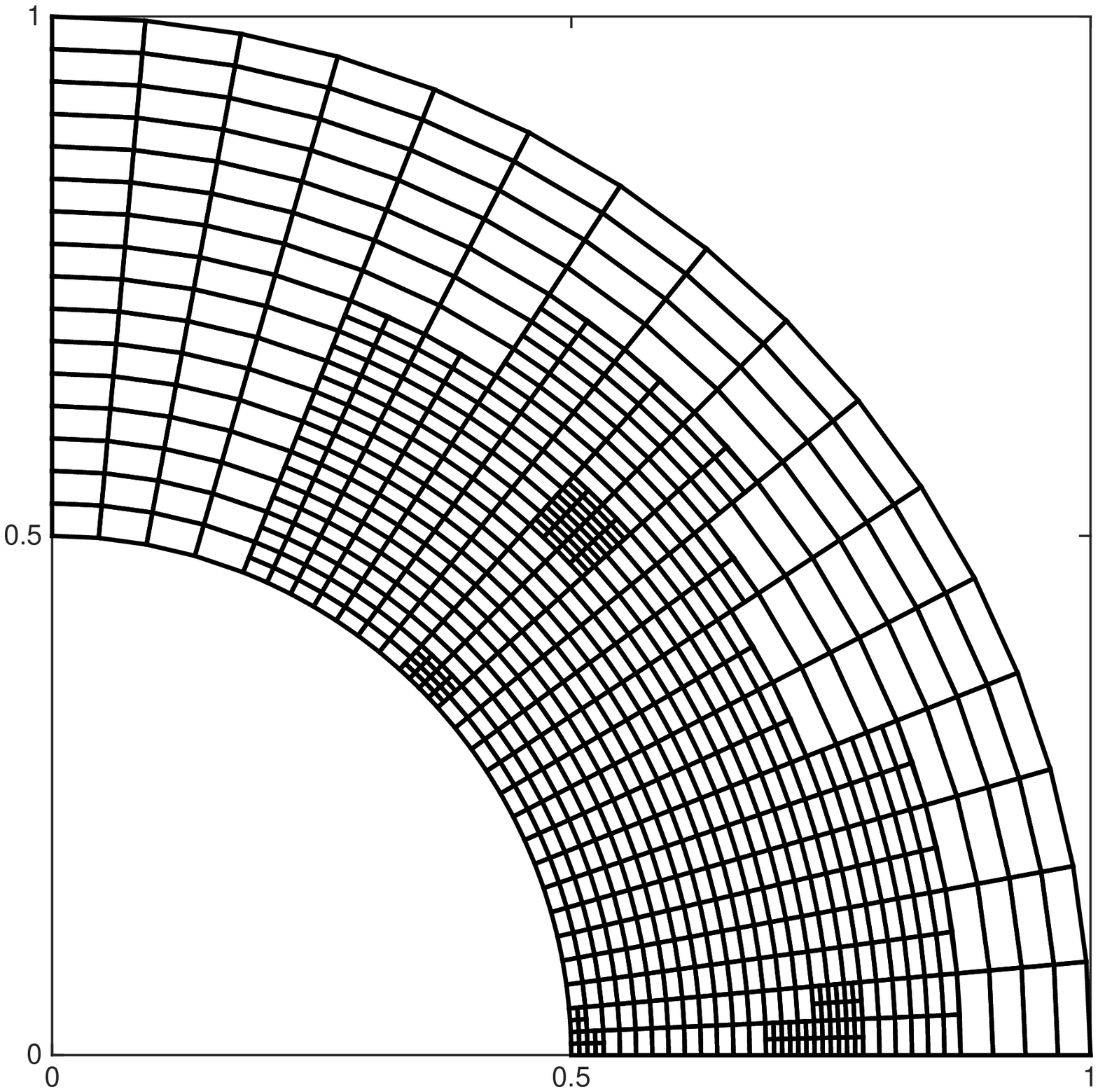}
\includegraphics[width=0.32\textwidth]{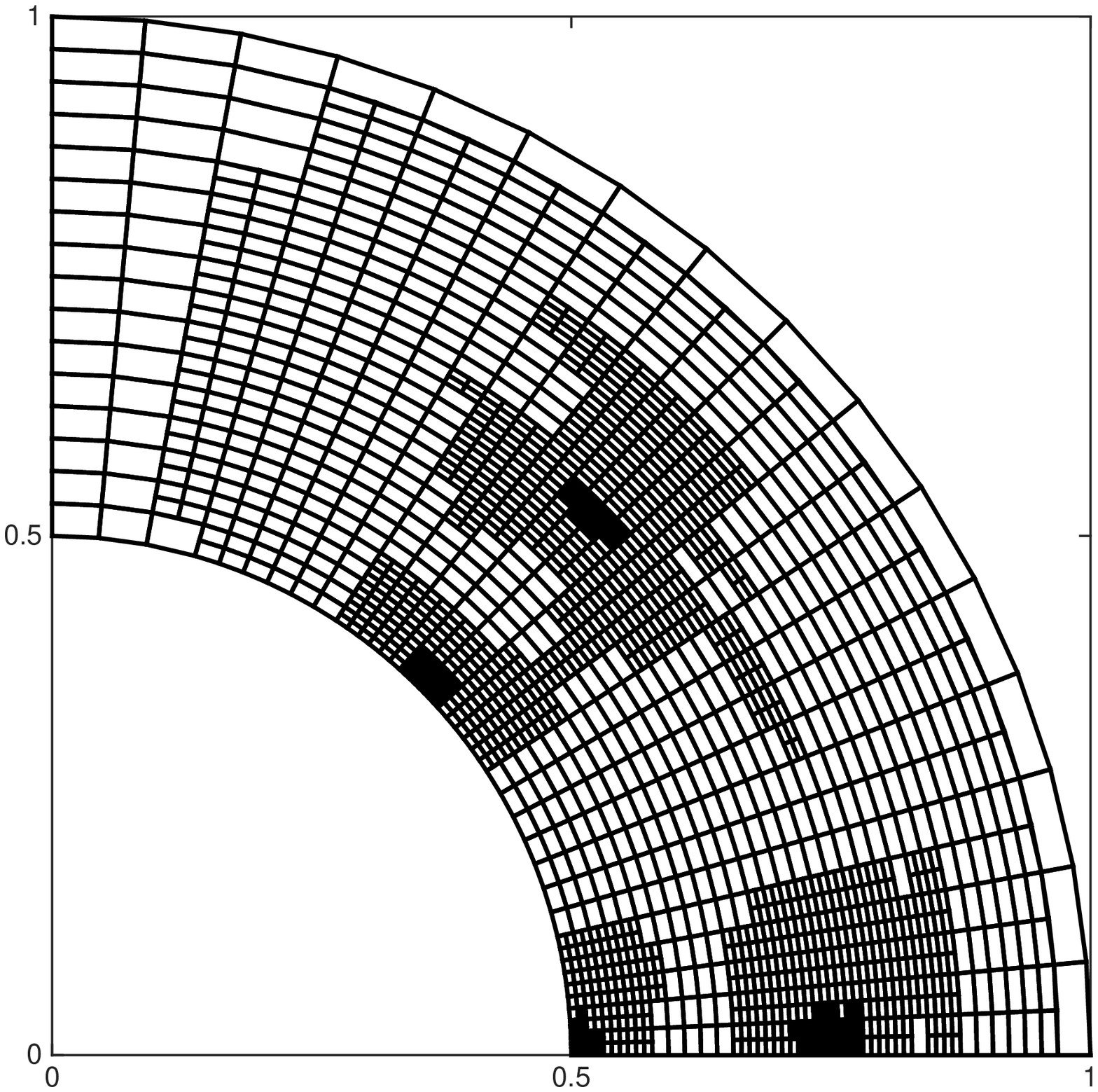}
\begin{minipage}{0.32\textwidth}
\begin{center}$\#\TT_4=265$\end{center}
\end{minipage}
\begin{minipage}{0.32\textwidth}
\begin{center}$\#\TT_5=718$\end{center}
\end{minipage}
\begin{minipage}{0.32\textwidth}
\begin{center}$\#\TT_{6}=1777$\end{center}
\end{minipage}

\caption{Hierarchical meshes generated by  Algorithm~\ref{the algorithm} with $\theta=0.5$ for the problem of Section~\ref{subsec:ring}   for $p_1=p_2=3$. 
}

\label{fig:bendshape_mesh}
\end{figure}

\begin{figure}
\psfrag{ON-2}[r][r]{\hspace{-10pt}\scalebox{0.5}{$\tiny \mathcal{O}(\#\mathcal{T_\ell}^{-2/2})$}}
\psfrag{ON-3}[r][r]{\hspace{-10pt}\scalebox{0.5}{$\tiny\mathcal{O}(\#\mathcal{T_\ell}^{-3/2})$}}
\psfrag{ON-4}[r][r]{\hspace{-10pt}\scalebox{0.5}{$\tiny \mathcal{O}(\#\mathcal{T_\ell}^{-4/2})$}}

\psfrag{Quarter of a ring, degree=2}{}
\psfrag{Quarter of a ring, degree=3}{}
\psfrag{Quarter of a ring, degree=4}{}

\psfrag{Number of elements}{\hspace{-24pt}\fontsize{5pt}{6pt}\selectfont number of elements $\#\TT_\ell$}
\psfrag{Energy error/Error estimator}{\hspace{-20pt}\fontsize{5pt}{6pt}\selectfont energy error/error estimator}

\psfrag{Error uniform}{\fontsize{5pt}{6pt}\selectfont err. unif.}
\psfrag{Error adaptive}{\fontsize{5pt}{6pt}\selectfont err. adap.}
\psfrag{Estimator uniform}{\fontsize{5pt}{6pt}\selectfont est. unif.}
\psfrag{Estimator adaptive}{\fontsize{5pt}{6pt}\selectfont est. adap.}

\centering 
 
\includegraphics[width=0.32\textwidth]{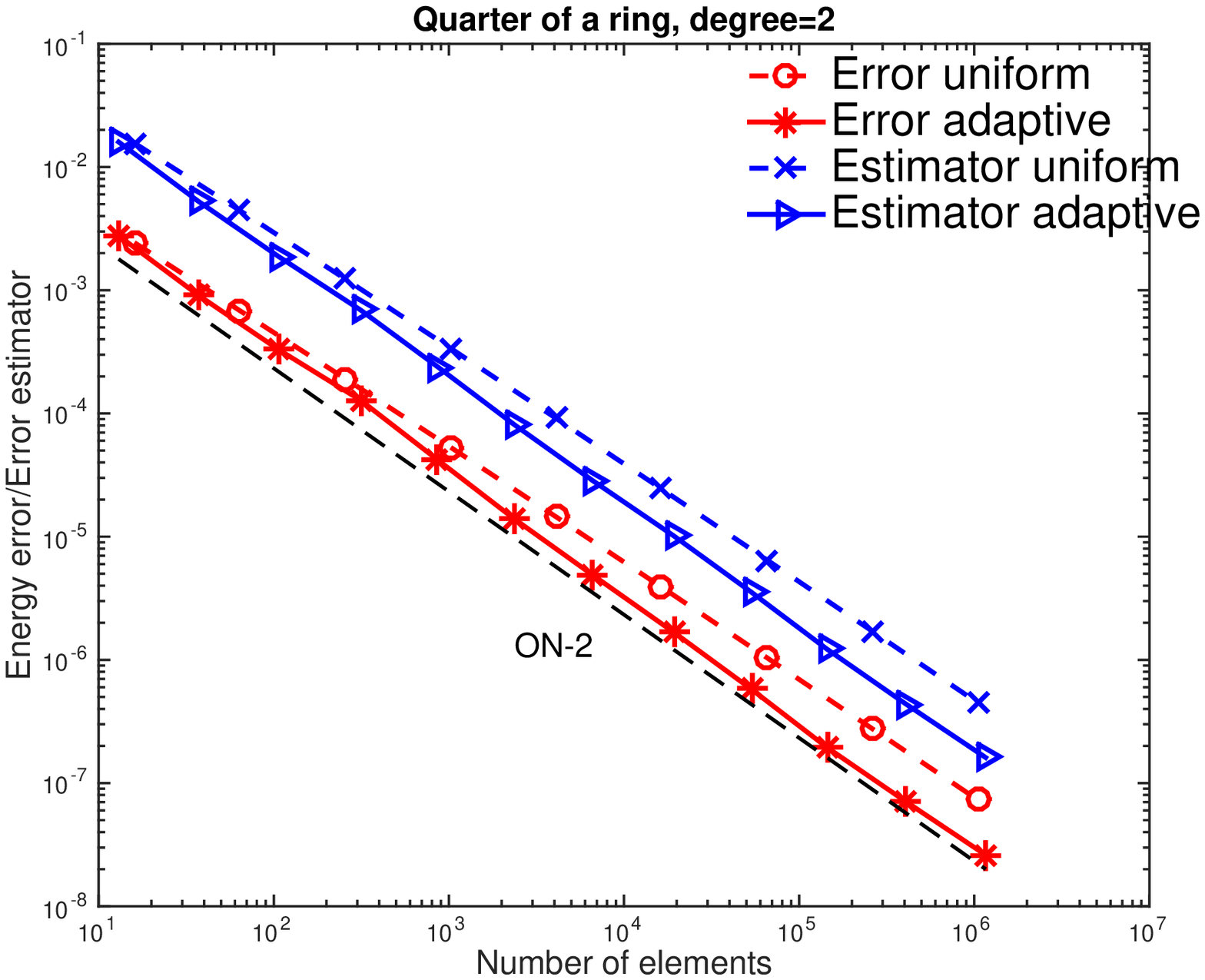}
\includegraphics[width=0.32\textwidth]{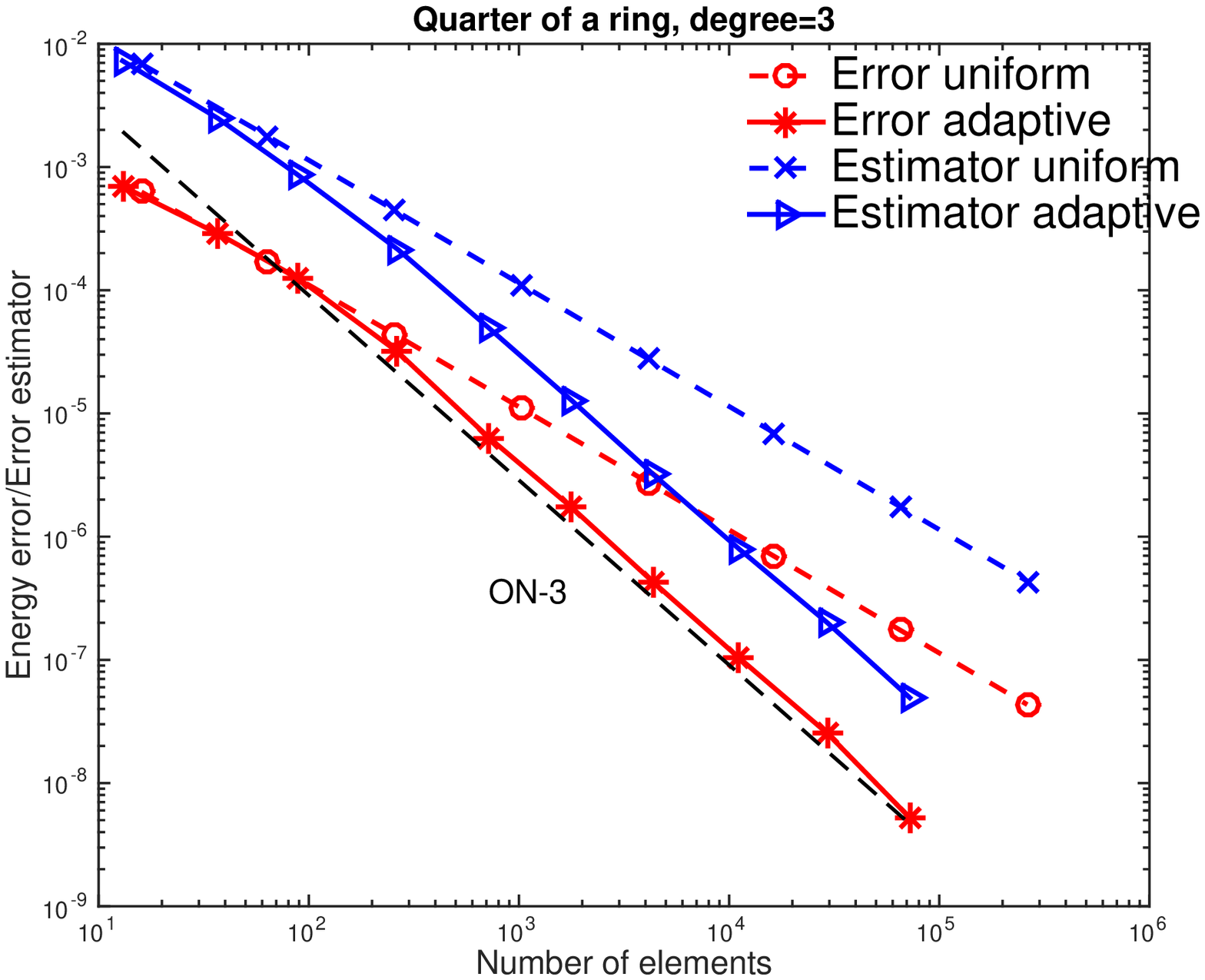}
\includegraphics[width=0.32\textwidth]{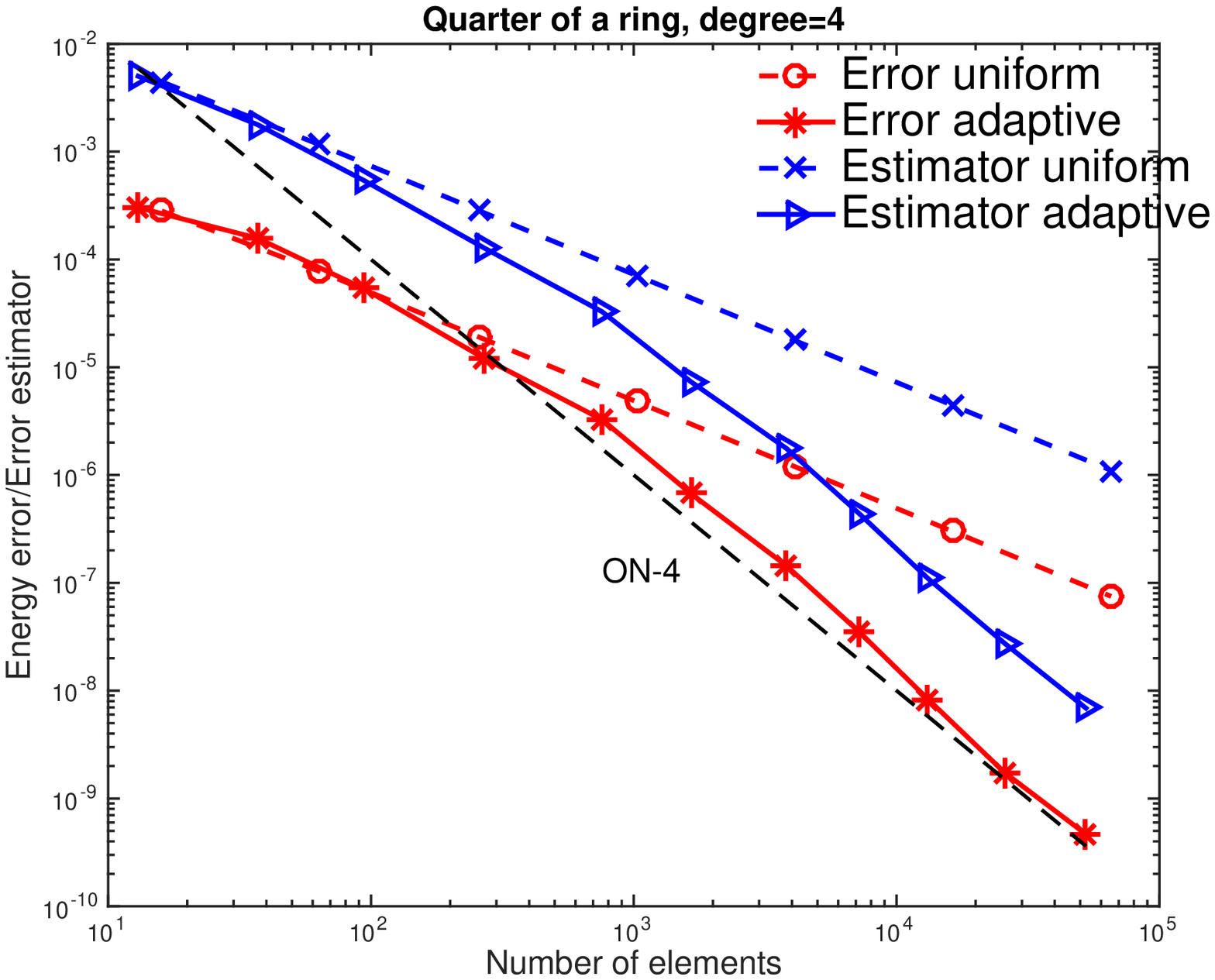}
\begin{minipage}{0.32\textwidth}
\begin{center}$p=2$\end{center}
\end{minipage}
\begin{minipage}{0.32\textwidth}
\begin{center}$p=3$\end{center}
\end{minipage}
\begin{minipage}{0.32\textwidth}
\begin{center}$p=4$\end{center}
\end{minipage}

\caption{
Energy error $\norm{\nabla u-\nabla U_\ell}{L^2(\Omega)}$ and  error estimator $\eta_\ell$ for the problem of Section~\ref{subsec:ring} on the quarter ring for uniform ($\theta=1$) and adaptive ($\theta=0.8$) mesh-refinement and different spline orders $p_1=p_2=p\in\{2,3,4\}$, where adaptivity always regains the respective optimal convergence rate.
}
\label{fig:bendshape} 
\end{figure}

\section*{Acknowledgement} The authors acknowledge support through the Austrian Science Fund (FWF) under grant P29096 
and grant W1245.

\bibliographystyle{alpha}
\bibliography{literature}

\end{document}